\documentclass[reqno,11pt]{article}
\usepackage{mathrsfs}
\usepackage{amssymb}
\usepackage{amsthm}
\usepackage{amsmath}
\usepackage{amsfonts}
\usepackage{mathtools}
\usepackage{enumerate}
\usepackage{bm}

\usepackage{graphicx}

\usepackage{bbm}
\usepackage{mathrsfs}
\usepackage{amssymb}
\usepackage[thicklines]{cancel}

\UseRawInputEncoding

\usepackage[usenames,dvipsnames]{xcolor}
\usepackage{soul}

\usepackage{txfonts}

\usepackage{psfrag}
\usepackage[numbers,sort&compress]{natbib}
\usepackage{setspace,color,wrapfig}
\usepackage[colorlinks,linkcolor=blue,citecolor=cyan]{hyperref}

\numberwithin{equation}{section}
\newtheorem{theorem}{Theorem}[section]
\newtheorem{lemma}[theorem]{Lemma}
\newtheorem{corollary}[theorem]{Corollary}
\newtheorem{proposition}[theorem]{Proposition}
\newtheorem{problem}[theorem]{Problem}
\theoremstyle{definition}
\newtheorem{definition}[theorem]{Definition}

\newtheorem{remark}[theorem]{Remark}
\theoremstyle{remark}

\begin{document}
\title{  Structural stability of cylindrical supersonic solutions to the steady Euler-Poisson system }
\author{Chunpeng Wang\thanks{School of Mathematics, Jilin University, Changchun, Jilin Province, China, 130012. Email: wangcp@jlu.edu.cn}\and
Zihao Zhang\thanks{School of Mathematics, Jilin University, Changchun, Jilin Province, China, 130012. Email: zhangzihao@jlu.edu.cn}}
\date{}

\newcommand{\de}{{\mathrm{d}}}
\def\div{{\rm div\,}}
\def\curl{{\rm curl\,}}
\def\th{\theta}
\newcommand{\ro}{{\rm rot}}
\newcommand{\sr}{{\rm supp}}
\newcommand{\sa}{{\rm sup}}
\newcommand{\va}{{\varphi}}
\newcommand{\me}{\mathcal{A}}
\newcommand{\ml}{\mathcal{V}}
\newcommand{\md}{\mathcal{D}}
\newcommand{\mg}{\mathcal{G}}
\newcommand{\mh}{\mathcal{H}}
\newcommand{\mf}{\mathcal{F}}
\newcommand{\ms}{\mathcal{S}}
\newcommand{\mt}{\mathcal{T}}
\newcommand{\mn}{\mathcal{N}}
\newcommand{\mb}{\mathcal{P}}
\newcommand{\mm}{\mathcal{B}}
\newcommand{\mj}{\mathcal{J}}
\newcommand{\mk}{\mathcal{K}}
\newcommand{\my}{\mathcal{U}}
\newcommand{\mw}{\mathcal{W}}
\newcommand{\mq}{\mathcal{Q}}
\newcommand{\ma}{\mathcal{L}}
\newcommand{\mc}{\mathcal{C}}
\newcommand{\mi}{\mathcal{I}}
\newcommand{\mz}{\mathcal{Z}}
\newcommand{\n}{\nabla}
\newcommand{\e}{\tilde}
\newcommand{\m}{\Omega}
\newcommand{\h}{\hat}
\newcommand{\w}{\Psi}
\newcommand{\x}{\bar}
\newcommand{\A}{\rho}
 \newcommand{\q}{{\rm R}}
\newcommand{\p}{{\partial}}
\newcommand{\z}{{\varepsilon}}
\renewcommand\figurename{\scriptsize Fig}
\pagestyle{myheadings} \markboth{Smooth supersonic Euler-Poisson  flows   }{Smooth supersonic Euler-Poisson  flows }\maketitle
\begin{abstract}
     This paper concerns the structural   stability of smooth cylindrically symmetric  supersonic Euler-Poisson flows in nozzles.  Both three-dimensional  and axisymmetric perturbations are considered.  On one hand, we establish the existence and uniqueness of  three-dimensional smooth supersonic  solutions to  the potential flow model of the steady  Euler-Poisson system.  On the other hand,  the existence and uniqueness of smooth   supersonic flows with nonzero vorticity   to the steady axisymmetric  Euler-Poisson system are proved. The problem is reduced to solve a nonlinear boundary value problem for a hyperbolic-elliptic mixed system. One of the key ingredients in the analysis of three-dimensional supersonic irrotational flows
is    the well-posedness theory for a linear second order     hyperbolic-elliptic coupled   system, which is achieved by using the multiplier method and the reflection
technique to derive the  energy estimates. For smooth axisymmetric supersonic  flows with nonzero  vorticity,  the deformation-curl-Poisson decomposition is utilized
 to reformulate  the steady axisymmetric Euler-Poisson system as a deformation-curl-Poisson   system together with several transport equations, so that one can design  a two-layer iteration  scheme to establish    the nonlinear structural stability  of the background supersonic flow within the class of axisymmetric rotational flows.
\end{abstract}
\begin{center}
\begin{minipage}{5.5in}
Mathematics Subject Classifications 2020: 35G60, 35J66, 35L72, 35M32, 76N10,76J20.\\
Key words:   supersonic Euler-Poisson  flows, structural stability, vorticity,  hyperbolic-elliptic coupled system,  deformation-curl-Poisson  decomposition.
\end{minipage}
\end{center}
\section{Introduction  }\noindent
\par
 In this paper, we  are concerned with    smooth
 supersonic Euler-Poisson flows in a cylindrical nozzle, which is   governed by the three-dimensional steady compressible Euler-Poisson system:
\begin{equation}\label{1-1-1}
\begin{cases}
\begin{aligned}
&\partial_{x_1}(\rho u_1)+\partial_{x_2}(\rho u_2)+\partial_{x_3}(\rho u_3)=0,\\
&\partial_{x_1}(\rho u_1^2)+\partial_{x_2}(\rho u_1 u_2)+\partial_{x_3}(\rho u_1 u_3)+\partial_{x_1} P =\rho \p_{x_1} \Phi,\\
&\partial_{x_1}(\rho u_1u_2)+\partial_{x_2}(\rho u_2^2)+\partial_{x_3}(\rho u_2 u_3)+\partial_{x_2} P=\rho \p_{x_2} \Phi,\\
&\partial_{x_1}(\rho u_1u_3)+\partial_{x_2}(\rho u_2 u_3)+\partial_{x_3}(\rho u_3^2)+\partial_{x_3} P=\rho \p_{x_3} \Phi,\\
&\div(\rho\mathcal{E}{\bf u}+P{\bf u})=\rho{\bf u}\cdot\nabla\Phi,\\
&{\Delta} \Phi= \rho-b({\bf x}).
\end{aligned}
\end{cases}
\end{equation}
 In the above system,   ${\bf x}=(x_1,x_2,x_3)$ is the  spatial variable,
 ${\bf u}=(u_1,u_2,u_3)$ is the macroscopic particle velocity, $\rho$ is the electron density, $P$ is the pressure, $\mathcal{E}$ is the total energy, respectively. The eletrostatic potential $ \Phi $  is generated by the Coulomb force of particles.  $b$  is a positive function representing the density of fixed,
positively charged background ions.  The system \eqref{1-1-1} is closed with the aid of definition of specific total energy and the equation of state
\begin{equation*}
\mathcal{E}=\frac{|{\bf u}|^2}{2}+\mathfrak{e} \quad \text{and}\quad P=P(\rho, \mathfrak{e}),
\end{equation*}
where $\mathfrak{e}$ is the internal energy.
 We consider the ideal polytropic gas for
which the total energy and the pressure are given by
\begin{equation*}
  \mathcal{E}=\frac1 2|{\bf u}|^2+\frac{ P}{(\gamma-1)\rho}
\quad{\rm {and}}\quad P(\rho,S)= e^{S}\rho^{\gamma},
\end{equation*}
where $S$ is  the entropy and $ \gamma> 1 $ is the adiabatic exponent. Denote   the Mach
number $M$ by
$$M=\frac{|{\bf u}|}{\sqrt{P_\rho(\rho, S)}}, $$
where
$ \sqrt{P_\rho(\rho, S)}$ is called the local sound speed.  If $M<1$,  the flow is said to be subsonic. If $M>1$, the flow is said to be supersonic.
\par The well-posedness of steady compressible  flows in nozzles is a nature issue in engineering application, such as designs of turbines and wind tunnels. This determines that it must be a fundamental
subject in the mathematical theory of fluid dynamics. The Euler-Poisson system can be used to model  the propagation of electrons in sub-micron semiconductor devices and plasmas \cite{MRS19},  and the biological transport
of ions for channel proteins \cite{Shu}.    If there is no Poisson portion,   the system \eqref{1-1-1} becomes the steady Euler  system which describes gas motion of compressible inviscid flows.  There have been many  interesting results on   smooth Euler flows through various nozzles, such as subsonic, subsonic-sonic, sonic-supersonic  and transonic flows, see \cite{CX14,DWX14,DXY11,DD11,nw16,WX2013,WX15,wang15,WX2016,WX2019,WX2020,wang22,WS15, WYZ21,WX23-1, WX23,WX23-2,XX10,XX2010} and references therein.   Compared with the steady Euler system,  the mathematical research for the steady Euler-Poisson system has essential difficulties. The first is that the system \eqref{1-1-1} changes type when the flow speed varies from subsonic to supersonic. For a subsonic state, \eqref{1-1-1} can be decomposed into
a nonlinear elliptic system  and homogeneous transport equations. For  a supersonic state, \eqref{1-1-1} can be decomposed into a  nonlinear hyperbolic-elliptic coupled system  and homogeneous transport equations. The second is that the Poisson equation for the eletrostatic potential is  coupled with the other equations in a nonlinear way. Therefore,  the study of multi-dimensional Euler-Poisson system  is much more
complicated and thus more interesting and challenging.
\par There have been a few significant
progress on  the   structural stability of physically nontrivial steady flows for the multi-dimensional Euler-Poisson system in a nozzle. For subsonic  flows, when the background solutions have
low Mach numbers and a small electric field,  the  existence and uniqueness
of three-dimensional subsonic flows with nonzero vorticity has been obtained in \cite{W14}.  In  \cite{BDX16}, by discovering  a special
structure of the elliptic system to yield a coercive estimate, the existence and stability of multi-dimensional subsonic potential flows  were proved. Later on,   they extended this result to the case of two-dimensional subsonic flow with nonzero
vorticity in \cite{BDX14} through the Helmholtz decomposition. They also investigated the  existence and stability of two-dimensional subsonic flows
with self-gravitiation in  \cite{BDX15}.  In \cite{BW18},   a  Helmholtz decomposition of the velocity field for three-dimensional axisymmetric flows was introduced to establish the
structural stability  of  axisymmetric
 subsonic flows with nonzero vorticity in a cylinder.     Different from this decomposition, the author in \cite{WS19} utilized the  deformation-curl-Poisson  decomposition to    prove  the structural stability of one-dimensional subsonic flows under
three-dimensional perturbations of the boundary conditions. Recently, the existence and uniqueness of two-dimensional subsonic flows
with self-gravitiation in an annulus were established in  \cite{DLW25}.
For supersonic  flows, the authors in \cite{BDXJ21} proved  the structural stability  of two-dimensional supersonic irrotational  flows  and rotational  flows. The  existence and uniqueness of  supersonic  potential flows  in a divergent nozzle was obtained in \cite{DWY23}. The authors in \cite{BP21,BP23} investigated the existence and stability of three-dimensional
supersonic  potential flows  in a rectangular nozzle.  For transonic  flows, the authors in \cite{Bae1}   established the existence and uniqueness of two-dimensional smooth transonic sloutions with nonzero vorticity.
 \par The first part of this paper is  to investigate  the  structural stability of smooth cylindrically  symmetric supersonic Euler-Poisson  flows within the class of irrotational flows.  The existence and uniqueness of  three-dimensional smooth supersonic irrotational flows rely on the well-posedness theory for a boundary value problem to   a linear system consisting of a second order hyperbolic equation and a second order elliptic
equation weakly coupled together. We find an appropriate multiplier to yield  a priori $H^1 $
energy estimate. Based on some  properties of the background supersonic flow, the multiplier is obtained by solving a differential equation.  With the $ H^1 $ energy estimate, one can view the  hyperbolic-elliptic coupled   system  as  a  hyperbolic equation and a elliptic
equation by
leaving only the principal term on the left-hand side of the equations and  use the method of reflection to derive the high order derivatives estimates. Then Galerkin's method with Fourier series can be used for the construction of the approximated solutions and a simple contraction mapping argument  yields the solution to the nonlinear problem.
 \par The second part of this paper is  to  establish  the  structural stability of smooth cylindrically  symmetric supersonic Euler-Poisson  flows within the class of axisymmetric rotational flows. To  treat smooth axisymmetric supersonic flows with
nonzero vorticity,  there are several major difficulties that must be
overcome. The first is to deal with the hyperbolic-elliptic coupled  structure in the supersonic region.     The steady axisymmetric  Euler-Poisson system is a
hyperbolic-elliptic mixed system, whose effective decomposition into elliptic and hyperbolic modes is crucial for the solvability of the nonlinear
boundary problem.  Here we  utilize  the deformation-curl-Poisson  decomposition introduced  in \cite{WX19,WS19}
to   reformulate the steady axisymmetric Euler-Poisson system as a deformation-curl-Poisson system  for the velocity field and the eletrostatic potential  together with three transport equations for the angular velocity, the Bernoulli's quantity and the entropy.  The second  is to choose some appropriate function
spaces and design an elaborate two-layer iteration scheme
to find the fixed point of the nonlinear problem. In the deformation-curl-Poisson system, the vorticity is resolved by an algebraic
equation for the angular velocity, the Bernoulli's quantity and the entropy. There is a loss of one derivative in the equation for the vorticity when dealing with supersonic flows. Using the continuity equation, we introduce a stream function to represent the angular velocity, the Bernoulli's function and the entropy as  functions of the stream function.
  By utilizing the advantage of one order higher regularity of the stream function than the radial velocity and the vertical velocity in the  region, we   gain one more  order derivatives
estimates for the angular velocity, the Bernoulli's function and the entropy than the radial velocity and the vertical velocity, which  is crucial for
us to close the energy estimates.
\par This paper will be arranged as follows. In Section 2, we formulate the problem in detail and state  main results.  In Section 3, we  establish the structural stability of the background supersonic flow within
 the class of  three-dimensional irrotational  flows.   In Section 4, we    establish the structural stability of the background supersonic flow within the class of axisymmetric rotational flows.

  \section{The  nonlinear problem and main results}\noindent
\par In this section, we give a detailed formulation of smooth supersonic Euler-Poisson flows in a cylindrical nozzle and
state  main results. To this end, introduce the cylindrical coordinates
$(r, \theta, z)$:
 \begin{eqnarray}\label{coor}
r=\sqrt{x_1^2+x_2^2},\quad \theta=\arctan \frac{x_2}{x_1},\quad z=x_3.
\end{eqnarray}
Assume that the velocity, the density, the pressure and the eletrostatic
potential are of the form
\begin{equation*}
\begin{aligned}
&{\bf u}({\textbf x})= U_{1}(r,\th,z)\mathbf{e}_r+U_{2}(r,\th,z)\mathbf{e}_\theta+U_{3}(r,\th,z)\mathbf{e}_z, \\
&\rho({\textbf x})=\rho(r,\th,z),\ \ P({\textbf x})=P(r,\th,z),\ \ \Phi({\textbf x})=\Phi(r,\th,z),\ \
\end{aligned}
\end{equation*}
 with
\begin{equation*}
\mathbf{e}_r =
\begin{pmatrix}
\cos \th  \\
  \sin \th \\
  0
\end{pmatrix}, \quad
\mathbf{e}_{\th } =
\begin{pmatrix}
  - \sin \th  \\
  \cos \th\\
  0
  \end{pmatrix}, \quad
  \mathbf{e}_z =
\begin{pmatrix}
 0 \\
  0 \\
  1
\end{pmatrix},
\end{equation*}
where $ U_1 $,  $ U_2 $ and  $ U_3  $ represent the radial velocity, the angular velocity  and the vertical velocity, respectively.
Then
the steady compressible  Euler-Poisson system \eqref{1-1-1} in the cylindrical coordinates takes the
form
\begin{eqnarray}\label{1-2}
\begin{cases}
\begin{aligned}
&\partial_r(r\rho U_1)+\partial_{\theta}(\rho U_2)+\p_z(r\rho U_3)=0,\\
&\rho\left(U_1\partial_r +\frac{U_2}{r}\partial_{\theta}+U_3\p_z\right)U_1+\partial_r P-\frac{\rho U_2^2}r=\rho\p_r\Phi,\\
&\rho\left(U_1\partial_r +\frac{U_2}{r}\partial_{\theta}+U_3\p_z\right)U_2
+\frac{\partial_{\theta}P}{r}+\frac{\rho U_1U_2}{r}=\frac{\rho\p_{\th}\Phi} r,\\
&\rho\left(U_1\partial_r +\frac{U_2}{r}\partial_{\theta}+U_3\p_z\right)U_3+\partial_{z}P=
\rho\partial_{z}\Phi,\\
&\rho\left(U_1\partial_r +\frac{U_2}{r}\partial_{\theta}+U_3\p_z\right)K=0,\\
&\left(\p_r^2+\frac 1 r\p_r+\frac{1}{r^2}\p_{\th}^2+\p_z^2\right)\Phi=\rho-b,
\end{aligned}
\end{cases}
\end{eqnarray}
where  the  Bernoulli's function $K$ is defined  as
\begin{equation*}
K:=\frac{1}{2}|{\bf u}|^2+\frac{\gamma P}{(\gamma-1)\rho}-\Phi=\frac{1}{2}|{\bf u}|^2+\frac{\gamma e^S\rho^{\gamma-1}}{\gamma-1}-\Phi.
\end{equation*}
In the cylindrical coordinates, the vorticity has the form $\text{curl }{\bf u}=\omega_1 {\bf e}_r + \omega_{2}{\bf e}_{\theta}+ \omega_3 {\bf e}_z$, where
\begin{eqnarray}\label{vor1}
\omega_1=\frac{1}{r}\partial_{\theta} U_3 -\partial_z U_2,\ \ \omega_{2}= \partial_z U_1- \partial_r U_3,\ \omega_3=\frac{1}{r}\partial_r(r U_2)-\frac{1}{r}\partial_{\theta} U_1.
\end{eqnarray}
The cylindrical nozzle $\Omega$  concerned in this paper can be described as
\begin{equation*}
\Omega=\{(r,\th,z):r_0<r<r_1,\ (\th,z)\in D\}, \ \ D=(-\theta_0,\theta_0)\times(-1,1),
\end{equation*}
where  $0<r_0<r_1<\infty,\theta_0\in (0,\frac{\pi}{2})$ are fixed positive constants. The entrance, exit and  nozzle walls of    $\Omega$   are denoted by
\begin{equation*}
\begin{aligned}
&\Gamma_{en}=\{(r,\th,z):r=r_0,\  -\th_0<\th<\th_0,-1<z<1\},\\
		&\Gamma_{ex}=\{(r,\th,z):r=r_1, \  -\th_0<\th<\th_0,-1<z<1\},\\
&\Gamma_{\th_0}^\pm=\{(r,\th,z):r_0<r<r_1,\  \th=\pm \th_0,-1<z<1\},\\
&\Gamma_{1}^\pm=\{(r,\th,z):r_0<r<r_1,\  -\th_0<\th<\th_0,z=\pm 1\}.\\
		\end{aligned}
\end{equation*}
\par  In this paper,  we focus on the structural stability of a  class of smooth cylindrically
symmetric supersonic flows. More precisely, fix $b$ to be a constant $b_0>0$, the background flow is described by smooth functions of the form
\begin{equation*}
{\bf u}({\textbf x})=\bar U(r)\mathbf e_r,
 \ \ \rho({\textbf x})=\bar\rho(r),\ \  S({\textbf x})=\bar S(r),\ \  \Phi({\textbf x})=\bar \Phi(r), \ \  r\in[r_0, r_1],
\end{equation*}
which solves the following system:
\begin{equation}\label{1-4}
	\begin{cases}
\begin{aligned}
		&(r \bar \rho  \bar U)'(r)=0,\ \ & r\in[r_0, r_1],\\
     	 &\bigg(\bar U \bar U'+\frac{1}{\bar\rho}\bar P'\bigg)(r)= \bar E(r),\ \ & r\in[r_0, r_1],\\
    	 &(\bar U \bar S')(r)=0, \ \ &  r\in[r_0, r_1],\\
    	 &(r\bar E)'(r)=r(\bar \rho(r)-b_0),\ \ &  r\in[r_0, r_1],\\
    \end{aligned}
	\end{cases}
\end{equation}
with the boundary conditions at the entrance:
\begin{equation}\label{1-4-b}
\bar\rho(r_0)=\rho_0,\quad
\bar U(r_0)=U_{0},\quad
 \bar S(r_0)=S_0, \quad
\bar E(r_0)=E_0.
\end{equation}
Here
$$ \bar P(r)=(e^{\bar S}\bar \rho^\gamma)(r) \ \ {\rm{and}} \ \ \bar E(r)=\bar\Phi'(r), \ \  r\in[r_0, r_1].$$
   \par  Define $J_0=r_0\rho_0U_{0}$. 
Then
\begin{equation*}
(\bar\rho\bar U)(r)=\frac{J_0}{r} \ \ {\rm{and}} \ \
\bar S(r) = S_0, \ \  r\in[r_0, r_1].
\end{equation*}
Therefore,  \eqref{1-4} can be reduced to the following ODE system for $(\bar U, \bar E)$:
\begin{equation}\label{1-5}
\begin{cases}
\begin{aligned}
&\bar U^{\prime}(r)=\frac{r(\bar U \bar E)(r)+ (\bar c^2\bar U) (r)}
{r( \bar U^2(r)- \bar c^2(r))}, \ \ & r\in[r_0, r_1],\\
&\bar E^{\prime}(r)=\bar \rho(r)-b_0-\frac{1}{r}\bar E(r),  \ \ &  r\in[r_0, r_1],\\
\end{aligned}
\end{cases}
\end{equation}
where $ \bar c^2(r)=\gamma  e^{ S_0} \bar\rho^{\gamma-1}(r) $.
\par We have the following well-posedness, which has been established in \cite{DWY23}.
\begin{proposition}\label{pro1}
Fix $b_0> 0 $. Given   constants $\rho_0>0$, $ U_{0}>0$,    $S_0>0$ and  $E_0$ satisfying
\begin{equation*}
0<J_0<J^\sharp \ \ {\rm{and}} \  \  U_\sharp<U_0<U^\sharp,
\end{equation*}
where
\begin{equation*}
J^\sharp=\left(\frac{r_0^{2(2\gamma-1)}}{2^{\gamma+1}(\gamma e^{S_0})^3}\right)^{\frac{1}
{2(\gamma-2)}},  \ \
U_\sharp=\left(\frac{\gamma e^{S_0}J_0^{\gamma-1}}{r_0^{\gamma-1}}\right)^{\frac{1}
{\gamma+1}}, \quad U^\sharp=(2r_0J_0)^{\frac13}.
\end{equation*}
Then there exists a positive constant $R_0$ depending only the data $(\gamma,r_0,b_0,\A_0,U_0,S_0, E_0) $ such that for any $ r_1\in(r_0,r_0+R_0)$,
  the initial value problem \eqref{1-4} and \eqref{1-4-b} has a unique smooth solution $(\bar\rho,\bar  U,  \bar S, \bar E)$ satisfying
\begin{eqnarray}\label{1-9}
U_{\sharp}<\bar U(r) < U^{\sharp}\ \  {\rm{and}}\ \
\bar U^2(r)>\bar c^2(r),\ \  r\in[r_0, r_1].
\end{eqnarray}

\end{proposition}
Define
\begin{eqnarray}\label{1-q}
 && \bar P(r)=e^{S_0} \bar\rho^{\gamma}(r) \ \ {\rm{and}} \ \
 \bar\Phi(r)= \int_{r_0}^{r} \bar E(s) \de s, \ \  r\in[r_0, r_1].
\end{eqnarray}
Then   $(\bar\rho,\bar  U, \bar P, \bar \Phi)(r)$ satisfies  \eqref{1-2} in $ \Omega $. Furthermore, denote  the  Bernoulli's function by
\begin{eqnarray}\label{1-10}
\bar K(r)=\frac{\bar U^2(r)}{2}+\frac{\gamma e^{S_0}\bar\rho^{\gamma-1}(r)}{\gamma-1}-\bar\Phi(r),\ \  r\in[r_0, r_1], \quad K_0= \frac{U_0^2}{2}+\frac{\gamma e^{S_0}\rho_0^{\gamma-1}}{\gamma-1} .
\end{eqnarray}
It follows from the second equation in \eqref{1-4} that
 \begin{eqnarray}\label{1-10-1}
\bar K(r)=K_0, \ \  r\in[r_0, r_1].
 \end{eqnarray}
\begin{definition}\label{de1}
$(\bar\rho,\bar  U, \bar P, \bar \Phi)$     is called the  background solution  associated with the entrance data $(b_0, \A_0,$\\ $ U_{0},S_{0}, E_{0})$.
  \end{definition}
   The main purpose of this work is to investigate the structural stability of the
background solution in the  cylindrical nozzle $\m$   by the following regimes:
\begin{enumerate}[ \rm (i)]
\item  Firstly, we  consider the structural stability of the background
solution under three-dimensional perturbations of suitable boundary
conditions  for the potential flow model.
 \item  Secondly, we  consider the structural stability of the background
solution under axisymmetric perturbations of suitable boundary
conditions   for the steady
Euler-Poisson system.
\end{enumerate}
\subsection{Smooth supersonic irrotational flows}\noindent
\par  We start with smooth supersonic  flows with zero vorticity.  For  the isentropic irrotational flows, $ S $ is a constant and  $ \text{curl }{\bf u}=0 $.  Without loss of  generality, we assume $ S=S_0 $.
By the vector identity ${\bf u}\cdot\nabla {\bf u}= \nabla \frac12 |{\bf u}|^2- {\bf u}\times \text{curl }{\bf u}$, the momentum equations in \eqref{1-1-1} imply
\begin{eqnarray}\label{momentum}
\nabla \bigg(\frac{1}2|{\bf u}|^2 +\frac{\gamma e^{S_0}\rho^{\gamma-1}}{\gamma-1}-\Phi\bigg)={\bf u}\times \text{curl }{\bf u}=0, \ \  \text{in}\  \ \Omega,
\end{eqnarray}
 form which one obtains  that $ \frac{1}2|{\bf u}|^2 +\frac{\gamma e^{S_0}\rho^{\gamma-1}}{\gamma-1}-\Phi$ is a constant. For simplicity, we assume that
\begin{eqnarray}\label{momentum-1}
 \frac{1}2|{\bf u}|^2 +\frac{\gamma e^{S_0}\rho^{\gamma-1}}{\gamma-1}-\Phi=K_0,  \ \  \text{in}\  \ \Omega.
\end{eqnarray}
In terms of the  cylindrical coordinates, one has
\begin{eqnarray*}
\frac{1}{r}\partial_{\theta} U_3 -\partial_z U_2=0, \ \  \partial_z U_1- \partial_r U_3=0,\ \frac{1}{r}\partial_r(r U_2)-\frac{1}{r}\partial_{\theta} U_1=0, \  \text{in}\  \ \Omega.
\end{eqnarray*}
Then there exists a potential function $ \phi(r,\th,z) $ such that
\begin{eqnarray*}
\p_r\phi=U_1, \ \ \p_\th\phi=rU_2, \ \ \p_z\phi=U_3, \ \  \text{in}\  \ \Omega.
\end{eqnarray*}
Thus it follows from \eqref{momentum-1} that the density can be rewritten as
\begin{eqnarray*}
\rho=
\left(\frac{\gamma-1}{\gamma e^{S_0}}\right)
^{\frac{1}{\gamma-1}}\bigg(K_0+\Phi-\frac{1}{2}\bigg((\p_r\phi)^2+
\frac{(\p_\th\phi)^2}
{r^2}+(\p_z\phi)^2\bigg)
\bigg)^{\frac{1}{\gamma-1}}, \  \text{in}\  \ \Omega.
\end{eqnarray*}
The system  \eqref{1-2}  can be reduced to the  the following system in $  \Omega $, which is called the potential flow model of the steady Euler-Poisson system:
\begin{equation}\label{2-9}
\begin{cases}
\begin{aligned}
&\partial_r(r\rho \p_r\phi)+\partial_\theta(\rho \p_{\th}\phi)+\p_z(r\rho \p_z\phi)=0,\\
&\left(\p_r^2+\frac 1 r\p_r+\frac{1}{r^2}\p_{\th}^2+\p_z^2\right)\Phi=\rho-b.\\
\end{aligned}
\end{cases}
\end{equation}

\par  For the potential flow, we investigate  the following problem.
\begin{problem}\label{probl1}
Given functions $(b^\ast, U_{1, en}^\ast, E_{ en}^\ast, \Phi_{ ex}^\ast)$, sufficiently close to $(b_0,U_{0},E_0,\bar\Phi(r_1))$,
find  a solution
$( \phi, \Phi)$ to the system \eqref{2-9} in $\m$ satisfying  the following properties.
\begin{enumerate}[ \rm (1)]
 \item
   (Positivity of the density and the radial velocity)
 \begin{equation}\label{1-c-c-c}
 \rho>0  \ \ {\rm{ and}} \ \ \p_r\phi>0,   \ \ {\rm{ in}} \ \ \overline{\m}.
 \end{equation}
\item (The boundary conditions)
\begin{equation}\label{1-c}
\begin{cases}
\phi(r_0,\th,z)=0,\ (\p_r\phi, \p_r\Phi)(r_0,\th,z)=(U_{1,en}^\ast, E_{en}^\ast)(\th,z),\ \ &{\rm{on}}\ \ \Gamma_{en},\\
\Phi(r_1,\th,z)=\Phi_{ex}^\ast(\th,z), \ \ &{\rm{on}}\ \ \Gamma_{ex},\\
\p_\th\phi(r,\pm\th_0,z)=\p_\th\Phi(r,\pm\th_0,z)=0, \ \  &{\rm{on}}\ \ \Gamma_{\th_0}^\pm,\\
\p_z\phi(r,\th,\pm 1)=\p_z\Phi(r,\th,\pm 1)=0, \ \  &{\rm{on}}\ \ \Gamma_{1}^\pm.\\
\end{cases}
\end{equation}
\item   (Supersonic speed)
\begin{equation}\label{1-c-z}
 (\p_r\phi)^2+
\frac{(\p_\th\phi)^2}
{r^2}+(\p_z\phi)^2>\gamma e^{S_0}{\rho}^{\gamma-1}, \ \ {\rm{in}}\ \
\overline{\m}.
\end{equation}
 \end{enumerate}
\end{problem}
\par The first main result in this paper states the structural stability of cylindrically  symmetric
supersonic  flows under three-dimensional perturbations of suitable boundary conditions, which also yields the existence and uniqueness of smooth supersonic irrotational flows to the  system \eqref{2-9}.
\begin{theorem}\label{th1}
  For given functions $b^\ast\in C^2(\overline{\m})$, $U_{ 1,en}^\ast\in C^3(\overline D)$ with  $\displaystyle{\min_{(\th,z)\in \overline D  }U_{1, en}^\ast>0 }$
  and $(E_{ en}^\ast,\Phi_{ ex}^\ast)\in \left(C^4(\overline D)\right)^2$, set
\begin{equation}\label{1-9-a}
\begin{aligned}
 \sigma(b^\ast,U_{ 1,en}^\ast,E_{ en}^\ast,\Phi_{ ex}^\ast)&:=\|b-b_0\|_{C^2(\overline{\m})}+\|U_{1, en}^\ast-U_{0}\|_{C^3(\overline D)}\\
 &\quad +\|(E_{ en}^\ast,\Phi_{ ex}^\ast)-(E_0,\bar\Phi(r_1))\|_{C^4(\overline D)}.\\
  \end{aligned}
  \end{equation}
  Let  $\bar\phi(r)=\int_{r_0}^r \bar U(s)\de s$ be the potential function of background supersonic solution defined in Definition \ref{de1}. For each $\epsilon_0\in(0,R_0)$ with $R_0>0$ given in Proposition \ref{pro1}, there exist a constant $\bar{r}_1^\ast\in(r_0,r_0+R_0-\epsilon_0]$  so  that for  $r_1\in(r_0, \bar{r}_1^\ast)$,
 if
$(b^\ast,U_{ 1,en}^\ast,E_{ en}^\ast,\Phi_{ ex}^\ast)$ satisfies
\begin{equation}\label{1-t-1}
\sigma(b^\ast,U_{ 1,en}^\ast,E_{ en}^\ast,\Phi_{ ex}^\ast)\leq  \sigma^\ast
\end{equation}
 with $\sigma^\ast>0$ depending only on  $(b_0, \A_0,U_{0}, S_{0}, E_{0},\th_0,r_0,r_1,\epsilon_0)$, and
 the compatibility conditions
\begin{equation}  \label{1-t-2}
\begin{cases}
\begin{aligned}
&\p_{\th} b^\ast=0,  \ {\rm{on}} \ \Gamma_{\th_0}^\pm, \ \p_\th U_{1,en}^\ast( \pm\th_0, z)
=\p_\th^k E_{en}^\ast( \pm\th_0, z)=\p_\th^k\Phi_{ex}^\ast( \pm\th_0, z)
=0, \   z\in[-1,1],\ k=1,3,\\
&\p_{z} b^\ast=0,  \ {\rm{on}} \ \Gamma_{1}^\pm, \ \p_z U_{1,en}^\ast( \th,\pm 1)
=\p_z^{k}E_{en}^\ast( \th,\pm 1)=\p_z^{k}\Phi_{ex}^\ast( \th,\pm 1)
=0, \quad   \th\in[-\th_0,\th_0],\ k=1,3,
\end{aligned}
\end{cases}
 \end{equation}
 then Problem \ref{probl1} has a   unique smooth   irrotational solution $(\phi, \Phi)\in \left( H^4(\m)\right)^2$, which satisfies
 \begin{eqnarray}\label{1-t-3}
\| (\phi,\Phi)- (\bar\phi,\bar\Phi)\|_{H^4(\m)}\leq   \mc_1^\ast \sigma(b^\ast,U_{ 1,en}^\ast,E_{ en}^\ast,\Phi_{ ex}^\ast),
  \end{eqnarray}
   and
   \begin{equation}\label{1-t-3-ss}
 c^2(\Phi,\n \phi)-|\n \phi|^2
 \leq - \kappa^\ast,  \ \ {\rm{in}} \ \ \overline\m.
\end{equation}
Here  $ c(\Phi,\n \phi) $ is the sound speed given by
 \begin{equation*}
c(\Phi,\n \phi)=\sqrt{(\gamma-1)\bigg(K_0+\Phi-\frac12|\n \phi|^2\bigg)}  \ \ {\rm{with}} \ \ \n=\left(\p_r,\frac{\p_\th}r,\p_z\right),
\end{equation*}
    and  the positive constants $ \mc_1^\ast$ and $\kappa^\ast$ depend only on  $(b_0, \A_0,U_{0}, S_{0}, E_{0},\th_0,r_0,r_1,\epsilon_0)$. Furthermore, the solution $(\phi, \Phi  )$ satisfies the compatibility conditions
\begin{equation}
\label{comp-cond-nlbvp-full}
\begin{cases}
\p_{\th}^k\phi=\p_{\th}^k\Phi=0,   \ \ {\rm{on}} \ \Gamma_{\th_0}^\pm, \ \  {\rm{for}} \ \ k=1,3,\\
\p_{z}^k\phi=\p_{z}^k\Phi=0,   \ \ {\rm{on}} \ \Gamma_{1}^\pm, \ \ \ {\rm{for}} \ \ k=1,3,
\end{cases}
\end{equation}
in the sense of trace.
For each $\alpha\in(0,\frac12)$, it follows from \eqref{1-t-3} that
\begin{eqnarray}\label{1-t-7}
  \| (\phi,\Phi)- (\bar\phi,\bar\Phi)\|_{C^{2,\alpha}(\overline{\m})}
 \leq   \mc_2^{\ast}\sigma(b^\ast,U_{ 1,en}^\ast,E_{ en}^\ast,\Phi_{ ex}^\ast)
  \end{eqnarray}
for some constant $ \mc_2^{\ast}>0$ depending only on $(b_0, \A_0,U_{0}, S_{0}, E_{0},  \th_0,r_0, r_1,\epsilon_0,\alpha)$.

 \end{theorem}

 \subsection{Smooth axisymmetric supersonic rotational flows}\noindent
 \par Next, we turn to  smooth axisymmetric supersonic   flows with nonzero vorticity.  Using the   cylindrical coordinates \eqref{coor}, any function $h({\bf x})$ can be represented as $h({\bf x})=h(r,\theta,z)$, and a vector-valued function ${\bf H}({\bf x})$ can be represented as ${\bf H}({\bf x})=H_1(r,\theta,z){\bf e}_r+ H_2(r,\theta,z){\bf e}_\th+ H_{3}(r,\theta,z){\bf e}_{z}$. We say that a function $h({\bf x})$ is  axisymmetric if its value is independent of $\theta$ and that a vector-valued function ${\bf H }= (H_1, H_2, H_{3})$ is axisymmetric if each of functions $H_1({\bf x}), H_{2}({\bf x})$ and $H_{3}({\bf x})$ is axisymmetric.
\par Assume that the velocity, the density, the pressure and the eletrostatic
potential are of the form
\begin{equation*}
\begin{aligned}
&{\bf u}({\textbf x})= U_{1}(r,z)\mathbf{e}_r+U_{2}(r,z)\mathbf{e}_\theta
+U_3(r,z)\mathbf{e}_z,\\
&\rho({\textbf x})=\rho(r,z),\ \ P({\textbf x})=P(r,z),\ \ \Phi({\textbf x})=\Phi(r,z).\\
\end{aligned}
\end{equation*}
Then the system \eqref{1-2} reduces to
\begin{eqnarray}\label{2-10}
\begin{cases}
\begin{aligned}
&\partial_r(r\rho U_1)+\p_{z}(r\rho U_3)=0,\\
&\rho\left(U_1\partial_r +U_3\p_{z}\right)U_1+\partial_r P-\frac{\rho U_2^2}{r}=\rho\p_r\Phi,\\
&\rho\left(U_1\partial_r +U_3\p_{z}\right)U_2+\frac{\rho U_1U_2}{r}=0,\\
&\rho\left(U_1\partial_r +U_3\p_{z}\right)U_3+\p_{z}P=\rho\p_{z}\Phi,\\
&\rho\left(U_1\partial_r+U_3\p_{z}\right)K=0,\\
&\bigg(\p_r^2+\frac 1 r\p_r+\p_{z}^2\bigg)\Phi=\rho-b.
\end{aligned}
\end{cases}
\end{eqnarray}
The flow region $\Omega $ is simplified as
 \begin{equation*}
\mn=\{(r,z): r_0<r<r_1, -1<z<1\}.
\end{equation*}
The entrance, exit and  nozzle walls of  $ \mn$ are denoted by
\begin{equation*}
\begin{aligned}
&\Sigma_{en}=\{(r,z):r=r_0,\  -1<z<1\},\\
		&\Sigma_{ex}=\{(r,z):r=r_1,\  -1<z<1\},\\
&\Sigma_{1}^\pm=
		\{(r,z):r_0<r<r_1,\ z=\pm 1\}.
	\end{aligned}
\end{equation*}
\par  For  axisymmetric   flows, we investigate  the following problem.
\begin{problem}\label{probl2}
Given functions $(b^\star, U_{1, en}^\star, U_{2, en}^\star, U_{3, en}^\star,K_{en}^\star,S_{en}^\star ,E_{ en}^\star, \Phi_{ ex}^\star)$  sufficiently close to $(b_0,U_{0},$\\ $0,0,K_0,S_0,E_0,\bar\Phi(r_1))$,
find  a solution
$(\rho, U_1,U_2, U_3,P, \Phi) $  to the   system \eqref{2-10} in $\mn$ satisfying  the following properties.
\begin{enumerate}[ \rm (1)]
 \item
   (Positivity of the density and the radial velocity)
 \begin{equation}\label{1-c-c-a}
 \rho>0  \ \ {\rm{ and}} \ \ U_1>0,   \ \ {\rm{ in}} \ \ \overline{\mn}.
 \end{equation}
\item (The boundary conditions)
\begin{equation}\label{1-c-c}
\begin{cases}
(U_1, U_2,   U_3,K,S,\p_r\Phi)(r_0,z)=(U_{1,en}^\star, U_{2,en}^\star,  U_{3,en}^\star,K_{en}^\star,S_{en}^\star,E_{en}^\star)(z),\ \ &{\rm{on}}\ \ \Sigma_{en},\\
\Phi(r_1,z)=\Phi_{ex}^\star(z), \ \ &{\rm{on}}\ \ \Sigma_{ex},\\
U_3(r,\pm 1)=\p_z\Phi(r,\pm 1)=0, \ \  &{\rm{on}}\ \ \Sigma_{1}^\pm.\\
\end{cases}
\end{equation}
\item   (Supersonic speed)
\begin{equation}\label{1-c-z-a}
U_1^2+U_2^2+U_3^2>\gamma e^{S}{\rho}^{\gamma-1},\ \  {\rm{in}}\ \ \overline{\mn}.
\end{equation} \end{enumerate}
\end{problem}
\par The second main result in this paper states the structural stability of cylindrically  symmetric
supersonic  flows under axisymmetric perturbations of suitable boundary conditions, which also yields the existence and uniqueness of smooth axisymmetric supersonic  rotational flows to the  system \eqref{2-10}.
\begin{theorem}\label{th2}
  For given functions $b^\star\in C^2(\overline{\mn})$, $U_{ 1,en}^\star\in C^3([-1,1])$ with  $\displaystyle{\min_{ [-1,1] }U_{1, en}^\star>0 }$ and $(U_{2, en}^\star,$\\$U_{3, en}^\star,K_{en}^\star,S_{en}^\star,E_{ en}^\star,\Phi_{ ex}^\star)\in \left(C^4([-1,1])\right)^6$, set
\begin{equation}\label{1-9-a-a}
\begin{aligned}
 &\sigma(b^\star,U_{ 1,en}^\star,U_{2, en}^\star,U_{3, en}^\star,K_{en}^\star,S_{en}^\star,E_{ en}^\star,\Phi_{ ex}^\star)
 :=\|b-b_0\|_{C^2(\overline{\mn})}+\|U_{1, en}^\star-U_{0}\|_{C^3([-1,1])}\\
 &\quad+\|(U_{2, en}^\star,U_{3, en}^\star,K_{en}^\star,S_{en}^\star,E_{ en}^\star,\Phi_{ ex}^\star)-(0,0,K_0,S_0,E_0,\bar\Phi(r_1))\|_{C^4([-1,1])}.\\
 \end{aligned}
  \end{equation}
  Let  $(\bar\rho,\bar  U, \bar P, \bar \Phi)$ be the  background supersonic solution defined in Definition \ref{de1}. For each $\epsilon_0\in(0,R_0)$ with $R_0>0$ given in Proposition \ref{pro1}, there exist a constant $\bar{r}_1^\star\in(r_0,r_0+R_0-\epsilon_0]$  so  that for $r_1\in(r_0, \bar{r}_1^\star)$,
 if
$(b^\star,U_{ 1,en}^\star,U_{2, en}^\star,U_{3, en}^\star,K_{en}^\star,S_{en}^\star,E_{ en}^\star,\Phi_{ ex}^\star)$ satisfies
\begin{equation}\label{1-t-1-2}
\sigma(b^\star,U_{ 1,en}^\star,U_{2, en}^\star,U_{3, en}^\star,K_{en}^\star,S_{en}^\star,E_{ en}^\star,\Phi_{ ex}^\star)\leq  \sigma^\star
\end{equation}
 with $\sigma^\star>0$ depending only on  $(b_0, \A_0,U_{0}, S_{0}, E_{0},r_0,r_1,\epsilon_0)$, and
 the compatibility conditions
\begin{equation}  \label{1-t-2-2}
\begin{aligned}
&\p_{z} b^\star=0, \ {\rm{on}} \ \Sigma_{1}^\pm,\ \ \frac{\de U_{1, en}^\star}{\de z}(\pm 1)=\frac{\de^k U_{2, en}^\star}{\de z}(\pm 1)=\frac{\de^{k-1}U_{3, en}^\star}{\de z^{k-1}}(\pm 1)=\frac{\de^k K_{ en}}{\de z^k}(\pm 1)\\
&=\frac{\de^k S_{ en}^\star}{\de z^k}(\pm 1)=\frac{\de^k E_{ en}^\star}{\de z^k}(\pm 1)
=\frac{\de^k \Phi_{ ex}^\star}{\de z^k}(\pm 1)=0, \ \ {\rm{for}} \ \ k=1,3,
\end{aligned}
\end{equation}
 then Problem \ref{probl2} has a   unique smooth   solution with nonzero vorticity  $(U_1,U_3,U_2,K,S,\Phi)\in \left(H^3(\mn)\right)^2$\\$\times \left(H^4(\mn)\right)^4$, which satisfies
 \begin{eqnarray}\label{1-t-3-2}
\|(U_1,U_3)-(\bar U,0)\|_{H^3(\mn)}+\| \Phi-\bar\Phi\|_{H^4(\mn)}
\leq  \mc_1^\star \sigma(b^\star,U_{ 1,en}^\star,U_{2, en}^\star,U_{3, en}^\star,K_{en}^\star,S_{en}^\star,E_{ en}^\star,\Phi_{ ex}^\star),
\end{eqnarray}
and
\begin{eqnarray}\label{1-t-4-2}
 \|(U_2,K,S)-(0,K_0,S_0)\|_{H^4(\mn)}
\leq   \mc_1^\star \|(U_{2, en}^\star,K_{en}^\star,S_{en}^\star)-(0,K_0,S_0)\|_{C^4([-1,1])},
\end{eqnarray}
and
   \begin{equation}\label{1-t-4-aa}
 c^2(K,U_1, U_2, U_3,\Phi)-(U_1^2+U_2^2+U_3^2)
 \leq - \kappa^\star,  \ \ {\rm{in}} \ \ \overline\mn.
\end{equation}
Here $ c(K,U_1,  U_2,U_3,\Phi) $ is the sound speed given by
 \begin{equation*}
c(K,U_1,  U_2,\Phi)=\sqrt{(\gamma-1)\bigg(K+\Phi-\frac12(U_1^2+U_2^2+U_3^2)\bigg)},
\end{equation*}
  and the  positive constants
 $C_1^\star$  and $ \kappa^\star $ depend only on  $(b_0, \A_0,U_{0}, S_{0}, E_{0},r_0,r_1,\epsilon_0)$. Furthermore, the solution $(U_1,U_2,U_3,K,S, \Phi  )$ satisfies the compatibility conditions
\begin{equation}
\label{comp-cond-nlbvp-full-1}
\begin{aligned}
\p_{z} U_1=\p_{z}^k U_2=\p_z^{k-1}U_3=\p_{z}^k K
=\p_{z}^k S=\p_{z}^k\Phi=0, \ \ {\rm{on}} \ \Sigma_{1}^\pm,  \ \  {\rm{for}} \ \ k=1,3
\end{aligned}
\end{equation}
in the sense of trace.
For each $\alpha\in(0,1)$, it follows from \eqref{1-t-3-2}-\eqref{1-t-4-2} that
\begin{eqnarray}\label{1-t-7-2}
\begin{aligned}
  \|(U_1,U_3)-(\bar U,0)\|_{C^{1,\alpha}(\overline{\mn})} + \|\Phi-\bar\Phi\|_{C^{2,\alpha}(\overline{\mn})}
  \leq   \mc_2^\star\sigma(b^\star,U_{ 1,en}^\star,U_{2, en}^\star,U_{3, en}^\star,K_{en}^\star,S_{en}^\star,E_{ en}^\star,\Phi_{ ex}^\star),
  \end{aligned}
  \end{eqnarray}
  and
  \begin{eqnarray}\label{1-t-4-3}
 \begin{aligned}
\|(U_2,K,S)-(0,K_0,S_0)\|_{C^{2,\alpha}(\overline{\mn})}
\leq  \mc_2^\star\|(U_{2, en}^\star,K_{en}^\star,S_{en}^\star)-(0,K_0,S_0)\|_{C^4([-1,1])}
\end{aligned}
  \end{eqnarray}
for some constant $ \mc_2^\star>0$ depending only on $(b_0, \A_0,U_{0}, S_{0}, E_{0},  r_0, r_1,\epsilon_0,\alpha)$.
\end{theorem}

\begin{remark}
{\it Theorems \ref{th1} and \ref{th2} state the existence and uniqueness of
three-dimensional smooth supersonic  irrotational flows and smooth axisymmetric supersonic rotational flows, respectively.  Our ultimate goal is  to investigate three-dimensional supersonic  flow  with
nonzero vorticity for the steady Euler-Poisson system in the cylindrical nozzle $\m$. By the deformation-curl-Poisson decomposition developed in \cite{WX19,WS19}, the system \eqref{1-2} is equivalent to the system
\begin{equation}\label{dc}\begin{cases}
\begin{aligned}
&(c^2-U_1^2)\p_r U_1 + \frac{c^2-U_2^2}{r}\p_{\theta} U_2
+ (c^2-U_3^2)\p_{z} U_3 + \frac{c^2 U_1}{r}+\bigg( U_1\p_r+\frac{U_2}{r}\p_\th + U_3 \p_{z} \bigg)\Phi\\
&=U_1(U_2\p_r U_2+ U_3 \p_r U_3)
 + U_2\bigg(\frac{U_1}{r}\p_{\theta} U_1+\frac{U_3}{r}\p_{\theta} U_3\bigg)+ U_3 (U_1\p_{z} U_1+U_2 \p_{z} U_2),\\
&\frac{1}{r}\p_{\theta} U_3- \p_{z} U_2=\omega_1,\\
&\p_{z} U_1- \p_r U_3=\omega_2=\frac{1}{U_1}\bigg(U_2\omega_1+\p_{z}K
		-\frac{e^S\rho^{\gamma-1}}{\gamma-1}\p_{z}S\bigg),\\
&\p_r U_2- \frac{1}{r}\p_\theta U_1 + \frac{U_2}{r}=\omega_3=\frac{1}{U_1}\bigg(U_3\omega_1-\frac{1}{r}\p_{\theta}K
		+\frac{e^S\rho^{\gamma-1}}{\gamma-1}\frac{1}{r}\p_{\theta}S\bigg),\\
&\bigg(\p_r^2+\frac 1 r\p_r+\frac{1}{r^2}\p_{\th}^2+\p_z^2\bigg)\Phi=\rho-b,
\end{aligned}
\end{cases}
\end{equation}
where $\omega_1$ is solved by the transport equation
\begin{equation}\label{dc-1}\begin{aligned}
&\bigg(
	\p_r+\frac{U_2}{rU_1}\p_{\theta}+\frac{U_3}{U_1}\p_{z}
	\bigg)\omega_1
	+\bigg(
	\frac{1}{r}+\frac{1}{r}\p_{\theta}\bigg(\frac{U_2}{U_1}\bigg)+\p_{z}\bigg
(\frac{U_3}{U_1}\bigg)
	\bigg)\omega_1
	+\frac{1}{r}\p_{\theta}\bigg(\frac{1}{U_1}\bigg)\p_{z}K
	-\p_{z}\bigg(\frac{1}{U_1}\bigg)\frac{1}{r}\p_{\theta}K\\
&-\frac{1}{\gamma-1}
\frac{e^S}{r}\p_{\theta}\bigg(\frac{\rho^{\gamma-1}}{U_1}\bigg)\p_{z}S
+\frac{1}{\gamma-1}\p_{z}\bigg(\frac{\rho^{\gamma-1}}{U_1}\bigg)\frac{e^S}{r}\p_{\theta}S=0.
\end{aligned}
\end{equation}
For three-dimensional supersonic  rotational flows, if one looks for a solution $(U_1,U_2,U_3, K,S,\Phi)$ in $\left(H^3(\Omega)\right)^5\times H^4(\Omega)$,  the source terms in the curl system belong to $H^2(\Omega)$, one thus can derive only the $H^2(\Omega)$ norm estimate for the velocity due to the type  of the deformation-curl-Poisson  system in the supersonic region. Thus there is a possible loss of derivatives. This is indeed an essential difficulty. 
For three-dimensional axisymmetric rotational flows, this difficulty can be overcome  by utilizing the one order higher regularity of the stream
function and the fact that the Bernoulli's quantity and the entropy can be represented as  functions
of the stream function.  However, in the three-dimensional case, such a stream
function formulation is unavailable in general, how to overcome this difficult issue will be investigated
in the future.  }
\end{remark}
\section{The stability analysis within  irrotational flows}\label{irrotational}\noindent
\par In this section, we derive a second order linear hyperbolic-elliptic coupled system and  establish a priori estimates
 to obtain the well-posedness of the linearized problem. Then we prove Theorem \ref{th1}.
 \subsection{The linearized problem }\noindent
 \par The first equation in \eqref{2-9} can be rewritten as
\begin{eqnarray}\label{3-1}
 \begin{aligned}
 &\bigg(c^2(\Phi,\n \phi)-(\p_r\phi)^2\bigg)\p_r^2\phi+\frac{1}{r^2}
 \bigg(c^2(\Phi,\n \phi)-\frac{(\p_\th\phi)^2}
{r^2}\bigg)\p_\th^2\phi+ \bigg(c^2(\Phi,\n \phi)-(\p_z\phi)^2\bigg)\p_z^2\phi
-\frac{\p_r\phi\p_\th\phi}{r^2}\p_{r\th}^2\phi\\
&+\frac{\p_r\phi(\p_\th\phi)^2}{r^3}
-\p_r\phi\p_z\phi\p_{rz}^2\phi+\frac{c^2(\rho) \p_r\phi}{r}-\frac{\p_r\phi\p_\th\phi}{r^2}\partial_{\theta r}^2 \phi-\frac{\p_{\th}\phi\p_z\phi}{r^2} \partial_{\theta z}^2 \phi-\p_r\phi \p_z\phi \partial_{zr}^2\phi\\
&-\frac{\p_{\th}\phi\p_z\phi}{r^2}  \partial_{z\th}^2 \phi+\p_r\phi\p_r\Phi+\frac{\p_\th\phi\p_\th\Phi}{r^2}+\p_z\phi\p_z\Phi=0, \ \ {\rm{in}}\ \ \Omega.\\
\end{aligned}
  \end{eqnarray}
  For the  background solution,  $\bar\phi(r)=\int_{r_0}^r \bar U(s)\de s$ solves
\begin{eqnarray}\label{3-2}
(\bar c^2-(\p_{r} \bar{\phi})^2)\p_{r}^2 \bar{\phi}+\frac{\bar c^2 \p_r\bar \phi}{r}+\p_r\bar\phi\p_r\bar\Phi =0, \ \ {\rm{in}}\ \ \Omega.
 \end{eqnarray}
\par   Set
 \begin{eqnarray*}
 \psi(r,\th,z)=\phi(r,\th,z)-\bar\phi(r), \quad \Psi(r,\th,z)=\Phi(r,\th,z)-\bar \Phi(r), \ \  (r,\th,z)\in \Omega.
 \end{eqnarray*}
 Then $ (\psi,\Psi ) $ satisfies
 \begin{equation}\label{3-3}
 \begin{cases}
  A_{11}(r,\th,z,\n\psi,\Psi)\p_r^2\psi+\left(A_{12}(r,\th,z,\n\psi)
  +A_{21}(r,\th,z,\n\psi)\right)
 \p_{r\th}^2\psi\\
 +
 \left(A_{13}(r,\th,z,\n\psi)
+A_{31}(r,\th,z,\n\psi)\right)\p_{rz}^2\psi+  A_{22}(r,\th,z,\n\psi,\Psi)\p_{\th}^2\psi\\
+\left(A_{23}(r,\th,z,\n\psi)
+A_{32}(r,\th,z,\n\psi)\right)\p_{\th z}^2\psi
+A_{33}(r,\th,z,\n\psi,\Psi)\p_{z}^2\psi\\
+\x a_1(r)\p_r\psi+\x b_1(r)\p_r\Psi+\x b_2(r)\Psi=F_1(r,\th,z,
 \n\psi,\n\Psi,\Psi),\ \ &{\rm{in}}\ \ \Omega,\\
 \left(\p_r^2+\frac 1 r\p_r+\frac{1}{r^2}\p_{\th}^2+\p_z^2\right)\Psi+\x a_2(r)\p_r\psi-\x b_3(r)\Psi=F_2(r,\th,z,
 \n\psi,\Psi),\ \ &{\rm{in}}\ \ \Omega,\\
 \psi(r_0,\th,z)=0, \  (\p_r\psi, \p_r\Psi)(r_0,\th,z)=(U_{1,en}^\ast, E_{en}^\ast)(\th,z)-(U_0,E_0),\ \ &{\rm{on}}\ \ \Gamma_{en},\\
\Psi(r_1,\th,z)=\Phi_{ex}^\ast(\th,z)-\bar \Phi(r_1), \ \ &{\rm{on}}\ \ \Gamma_{ex},\\
\p_\th\psi(r,\pm\th_0,z)=\p_\th\Psi(r,\pm\th_0,z)=0, \ \ &{\rm{on}}\ \ \Gamma_{\th_0}^\pm,\\
\p_z\psi(r,\th,\pm 1)=\p_z\Psi(r,\th,\pm 1)=0, \ \ &{\rm{on}}\ \ \Gamma_{1}^\pm,\\
\end{cases}
 \end{equation}
 where
 \begin{equation*}
 \begin{aligned}
&A_{11}(r,\th,z,\n\psi,\Psi)=c^2(\Psi+\bar \Phi,\n \psi+\n \bar\phi)-(\p_r\psi+\bar U)^2, \\ &A_{12}(r,\th,z,\n\psi)=A_{21}(r,\th,z,\n\psi)=-\frac{(\p_r\psi+\bar U)\p_\th\psi}{r^2}, \\
&A_{13}(r,\th,z,\n\psi)=A_{31}(r,\th,z,\n\psi)=-(\p_r\psi+\bar U)\p_z\psi,\\
&A_{22}(r,\th,z,\n\psi,\Psi)=\frac{1}{r^2}
 \bigg(c^2(\Psi+\bar \Phi,\n \psi+\n \bar\phi)-\frac{(\p_\th\psi)^2}{r^2}\bigg), \\
&A_{23}(r,\th,z,\n\psi)=A_{32}(r,\th,z,\n\psi) =-\frac{\p_\th\psi\p_z\psi}{r^2}, \\
&A_{33}(r,\th,z,\n\psi,\Psi)=c^2(\Psi+\bar \Phi,\n \psi+\n \bar\phi)-(\p_z\psi)^2
 ,\\
 &\x a_1(r)=\frac{\bar c^2-(\gamma-1)\bar U^2}{r}-{(\gamma+1)(\bar U\bar U')}+{\bar E},\ \
 \x a_2(r)=\frac{\bar \rho\x U}{\bar c^2},\\
&\x b_1(r)={\bar U},\ \ \x b_2(r)={(\gamma-1)\bar U'}+\frac{(\gamma-1)\bar U  }{r}
, \ \  \x b_3(r)=\frac{\bar \rho}{\bar c^2},\\
&F_1(r,\th,z,
 \n\psi,\n\Psi,\Psi)=-\frac{\p_r\psi}r(c^2(\Phi,\n \phi)-\bar c^2)
 +\bigg(\frac{\gamma-1}2\bigg(\frac{(\p_\th\psi)^2}{r^2}+(\p_z\psi)^2\bigg)
 \frac{\gamma+1}2(\p_r\psi)^2\bigg)\bar U'\\
&\quad +\frac{\gamma-1}{2r}\bigg(\p_r\psi)^2
 \frac{(\p_\th\psi)^2}
{r^2}+(\p_z\psi)^2
\bigg)\bar U-\frac{\p_r\psi(\p_\th\psi)^2}{r^3}-
\p_r\psi\p_r\Psi-\frac{\p_\th\psi\p_\th\Psi}{r^2}+\p_z\psi\p_z\Psi,\\
&F_2(r,\th,z,
 \n\psi,\Psi)=\rho-\bar \rho+\x a_2\p_r\psi-\x b_3\Psi-(b^\ast-b_0).
\end{aligned}
 \end{equation*}
\par  Define the iteration set
\begin{equation}\label{3-4}
\begin{aligned}
  \ma_{\delta^\ast, r_1}&=\bigg\{(\psi,\Psi)\in \left(H^4(\m)\right)^2:\,\,\|(\psi,\Psi)\|_{H^4(\m)}\leq \delta^\ast,
    \p_{\th}^k \psi=\p_{\th}^k \Psi=0, \ \ {\rm{on}} \ \Gamma_{\th_0}^\pm,\\
    &\quad\ \ \ \p_{z}^k \psi=\p_{z}^k \Psi=0, \ \ {\rm{on}} \ \Gamma_{1}^\pm, \ \  {\rm{for}} \ \ k=1,3\bigg\},\\
   \end{aligned}
  \end{equation}
  where the positive constants $\delta^\ast$ and $r_1$ will be determined   later.
  For any given  $ (\e\psi,\e\Psi)\in \ma_{\delta^\ast, r_1} $,  we  consider the following
linearized problem:
  \begin{equation}\label{3-5}
 \begin{cases}
 \e L_1(\psi,\Psi)= A_{11}(r,\th,z,\n\e\psi,\e\Psi)\p_r^2\psi
+2A_{12}(r,\th,z,\n\e\psi)\p_{r\th}^2\psi\\
 \qquad\qquad\ \ +2A_{13}(r,\th,z,\n\e\psi)\p_{rz}^2\psi
 +  A_{22}(r,\th,z,\n\e\psi,\e\Psi)\p_{\th}^2\psi\\
\qquad\qquad\ \ +2A_{23}(r,\th,z,\n\e\psi)\p_{\th z}^2\psi +A_{33}(r,\th,z,\n\e\psi,\e\Psi)\p_{z}^2\psi\\
\qquad\qquad\ \ +\x a_1(r)\p_r\psi+\x b_1(r)\p_r\Psi+\x b_2(r)\Psi\\
\qquad\qquad\ \ =F_1(r,\th,z,
 \n\e\psi,\n\e\Psi,\e\Psi),\ \ &{\rm{in}}\ \ \Omega,\\
\e L_2(\psi,\Psi)= \left(\p_r^2+\frac 1 r\p_r+\frac{1}{r^2}\p_{\th}^2+\p_z^2\right)\Psi+\x a_2(r)\p_r\psi-\x b_3(r)\Psi\\
\qquad\quad \ \ \ =F_2(r,\th,z,
 \n\e\psi,\e\Psi),\ \ &{\rm{in}}\ \ \Omega,\\
 \psi(r_0,\th,z)=0, \ (\p_r\psi, \p_r\Psi)(r_0,\th,z)=(U_{1,en}^\ast, E_{en}^\ast)(\th,z)-(U_0,E_0),\ \ &{\rm{on}}\ \ \Gamma_{en},\\
\Psi(r_1,\th,z)=\Phi_{ex}^\ast(\th,z)-\bar \Phi(r_1), \ \ &{\rm{on}}\ \ \Gamma_{ex},\\
\p_\th\psi(r,\pm\th_0,z)=\p_\th\Psi(r,\pm\th_0,z)=0, \ \ &{\rm{on}}\ \ \Gamma_{\th_0}^\pm,\\
\p_z\psi(r,\th,\pm 1)=\p_z\Psi(r,\th,\pm 1)=0, \ \ &{\rm{on}}\ \ \Gamma_{1}^\pm.\\
 \end{cases}
 \end{equation}
 Set
  \begin{eqnarray*}
  \begin{aligned}
 &\h\Psi(r,\th,z)=\Psi(r,\th,z)-\Psi_{bd}(r,\th,z),  \ \ (r,\th,z)\in \Omega,
 \end{aligned}
 \end{eqnarray*}
 where
 \begin{eqnarray*}
  \Psi_{bd}(r,\th,z)=(r-r_1)(E_{en}^\ast(\th,z)-E_0)-(\Phi_{ex}^\ast(\th,z)-\bar\Phi(r_1)),\ \ (r,\th,z)\in \Omega.
 \end{eqnarray*}
  Then $(\psi,\h\Psi) $ satisfies
 \begin{equation}\label{3-6}
 \begin{cases}
\e L_1(\psi,\h\Psi)= A_{11}(r,\th,z,\n\e\psi,\e\Psi)\p_r^2 \psi+2A_{12}(r,\th,z,\n\e\psi)\p_{r\th}^2 \psi\\
\qquad\qquad\ \ +
 2A_{13}(r,\th,z,\n\e\psi)\p_{rz}^2 \psi
 +  A_{22}(r,\th,z,\n\e\psi)\p_{\th}^2 \psi\\
\qquad\qquad\ \ +2A_{23}(r,\th,z,\n\e\psi)\p_{\th z}^2 \psi +A_{33}(r,\th,z,\n\e\psi,\e\Psi)\p_{z}^2 \psi\\
\qquad\qquad\ \ +\x a_1(r)\p_r\psi+\x b_1(r)\p_r\h\Psi+\x b_2(r)\h\Psi\\
\qquad\qquad\ \ =F_3(r,\th,z,
 \n\e\psi,\n\e\Psi,\e\Psi),\ \ &{\rm{in}}\ \ \Omega,\\
\e L_2(\psi,\h\Psi)= \left(\p_r^2+\frac 1 r\p_r+\frac{1}{r^2}\p_{\th}^2+\p_z^2\right) \h\Psi+\x a_2(r)\p_r \psi-\x b_3(r) \h\Psi\\
\qquad\quad \ \ \ =F_4(r,\th,z,
 \n\e\psi,\e\Psi),\ \ &{\rm{in}}\ \ \Omega,\\
 \p_r\psi(r_0,\th, z)=F_5(\th, z),\ \psi(r_0,\th, z)=\p_r\h\Psi(r_0,\th, z)=0,\ \ &{\rm{on}}\ \ \Gamma_{en},\\
\h\Psi(r_1,\th,z)=0, \ \ &{\rm{on}}\ \ \Gamma_{ex},\\
\p_\th\psi(r,\pm\th_0,z)=\p_\th\h\Psi(r,\pm\th_0,z)=0, \ \ &{\rm{on}}\ \ \Gamma_{\th_0}^\pm,\\
\p_z\psi(r,\th,\pm 1)=\p_z\h\Psi(r,\th,\pm 1)=0, \ \ &{\rm{on}}\ \ \Gamma_{1}^\pm,\\
 \end{cases}
 \end{equation}
 where
 \begin{equation*}
 \begin{aligned}
 &  F_3(r,\th,z,
 \n\e\psi,\n\e\Psi,\e\Psi)= F_1(r,\th,z,
 \n\e\psi,\n\e\Psi,\e\Psi)-\e L_1(0,\Psi_{bd}),\ \ (r,\th,z)\in \Omega,\\
 &  F_4(r,\th,z,
 \n\e\psi,\e\Psi)= F_2(r,\th,z,
 \n\e\psi,\e\Psi)-\e L_2(0,\Psi_{bd}), \qquad \qquad \ (r,\th,z)\in \Omega,\\
 &F_5(\th,z)=U_{1,en}^\ast(\th,z)-U_0, \qquad \qquad\qquad\qquad  \qquad \qquad \qquad \ (\th,z)\in D.
 \end{aligned}
 \end{equation*}
 It   can be directly checked that there exists a constant $C>0$ depending only on $(b_0, \A_0, U_{0},S_{0},E_{0})$ such that
 \begin{equation}\label{3-45-e}
 \begin{aligned}
 &\|  F_3(\cdot,
 \n\e\psi,\n\e\Psi,\e\Psi)\|_{H^3(\m)}+\|  F_4(\cdot,
 \n\e\psi,\e\Psi)\|_{H^2(\m)}+\|F_5\|_{H^3(\overline D)}\\
 &\leq C\left(\|  F_1(\cdot,
 \n\e\psi,\n\e\Psi,\e\Psi)\|_{H^3(\m)}+\|  F_2(\cdot,
 \n\e\psi,\e\Psi)\|_{H^2(\m)}+\sigma(b^\ast,U_{ 1,en}^\ast,E_{ en}^\ast,\Phi_{ ex}^\ast)\right)\\
 &\leq C \left(\Vert(\tilde\psi,\tilde\Psi)\Vert^2_{H^4(\Omega)}
 +\sigma(b^\ast,U_{ 1,en}^\ast,E_{ en}^\ast,\Phi_{ ex}^\ast)\right)\leq
 C\left((\delta^\ast)^2+\sigma(b^\ast,U_{ 1,en}^\ast,E_{ en}^\ast,\Phi_{ ex}^\ast)\right).
 \end{aligned}
 \end{equation}
 \par For a fixed  $ (\e\psi,\e\Psi)\in \ma_{\delta^\ast, r_1} $, denote
  \begin{equation*}
  \begin{aligned}
&\e A_{ii}(r,\th,z):=A_{ii}(r,\th,z,\n\e\psi,\e\Psi),\quad \quad  (r,\th,z)\in \Omega, \  i=1,2,3, \\
&\e A_{ij}(r,\th,z):=A_{ij}(r,\th,z,\n\e\psi),\quad\quad\ \ \ \ (r,\th,z)\in \Omega, \  i\neq j =1,2,3, \\
 &\e F_3(r,\th,z):=F_3 (r,\th,z,\n\e\psi,\n\e\Psi,\e\Psi),\    (r,\th,z)\in \Omega, \\
 &\e F_4(r,\th,z):=F_4 (r,\th,z,\n\e\psi,\e\Psi),\quad\quad    (r,\th,z)\in \Omega.
 \end{aligned}
 \end{equation*}
   Then we have the following proposition.
\begin{proposition}\label{pro2}
 Let  $R_0$ be given by  Proposition \ref{pro1}.
\begin{enumerate}[ \rm (i)]
\item For any $r_1\in(r_0,r_0+R_0)$, set
\begin{eqnarray*}
\begin{aligned}
\bar A_{11}(r)=\bar c^2-{\bar U^2}, \  \
\bar A_{22}(r)=\frac{\bar c^2}{r^2}, \  \
\bar A_{33}(r)={\bar c^2},\ \ \bar A_{12}(r)=
\bar A_{13}(r)=\bar A_{23}(r)=0, \ \    r\in [r_0, r_1].
\end{aligned}
\end{eqnarray*}
 Then there exists a constant $\bar\mu_1\in(0,1)$ depending only on $(b_0, \A_0,U_{0}, S_{0},E_{0})$ such that
\begin{eqnarray}\label{2-7-a}
&&\bar\mu_1\leq-\bar A_{11}(r)\leq \frac{1}{\bar\mu_1},\ \ \bar\mu_1\leq\bar A_{ii}(r)\leq \frac{1}{\bar\mu_1}, \ i=2,3,
\ \  r\in [r_0, r_1] .
\end{eqnarray}
Furthermore, the coefficients $\bar A_{ii}$,   $\x b_i $ and $\x a_j  $ for $ i=1,2,3 $ and  $j=1,2$  are smooth functions. More precisely, for each $k\in\mathbb{Z}^+$, there exists a constant $\bar C_k>0$ depending only on $(b_0, \A_0,U_{0}, S_{0}, E_{0},r_0,r_1,k)$ such that
\begin{equation}\label{2-7-b}
\Vert (\bar A_{11},\bar A_{22},\bar A_{33}, \x a_1,\x a_2, \x b_1, \x b_2,\x b_3)\Vert_{C^k([r_0,r_1])}\leq\bar C_k.
\end{equation}
\item For each $\epsilon_0\in(0, R_0)$,
there exists a positive constant $\delta_0^\ast$ depending only on  $(b_0,   \A_0, U_{0},S_{0},E_{0},\epsilon_0)$ such that   for $r_1\in(r_0,r_0+R_0-\epsilon_0]$ and $\delta^\ast\leq \delta_0^\ast$, the coefficients $ (\e A_{11}, \e A_{12},\e A_{13},\e A_{22},\e A_{23},\e A_{33}) $ for $ (\e\psi,\e\Psi)\in \ma_{\delta^\ast, r_1}  $ with $\ma_{\delta^\ast, r_1} $    given by \eqref{3-4}  satisfy the following
properties.
\begin{enumerate}[ \rm (a)]
 \item There exists a positive constant  $C$ depending only on $(b_0, \A_0, U_{0},S_{0},E_{0},\epsilon_0)$ such that
 \begin{eqnarray}\label{2-7-f}
 \|(\e A_{11}, \e A_{12},\e A_{13},\e A_{22},\e A_{23},\e A_{33})-(\bar A_{11}, 0,0,\bar A_{22},0,\bar A_{33})\|_{H^3(\m)}\leq C\delta^\ast.
\end{eqnarray}
\item The coefficients $ (\e A_{11}, \e A_{12},\e A_{13},\e A_{22},\e A_{23},\e A_{33}) $ satisfy the following compatibility conditions:
\begin{eqnarray}\label{2-7-f-1}
\begin{cases}
\begin{aligned}
&\e A_{12}=\e A_{23}=\p_{\th}\e A_{11}=\p_{\th}\e A_{13}=\p_{\th}\e A_{22}=\p_{\th}\e A_{33} =0, \ \ {\rm{on}} \ \Gamma_{\th_0}^\pm,\\
  &\e A_{13}=\e A_{23}=\p_{z}\e A_{11}=\p_{z}\e A_{12}=\p_{z}\e A_{22}=\p_{z}\e A_{33} =0, \ \ \ {\rm{on}} \ \Gamma_{1}^\pm.
  \end{aligned}
  \end{cases}
  \end{eqnarray}

\item
   There exists a constant  $\mu_1\in (0,1)$  depending only on $(b_0,   \A_0, U_{0},S_{0},E_{0},\epsilon_0)$ so that  the coefficients $ (\e A_{11},\e A_{22},\e A_{33}) $ satisfy
 \begin{equation}\label{2-7-e}
{\mu_1}\leq -\e A_{11}(r,\th,z)\leq\frac1{\mu_1}, \ \ {\mu_1}\leq \e A_{ii}(r,\th,z)\leq\frac1{\mu_1} , \ i=2,3,  \ \ (r,\th,z)\in\overline{\m} .
\end{equation}
 Furthermore, The matrix $[\e A_{ij}]_{i,j=2}^3 $ is positive definite in  $\overline{\m} $ with
\begin{equation}\label{2-7-e-e}
\frac{\mu_1}2(\xi_2^2+\xi_3^2)\leq \sum_{i,j=2}^3 \e A_{ij}\xi_i \xi_j \leq \frac2{\mu_1}(\xi_2^2+\xi_3^2),
\ \ \forall (\xi_2,\xi_3)\in \mathbb{R}^2.
\end{equation}
\end{enumerate}
\end{enumerate}
\end{proposition}
\subsection{Energy estimates  for the linearized problem }\noindent
\par In this subsection, we  derive the energy estimates for the problem   \eqref{3-6} under the assumptions that $(\e A_{11}, \e A_{12},\e A_{13},\e A_{22},\e A_{23},\e A_{33})\in \left(C^{\infty}(\overline{\Omega})\right)^6$ and   \eqref{2-7-f}-\eqref{2-7-e-e}  hold.
\begin{lemma}\label{pro4}
 Let   $R_0 $ and  $ \delta_0^\ast$  be from  Proposition \ref{pro1} and Proposition \ref{pro2}, respectively.
  For each $\epsilon_0\in(0,R_0)$, there exists a constant $\bar r_1^\ast\in(r_0,r_0+R_0-\epsilon_0]$ depending only $(b_0,   \A_0, U_{0},S_{0},E_{0},\epsilon_0)$ and a sufficiently small constant  $ \delta_1^\ast\in(0,\frac{\delta_0^\ast}2]$
such that for  $r_1\in(r_0,\bar r_1^\ast)$ and $\delta^\ast\leq \delta_1^\ast$, the classical  solution $ (\psi,\h\Psi) $ to  \eqref{3-6} satisfies  the energy estimate
\begin{equation}\label{3-8}
\|(\psi,\h\Psi)\|_{H^1(\m)}\leq C\left(\|\e F_3\|_{L^2(\m)}+\|\e F_4\|_{L^2(\m)}+\|F_5\|_{L^2(\overline D)}\right),
\end{equation}
 where the constant   $C>0$ depends only on $(b_0,   \A_0, U_{0},S_{0},E_{0},\th_0,r_0,r_1,\epsilon_0)$.
\end{lemma}
\begin{proof}
We   utilize the multiplier method to derive the energy  estimate  \eqref{3-8}. More precisely,
let $Z(r) $ be  a smooth function defined   in $[r_0,r_1]$ to be determined, set
\begin{equation*}
\begin{aligned}
I_1(\psi,\h\Psi,Z)=\iiint_{\m} {Z}L_1(\psi,\h\Psi)\p_r\psi\de r \de \th\de z, \ \ I_2(\psi,\h\Psi)=\iiint_{\m}L_2(\psi,\h\Psi)\h\Psi\de r \de \th \de z.
\end{aligned}
\end{equation*}
For $ I_1 $,  integration by parts leads to
$$ I_1(\psi,\h\Psi,Z)=\sum_{i=1}^6 I_{1i}(\psi,\h\Psi,Z), $$
where
\begin{equation*}
\begin{aligned}
I_{11}(\psi,\h\Psi,Z)&=\iiint_{\m}Z\e A_{11}\p_r^2 \psi \p_r\psi\de r \de \th \de z\\
&=\iint_{D}\frac12\bigg(Z\e A_{11}(\p_r \psi)^2 \bigg)(r_1,\th,z) \de \th\de z-\iint_{D}\frac12\bigg(Z\e A_{11}(\p_r \psi)^2 \bigg)(r_0,\th,z) \de \th\de z\\
&\quad -\iiint_{\m}\frac12\p_r(Z\e A_{11})(\p_r \psi)^2\de r \de \th\de z,\\
I_{12}(\psi,\Psi,Z)&=\iiint_{\m}\bigg(2Z\e A_{12}\p_{r\th}^2 \psi \p_r\psi+2Z\e A_{13}\p_{rz}^2 \psi \p_r\psi\bigg)\de r \de \th\de z\\
&=-\iiint_{\m}\bigg(\p_{\th}(Z\e A_{12})(\p_r \psi)^2+\p_{z}(Z\e A_{13})(\p_r \psi)^2\bigg)\de r \de \th \de z,\\
I_{13}(\psi,\h\Psi,Z)&=\iiint_{\m}Z\e A_{22}\p_{\th}^2 \psi \p_r\psi\de r \de \th \de z\\
&=-\iiint_{\m}Z\p_\th \e A_{22}\p_r\psi\p_\th\psi \de r\de \th\de z-\iint_{D}\frac12\bigg(Z\e A_{22}(\p_\th \psi)^2 \bigg)(r_1,\th,z) \de \th\de z\\
&\quad+
\iiint_{\m}\frac12\p_r(Z\e A_{22})(\p_\th \psi)^2\de r \de \th\de z,\\
I_{14}(\psi,\h\Psi,Z)&=\iiint_{\m}2Z\e A_{23}\p_{\th z}^2 \psi \p_r\psi\de r \de \th \de z\\
&=\iiint_{\m}\bigg(-Z(\p_\th\e A_{23}\p_z\psi+\p_z\e A_{23}\p_\th\psi)\p_{r}\psi+\p_r(Z\e A_{23})\p_{\th}\psi\p_{z}\psi\bigg)\de r \de \th \de z\\
&\quad-\iint_D\bigg(Z\e A_{23}\p_{\th}\psi\p_{z}\psi\bigg)(r_1,\th,z) \de \th \de z,\\
\end{aligned}
\end{equation*}
\begin{equation*}
\begin{aligned}
I_{15}(\psi,\h\Psi,Z)&=\iiint_{\m}Z\e A_{33}\p_{z}^2 \psi \p_r\psi\de r \de \th \de z\\
&=-\iiint_{\m}Z\p_z \e A_{33}\p_r\psi\p_z\psi \de r\de \th\de z-\iint_{D}\frac12\bigg(Z\e A_{33}(\p_z \psi)^2 \bigg)(r_1,\th,z) \de \th\de z\\
&\quad+
\iiint_{\m}\frac12\p_r(Z\e A_{33})(\p_z \psi)^2\de r \de \th\de z,\\
I_{16}(\psi,\h\Psi,Z)&=\iiint_{\m}Z(\x a_1\p_r\psi+\x b_1\p_r\h\Psi+\x b_2\h\Psi)\p_r\psi\de r \de \th\de z.\\
\end{aligned}
\end{equation*}
For $I_2 $,   integration by parts gives
\begin{equation*}
\begin{aligned}
I_2(\psi,\h\Psi)&=\iiint_{\m}\left(\bigg(\p_r^2+\frac 1 r\p_r+\frac{1}{r^2}\p_{\th}^2+\p_z^2\bigg)\h\Psi+\x a_2\p_r\psi-\x b_3\h\Psi\right)\h\Psi\de r\de \th\de z\\
&=-\iiint_{\m}(\p_r\h\Psi)^2\de r\de \th\de z-\iiint_{\m}\frac1{r^2}(\p_\th\h\Psi)^2\de r\de \th\de z -\iiint_{\m}(\p_z\h\Psi)^2\de r\de \th\de z\\
&\quad-\iint_{D}\left(\frac{\h\Psi^2}{2r}\right)(r_0,\th,z)\de \th\de z+\iiint_{\m}\frac{\h\Psi^2}{2r^2}\de r\de \th \de z+\iiint_{\m}\x a_2\p_r\psi\h\Psi\de r\de \th\de z\\
&\quad-\iiint_{\m}\x b_3\h\Psi^2\de r\de \th\de z.
\end{aligned}
\end{equation*}
\par Therefore, one has
\begin{equation*}
\begin{aligned}
I_1(\psi,\h\Psi,Z)+I_2(\psi,\h\Psi)=-\left(J_1(\psi,\h\Psi,Z)+J_2(\psi,\h\Psi,Z)
+J_3(\psi,\h\Psi,Z)+J_4(\psi,\h\Psi,Z)\right),
\end{aligned}
\end{equation*}
where
\begin{equation*}
\begin{aligned}
J_1(\psi,\h\Psi,Z)&=\iiint_{\m} \bigg(\bigg(\frac12\p_r(Z\e A_{11})+Z\p_\th\e A_{12}+Z\p_z\e A_{13}-\bar a_1Z\bigg)(\p_r\psi)^2-\frac12\p_{r}(Z\e A_{22})(\p_\th\psi)^2\\
&\qquad\qquad-\frac12\p_{r}(Z\e A_{33})(\p_z\psi)^2+(\p_r\h\Psi)^2+\bigg(\x b_3-\frac1{2r^2}\bigg)\h\Psi^2+\frac1{r^2}(\p_\th\h\Psi)^2+(\p_z\h\Psi)^2\bigg)\de r\de \th \de z,
\\
J_2(\psi,\h\Psi,Z)&=\iiint_{\m}\bigg(-Z\x b_1\p_r\h\Psi \p_r\psi-Z\x b_2\h\Psi \p_r\psi -\x a_2\p_r\psi\h\Psi\bigg)\de r\de \th \de z,\\
J_3(\psi,\h\Psi,Z)&=\iiint_{\m}\bigg(Z\p_\th \e A_{22}\p_r\psi\p_\th\psi+Z\p_\th\e A_{23}\p_z\psi\p_{r}\psi+Z\p_z\e A_{23}\p_\th\psi \p_{r}\psi \\
&\qquad\qquad+Z\p_z \e A_{33}\p_r\psi\p_z\psi-\p_r(Z\e A_{23})\p_{\th}\psi\p_{z}\psi\bigg)\de r\de \th \de z,\\
\end{aligned}
\end{equation*}
\begin{equation*}
\begin{aligned}
J_4(\psi,\h\Psi,Z)&=\iint_{D}\bigg(-\frac12\bigg(Z\e A_{11}(\p_r \psi)^2 \bigg)(r_1,\th,z) +\frac12\bigg(Z\e A_{11}(\p_r \psi)^2 \bigg)(r_0,\th,z)+\frac12\bigg(Z\e A_{22}(\p_\th \psi)^2 \bigg)(r_1,\th,z)\\
&\qquad\quad+ \bigg(Z\e A_{23}\p_{\th}\psi\p_{z}\psi\bigg)(r_1,\th,z)+\frac12\bigg(Z\e A_{33}(\p_z \psi)^2 \bigg)(r_1,\th,z)+\bigg(\frac{\h\Psi^2}{2r}\bigg)(r_0,\th,z)\bigg)\de \th\de z.
 \end{aligned}
\end{equation*}
 It follows from Proposition  \ref{pro1} that
$$ \bar U^3(r)<2r_0J_0, \ \  r\in [r_0,r_0+R_0).$$
Thus one  derives
\begin{equation*}
\x b_3(r)-\frac1{2r^2}=\frac{\bar \rho(r)}{\bar c^2(r)}-\frac1{2r^2}>\frac{\bar \rho(r)}{\bar U^2(r)}-\frac1{2r^2}
>\frac{(\bar \rho\bar U)(r)}{2r_0 J_0}-\frac1{2r^2}>0, \ \  r\in [r_0,r_0+R_0).
\end{equation*}
For each $ r_1\in (r_0,r_0+R_0) $, set
\begin{equation}\label{2-19-2}
 \xi_{r_1}=\min_{[r_0,r_1]}\bigg(\x b_3(r)-\frac1{2r^2}\bigg),
\end{equation}
which is a strictly positive constant depending only on $(b_0,  \A_0,U_{0},S_{0},E_{0})$. Then one can use the Cauchy inequality to obtain that
\begin{equation*}
\begin{split}
\vert J_2(\psi,\h\Psi,Z)\vert
&\leq\iiint_{\Omega}\bigg(\frac{1}{8}(\p_r{\h\Psi})^2+\frac{\xi_{r_1}}{4}{\h\Psi}^2
+f(Z,r)(\p_r\psi)^2\bigg)
\mathrm{d}r\mathrm{d}\theta \de z,
\end{split}
\end{equation*}
where
\begin{equation*}
 f(Z,r)=
2\bigg(\bar b_1^2(r)+\frac{\bar b_2^2(r)}{\xi_{r_1}}\bigg)Z^2(r)
+2\frac{\bar a_2^2(r)}{\xi_{r_1}}.
\end{equation*}
\par Note that $ \e A_{ij}$  is close to $ \x A_{ij}$ provided that $ \delta^\ast $ in  the definition \eqref{3-4}  is sufficiently small. We first estimate
 \begin{equation}\label{3-18}
 \begin{aligned}
 &-(\x I_1(\psi,\h\Psi,Z)+\x I_2(\psi,\h\Psi))\\
 &=\iiint_{\m}\bigg(\bigg(\frac12(Z\x A_{11})'-\bar a_1Z\bigg)(\p_r\psi)^2
 -\frac12(Z\x A_{22})'(\p_\th\psi)^2- \frac12(Z\x A_{33})'(\p_z\psi)^2\\
 &\qquad\qquad+(\p_r\h\Psi)^2+\bigg(\x b_3-\frac1{2r^2}\bigg)\h\Psi^2+\frac1{r^2}(\p_\th\h\Psi)^2+(\p_z\h\Psi)^2\bigg)\de r\de \th \de z
 +J_2(\psi,\h\Psi,Z)\\
 &\quad+\iint_{D}\bigg(-\frac12\bigg(Z\x A_{11}(\p_r \psi)^2 \bigg)(r_1,\th,z) \de \th\de z+\frac12\bigg(Z\x A_{11}(\p_r \psi)^2 \bigg)(r_0,\th,z)\\
&\qquad\qquad\ \ +\frac12\bigg(Z\x A_{22}(\p_\th \psi)^2 \bigg)(r_1,\th,z)+\frac12\bigg(Z\x A_{33}(\p_z \psi)^2 \bigg)(r_1,\th,z)+\bigg(\frac{\h\Psi^2}{2r}\bigg)(r_0,\th,z)\bigg)\de \th\de z.\\
 \end{aligned}
\end{equation}
If the function $Z $ satisfies $ Z(r)>0 $ in $[r_0,r_1]$, it holds that
\begin{equation}\label{3-18-e}
 \begin{aligned}
 &-(\x I_1(\psi,\h\Psi,Z)+\x I_2(\psi,\h\Psi))\\
  &\geq \iiint_{\m}\bigg(\bigg(\frac12(Z\x A_{11})'-\bar a_1Z- f(Z,r)\bigg)(\p_r\psi)^2
 -\frac12(Z\x A_{22})'(\p_\th\psi)^2- \frac12(Z\x A_{33})'(\p_z\psi)^2\\
& \qquad\qquad\ \ \ \ +\frac{3\xi_{r_1}}{4}{\h\Psi}^2+
 \frac{7}{8}(\p_{r}{\h\Psi})^2+\frac{1}{r^2}(\p_{\th}{\h\Psi})^2+(\p_z\h\Psi)^2\bigg)
  \mathrm{d}r\mathrm{d}\theta\de z\\
 & \quad
 +
\iint_D\bigg(\x\mu_1\bigg(\frac{Z}{2}\left((\p_r\psi)^2+(\p_\th\psi)^2+(\p_z\psi)^2
\right)\bigg)
(r_1,\th,z)+\bigg(\frac{{\h\Psi}^2}{2r}\bigg)(r_0,\th,z) -\bigg(\frac{ZF_5^2}{2\x\mu_1}\bigg)(r_0,\th,z)\bigg)\mathrm{d}\theta \de z,
\end{aligned}
\end{equation}
where  $\x\mu_1 $ is given in \eqref{2-7-a}.  Hence we need to chose $ Z $ and $ r_1 $  such that for $ r\in[r_0,r_1] $, the function $ Z $ satisfies:
\begin{enumerate}[ \rm (i)]
 \item $ \min_{[r_0,r_1]} Z(r)>0 $;
 \item $\frac12(Z\x A_{11})'(r)-(\bar a_1Z)(r)- f(Z,r)\geq\lambda_0$;
 \item $-\frac12(Z\x A_{22})'(r)\geq\lambda_0$;
 \item $-\frac12(Z\x A_{33})'(r)\geq\lambda_0$,
 \end{enumerate}
    for some positive   constant $ \lambda_0 $.
    \par Set
\begin{equation*}
\xi_{1}=\min_{[r_0,r_0+R_0-\epsilon_0]}\bigg(\x b_3(r)-\frac1{2r^2}\bigg).
\end{equation*}
Then  for $ r_1\in(r_0,r_0+R_0-\epsilon_0]$, one has
\begin{equation}\label{mu_1}
\xi_{r_1}\geq\xi_{1}>0.
\end{equation}
Define
\begin{equation*}
\begin{aligned}
&\mathfrak {a}_0:=
\max_{(r_0,r_0+R_0-\epsilon_0]}\frac{4\bar a_2^2}{\bar\mu_1\xi_{1}},\qquad \
\mathfrak {a}_2:=\max_{(r_0,r_0+R_0-\epsilon_0]}\frac{4}{\bar\mu_1}
\bigg(\bar b_1^2+\frac{\bar b_2^2}{\xi_{1}}\bigg),\\
&\mathfrak {a}_{1,1}:=\min_{(r_0,r_0+R_0-\epsilon_0]}
\bigg(-\frac{1}{\x{A}_{11}}\bigg(\frac{\bar A_{11}'}{2}-\x a_1\bigg)\bigg),\ \
\mathfrak {a}_{1,2}:=\min_{(r_0,r_0+R_0-\epsilon_0]}
\bigg(-\frac{\bar A_{22}'}{2\bar A_{22}}\bigg),\\
 &\mathfrak {a}_{1,3}:=\min_{(r_0,r_0+R_0-\epsilon_0]}
\bigg(-\frac{\bar A_{33}'}{2\bar A_{33}}\bigg), \ \
\mathfrak {a}_1:=\min\{\mathfrak {a}_{1,1},\mathfrak {a}_{1,2},\mathfrak {a}_{1,3}\}.
\end{aligned}
\end{equation*}
It follows from Proposition  \ref{pro2} that there exists a constant $C_b>0$ depending only on $(b_0, \A_0,U_{0}, S_{0}, $\\$E_{0},\epsilon_0)$ such that
\begin{equation*}
0<\mathfrak {a}_i\leq C_b, \ \ i=0,2, \ \ \hbox{and} \ \ \vert \mathfrak {a}_1\vert\leq C_b.
\end{equation*}
Then the conditions ${\rm (i)-\rm (iv)}$ are satisfied as long as $Z>0$ and solves the ODE
\begin{equation}\label{ODE}
-Z'-\mathfrak {a}_2Z^2+2\mathfrak {a}_1Z-\mathfrak {a}_0=\lambda_0, \ \ r\in[r_0,r_1].
\end{equation}
  \eqref{ODE} can be solved by dividing two cases $ \mathfrak {a}_0\mathfrak {a}_2-\mathfrak {a}_1^2>0 $ and $ \mathfrak {a}_0\mathfrak {a}_2-\mathfrak {a}_1^2<0 $. Furthermore,   there exists a  positive constant $\bar r_1^\ast\in(r_0,r_0+R_0-\epsilon_0]$  so that for $r_1\in(r_0,\bar r_1^\ast)$, one can find a  positive constant  $\lambda_0 $  depending only on $(b_0, \A_0, U_{0},S_{0},E_{0},r_0,r_1,\epsilon_0)$   such  that
 the solution $ Z $ satisfies all the four  conditions stated in the above. The details  can be found in      \cite[Page 614-615]{DWY23}.
Therefore, we obtain from \eqref{3-18-e} that
\begin{equation}\label{3-19}
 \begin{aligned}
 &-(\x I_1(\psi,\h\Psi,Z)+\x I_2(\psi,\h\Psi))\\
  &\geq \iiint_{\m}\bigg(\lambda_0\left((\p_r\psi)^2+(\p_\th\psi)^2+(\p_z\psi)^2\right) +\frac{3\xi_{r_1}}{4}{\h\Psi}^2+
 \frac{7}{8}(\p_{r}{\h\Psi})^2+\frac{1}{r^2}(\p_{\th}{\h\Psi})^2+(\p_z\h\Psi)^2\bigg)
  \mathrm{d}r\mathrm{d}\theta\de z\\
 & \quad
 +
\iint_D\bigg(\x\mu_1\bigg(\frac{Z}{2}\left(
(\p_r\psi)^2+(\p_\th\psi)^2+(\p_z\psi)^2
\right)\bigg)
(r_1,\th,z)+\bigg(\frac{{\h\Psi}^2}{2r}\bigg)(r_0,\th,z) -\bigg(\frac{ZF_5^2}{2\x\mu_1}\bigg)(r_0,\th,z)\bigg)\mathrm{d}\theta \de z.
\end{aligned}
\end{equation}
\par By Proposition \ref{pro2}, one can fix a small constant $ \delta_1^\ast\in(0, \frac{\delta_0^\ast}2] $ depending on $ (b_0, \A_0, U_{0},S_{0},E_{0},\epsilon_0)$ so that if the constant $ \delta^\ast $ in  the definition \eqref{3-4} satisfies
$ \delta^\ast \leq \delta_1^\ast, $ the
classical solution $ (\psi,\h\Psi) $ to the linear boundary value problem \eqref{3-6} associated with $ (\e\psi,\e\Psi)\in \ma_{\delta^\ast, r_1}  $ satisfies the estimate
\begin{equation}\label{3-20}
 \begin{aligned}
 &-( I_1(\psi,\h\Psi,Z)+ I_2(\psi,\h\Psi))\\
  &\geq \iiint_{\m}\bigg(\frac{\lambda_0}2\left((\p_r\psi)^2+(\p_\th\psi)^2+(\p_z\psi)^2\right) +\frac{3\xi_{r_1}}{4}{\h\Psi}^2+
 \frac{7}{8}(\p_{r}{\h\Psi})^2+\frac{1}{r^2}(\p_{\th}{\h\Psi})^2+(\p_z\h\Psi)^2\bigg)
  \mathrm{d}r\mathrm{d}\theta\de z\\
 & \quad
 +
\iint_D\bigg(\mu_1\bigg(\frac{Z}{2}\left(
(\p_r\psi)^2+(\p_\th\psi)^2+(\p_z\psi)^2
\right)\bigg)
(r_1,\th,z)+\bigg(\frac{{\h\Psi}^2}{2r}\bigg)(r_0,\th,z) -\bigg(\frac{ZF_5^2}{2\mu_1}\bigg)(r_0,\th,z)\bigg)\mathrm{d}\theta \de z,
\end{aligned}
\end{equation}
where $\mu_1 $ is given in \eqref{2-7-e-e}.
Furthermore, \eqref{3-6} shows that
 \begin{equation}\label{3-21}
 \begin{aligned}
 &-( I_1(\psi,\h\Psi,Z)+ I_2(\psi,\h\Psi))=-\iiint_{\m}\left(Z\e F_3\p_r\psi+\e F_4\h\Psi\right)\de r \de \th\de z\\
 &\leq\iiint_{\Omega}\bigg(\frac{\lambda_0}{4}(\p_r\psi)^2
+\frac{1}{\lambda_0}\vert Z\e F_3\vert^2
+\frac{\xi_{r_1}}{4}{\h\Psi}^2+\frac{1}{\xi_{r_1}}
\e F_4^2\bigg)\mathrm{d}r\mathrm{d}\theta\de z.
 \end{aligned}
 \end{equation}
  Note that $\max_{r\in[r_0,r_1]}Z(r)>0$  depends only on $(b_0,   \A_0, U_{0},S_{0},E_{0},r_0,r_1,\epsilon_0)$.  Then \eqref{3-21} together with \eqref{3-20} and the Poincar\'{e} inequality   yields that
\begin{equation*}
\|(\psi,\h\Psi)\|_{H^1(\m)}\leq C\left(\|\e F_3\|_{L^2(\m)}+\|\e F_4\|_{L^2(\m)}+\|F_5\|_{L^2(\overline D)}\right),
\end{equation*}
 where the constant  $C>0$ depends only on $(b_0,   \A_0, U_{0},S_{0},E_{0},\th_0,r_0,r_1,\epsilon_0)$.
\end{proof}
With the $ H^1 $ energy estimate,  one can  regard the system
 \eqref{3-6} as two individual
boundary value problems for $ \psi $ and $\h\Psi $ by leaving only the principal term on the left-hand side of
the equations to  derive the higher order derivative estimates for $ (\psi,\h\Psi) $.
\begin{lemma}\label{pro5}
  For fixed $\epsilon_0\in(0,R_0)$, let  $\bar r_1^\ast\in(r_0,r_0+R_0-\epsilon_0]$ and $ \delta_1^\ast\in(0, \frac{\delta_0^\ast}2] $ be from  Lemma \ref{pro4}. There exists a positive constant $C$ depending only on $(b_0,   \A_0, U_{0},S_{0},E_{0},\th_0,r_0,r_1,\epsilon_0)$ and a sufficiently small constant  $ \delta_2^\ast\in(0,{\delta_1^\ast}]$
such that for  $ r_1\in (r_0,\bar r_1^\ast) $ and $\delta^\ast\leq \delta_2^\ast$, the classical   solution $ (\psi,\h\Psi) $ to  \eqref{3-6} satisfies
\begin{eqnarray}\label{H^4-1-1-1}
&&\Vert\psi\Vert_{H^4(\Omega)}
\leq C\left(\|\e F_3\|_{H^3(\m)}+\|\e F_4\|_{H^2(\m)}+\|F_5\|_{H^3(\overline D)}\right),\\\label{H^4-2-2}
&&\Vert\h\Psi\Vert_{H^4(\Omega)}
\leq C\left(\|\e F_3\|_{H^2(\m)}+\|\e F_4\|_{H^2(\m)}+\|F_5\|_{H^2(\overline D)}\right).
\end{eqnarray}

\end{lemma}
\begin{proof}

The proof is divided into three steps.
 \par { \bf Step 1. A priori $H^2$-estimate.}
\par We first derive the $H^2$ estimate for $\h\Psi$.
   $ \h\Psi $ can be regarded as a classical solution to the linear boundary value problem:
    \begin{equation}\label{3-22}
\begin{cases}
\bigg(\p_r^2+\frac 1 r\p_r+\frac{1}{r^2}\p_{\th}^2+\p_z^2\bigg)\h\Psi-\x b_3\h\Psi=-\x a_2\p_r\psi+\e F_4,\ \ &{\rm{in}}\ \ \Omega,\\
\p_r \h\Psi(r_0,\th,z)=0,\ \ &{\rm{on}}\ \ \Gamma_{en},\\
\h\Psi(r_1,\th,z)=0, \ \ &{\rm{on}}\ \ \Gamma_{ex},\\
\p_\th \h\Psi(r,\pm\th_0,z)=0, \ \  &{\rm{on}}\ \ \Gamma_{\th_0}^\pm,\\
\p_z \h\Psi(r,\th,\pm 1)=0, \ \  &{\rm{on}}\ \ \Gamma_{1}^\pm.\\
\end{cases}
\end{equation}
Then one can   use the method of reflection
and  apply \cite[Theorems 8.8 and 8.12]{GT98} to obtain the estimate
 \begin{equation*}
 \begin{aligned}
\Vert\h\Psi\Vert_{H^2(\Omega)}
\leq C\left(\Vert\e F_4\Vert_{L^2(\Omega)}+\Vert\p_r\psi\Vert_{L^2(\Omega)}
+\Vert\h\Psi\Vert_{L^2(\Omega)}\right).
\end{aligned}
\end{equation*}
Combining this estimate with \eqref{3-8} yields
\begin{equation}\label{3-23}
 \begin{aligned}
\Vert\h\Psi\Vert_{H^2(\Omega)}
\leq C\left(\|\e F_3\|_{L^2(\m)}+\|\e F_4\|_{L^2(\m)}+\|F_5\|_{L^2(\overline D)}\right)
\end{aligned}
\end{equation}
   for the constant $C>0$ depending only on $(b_0,   \A_0, U_{0},S_{0},E_{0},\th_0,r_0,r_1,\epsilon_0)$.
 \par In the following, we derive the $H^2$ estimate for $\psi$.
 The first equation in \eqref{3-6} can be rewritten as
  \begin{equation}\label{3-24}
  \begin{aligned}
&\e A_{11}\p_r^2 \psi+2\e A_{12}\p_{r\th}^2\psi+
 2\e A_{13}\p_{rz}^2\psi
 + \e A_{22}\p_{\th}^2\psi+2\e A_{23}\p_{\th z}^2\psi
  +\e A_{33}\p_{z}^2\psi\\
 &=-\x a_1\p_r\psi -\x b_1\p_r\h\Psi-\x b_2\h\Psi+\e F_3, \ \ {\rm{in}}\ \ \Omega.
\end{aligned}
\end{equation}
Let $ h_1=\p_r\psi $. Differentiating \eqref{3-24} with respect to $r$ yields
\begin{equation}\label{3-25}
  \begin{aligned}
&\e A_{11}\p_r^2 h_1+2\e A_{12}\p_{r\th}^2h_1+
 2\e A_{13}\p_{rz}^2h_1
 + \e A_{22}\p_{\th}^2h_1+2\e A_{23}\p_{\th z}^2h_1+\e A_{33}\p_{z}^2h_1\\
&
=-\p_r\e A_{11}\p_rh_1-2\p_r\e A_{12}\p_{\th}h_1-2\p_r\e A_{13}\p_{z}h_1-\p_r\e A_{22}\p_{\th}^2\psi-2\p_r\e A_{23}\p_{\th z}^2\psi\\
&\quad -\p_r\e A_{33}\p_{z}^2\psi-\x a_1\p_rh_1-\x a_1'h_1+\p_r(-\x b_1\p_r\h\Psi-\x b_2\h\Psi+\e F_3), \  \ {\rm{in}}\ \ \Omega.
\end{aligned}
\end{equation}
For $ r_0<t<r_1 $, set
 $$ \Omega_t=\{(r,\theta,z):r_0<r<t,\, (\th,z)\in D \}.$$  By  integrating by parts with  the  boundary conditions in \eqref{3-6}, one obtains
 \begin{equation*}
\begin{aligned}
V(h_1)&=\iiint_{\Omega_t}\left(\e A_{11}\p_r^2 h_1+2\e A_{12}\p_{r\th}^2h_1+
 2\e A_{13}\p_{rz}^2h_1
 + \e A_{22}\p_{\th}^2h_1+2\e A_{23}\p_{\th z}^2h_1+\e A_{33}\p_{z}^2h_1\right)\p_r h_1 \de r \de \th \de z\\
 &=\iint_{D}\bigg(\frac{\e A_{11}}{2}(\p_r h_1)^2-\frac{\e A_{22}}{2}(\p_\th h_1)^2-\frac{\e A_{33}}{2}(\p_z h_1)^2-\e A_{23}\p_{\th}h_1\p_{z}h_1\bigg)(t,\th,z)\de \th\de z\\
 &\quad-\iint_{D}\bigg(\frac{\e A_{11}}{2}(\p_r h_1)^2
 -\frac{\e A_{22}}{2}(\p_\th h_1)^2-\frac{\e A_{33}}{2}(\p_z h_1)^2-\e A_{23}\p_{\th}h_1\p_{z}h_1\bigg)(r_0,\th,z)\de \th\de z\\
 &\ \ -\iiint_{\Omega_t}\bigg(\bigg(\frac{\p_r\e A_{11}}{2}+\p_\th\e A_{12}+\p_z\e A_{13}\bigg)
 (\p_rh_1)^2
  +\p_\th\e A_{22}\p_\th h_1\p_r h_1-\frac{\p_r\e A_{22}}{2}(\p_\th h_1)^2+\p_z\e A_{33}\p_z h_1\p_r h_1\\
 &\qquad\quad\quad \ \ -\frac{\p_r\e A_{33}}{2}(\p_z h_1)^2+(\p_\th\e A_{23}\p_zh_1
  +\p_z\e A_{23}\p_\th h_1)\p_{r}h_1-\p_r(\e A_{23})\p_{\th}h_1\p_{z}h_1\bigg)\de r \de \th \de z.
 \end{aligned}
 \end{equation*}
  Note that $ \delta_{1}^{\ast}\in(0,\frac{\delta_0^{\ast}}2] $ is fixed  depending only on $(b_0,   \A_0, U_{0},S_{0},E_{0},\epsilon_0)$. Then if
 $\delta^{\ast} \leq  \delta_{1}^{\ast}, $
 it follows from \eqref{2-7-f} that
  \begin{eqnarray}\label{2-7-f-p}
 \|(\e A_{11}, \e A_{12},\e A_{13},\e A_{22},\e A_{23},\e A_{33})\|_{H^3(\m)}\leq C
\end{eqnarray}
for the constant $C>0$ depending only on $(b_0, \A_0, U_{0},S_{0},E_{0},\epsilon_0)$.
  Combining  Proposition \ref{pro2}, the Morrey inequality, the Cauchy inequality and  \eqref{2-7-f-p} yields that \begin{equation}\label{3-26}
\begin{aligned}
V(h_1)&\leq \iint_{D}\frac{1}{\mu_1}\left(\p_r h_1)^2+(\p_\th h_1)^2+(\p_z h_1)^2\right)(r_0,\th,z)\de \th\de z\\
 &\quad-\iint_{D}\frac{\mu_1}{2}\left(\p_r h_1)^2+(\p_z h_1)^2+(\p_\th h_1)^2\right)(t,\th,z)\de \th\de z\\
 &\quad +C\iiint_{\Omega_t}\left((\p_r h_1)^2+(\p_\th h_1)^2+(\p_z h_1)^2\right)\de r \de \th \de z.
 \end{aligned}
 \end{equation}
 Furthermore, it follows from  \eqref{3-25} that
 \begin{equation*}
\begin{aligned}
V(h_1)&=\iiint_{\Omega_t}\left(\e A_{11}\p_r^2 h_1+2\e A_{12}\p_{r\th}^2h_1+
 2\e A_{13}\p_{rz}^2h_1
 + \e A_{22}\p_{\th}^2h_1+2\e A_{23}\p_{\th z}^2h_1+\e A_{33}\p_{z}^2h_1\right)\p_r h_1 \de r \de \th \de z\\
 &=\iiint_{\Omega_t}\bigg(-\p_r\e A_{11}\p_rh_1-2\p_r\e A_{12}\p_{\th}h_1-2\p_r\e A_{13}\p_{z}h_1-\p_r\e A_{22}\p_{\th}^2\psi-2\p_r\e A_{23}\p_{\th z}^2\psi\\
&\qquad\qquad \ \ -\p_r\e A_{33}\p_{z}^2\psi-\x a_1\p_rh_1-\x a_1'h_1+\p_r(-\x b_1\p_r\h\Psi-\x b_2\h\Psi+\e F_3)\bigg)\p_r h_1 \de r \de \th \de z.\\
\end{aligned}
 \end{equation*}
 Then one can   use the Cauchy inequality  and combine \eqref{3-8},    \eqref{3-23},  and \eqref{2-7-f-p} to  obtain
 \begin{equation}
\label{3-27}
\begin{split}
&\iint_{D}\frac{\mu_1}{2}\bigg((\p_r h_1)^2+(\p_\th h_1)^2+(\p_z h_1)^2\bigg)(t,\th,z)\de \th \de z
\leq
C\left(\|\e F_3\|_{L^2(\m)}+\|\e F_4\|_{L^2(\m)}+\|F_5\|_{C^0(\overline D)}\right)^2\\
&+\iint_{D}\mu_1\bigg((\p_r h_1)^2+(\p_\th h_1)^2+(\p_z h_1)^2\bigg)(r_0,\th, z)\de \th \de z
+C\iiint_{\Omega_t}\left((\p_r h_1)^2+(\p_\th h_1)^2+(\p_z h_1)^2\right.\\
&\left.+ (\p_r\e F_3)^2+(\p_{\th}^2\psi)^2+(\p_{\th z}^2\psi)^2+(\p_{z}^2\psi)^2\right)\mathrm{d}r\mathrm{d}\theta \de z.
\end{split}
\end{equation}
The estimate constant $C>0 $ in \eqref{3-26}-\eqref{3-27} depends only on $(b_0,   \A_0, U_{0},S_{0},E_{0},\th_0,r_0,r_1,\epsilon_0)$.
\par Note that $h_1(r_0,\th,z)=F_5(\th,z)$. Then one has
\begin{equation}\label{3-28}
 \iint_{D}\bigg((\p_\th h_1)^2+(\p_z h_1)^2\bigg)(r_0,\th,z)\de \th\de z= \iint_{D}(|\p_\th F_5|^2+|\p_z F_5|^2)(\th,z)\de \th \de z\leq C\|F_5\|_{H^1(D)}^2.
 \end{equation}
In addition,  \eqref{3-24} yields that
   \begin{equation*}
   \begin{aligned}
   &\iint_{D}\bigg(\e A_{11}(\p_r h_1)^2+2\e A_{12}\p_{\th}h_1\p_r h_1+
 2\e A_{13}\p_{z}h_1\p_r h_1
\bigg)(r_0,\th,z) \de \th\de z\\
 &=\iint_{D}\bigg((-\e A_{22}\p_{\th}^2\psi-2\e A_{23}\p_{\th z}^2\psi
  -\e A_{33}\p_{z}^2\psi-\x a_1h-\x b_1\p_r\h\Psi-\x b_2\h\Psi+\e F_3)\p_r h\bigg)(r_0,\th,z) \de \th \de z.
 \end{aligned}
\end{equation*}
 Owing to $ \psi(r_0,\th,z)=\p_r \h\Psi(r_0,\th,z)=0$, it follows from   the Cauchy inequality and the trace theorem  together with  \eqref{3-8},    \eqref{3-23} and \eqref{3-28}    that
\begin{equation}\label{3-29}
  \iint_{D}(\p_r h)^2(r_0,\th,z )\de \th \de z\leq C\left(\|\e F_3\|_{H^1(\m)}+\|\e F_4\|_{L^2(\m)}+\|F_5\|_{H^1(\overline D)}\right)^2.
 \end{equation}
 Then collecting the estimates \eqref{3-27}-\eqref{3-29} leads to
 \begin{equation}
\label{3-30}
\begin{split}
&\iint_{D}\frac{\mu_1}{2}\bigg((\p_r h_1)^2+(\p_\th h_1)^2+(\p_z h_1)^2\bigg)(t,\th,z)\de \th \de z
\leq
C\left(\|\e F_3\|_{H^1(\m)}+\|\e F_4\|_{L^2(\m)}+\|F_5\|_{H^1(\overline D)}\right)^2\\
&+C\iiint_{\Omega_t}\left((\p_r h_1)^2+(\p_\th h_1)^2+(\p_z h_1)^2
+(\p_{\th}^2\psi)^2+(\p_{\th z}^2\psi)^2+(\p_{z}^2\psi)^2\right)\mathrm{d}r\mathrm{d}\theta \de z.
\end{split}
\end{equation}
The estimate constant $C>0 $ in \eqref{3-28}-\eqref{3-30} depends only on $(b_0,   \A_0, U_{0},S_{0},E_{0},\th_0,r_0,r_1,\epsilon_0)$.
 \par Next, we estimate $$\iiint_{\Omega_t}\left(
(\p_{\th}^2\psi)^2+(\p_{\th z}^2\psi)^2+(\p_{z}^2\psi)^2\right)\mathrm{d}r\mathrm{d}\theta \de z. $$
For each $ r\in(r_0,r_1) $,  there holds
\begin{equation}\label{3-31-t}
\begin{aligned}
& \e A_{22}\p_{\th}^2\psi+2\e A_{23}\p_{\th z}^2\psi
  +\e A_{33}\p_{z}^2\psi\\
 &=-\e A_{11}\p_r^2 \psi-2\e A_{12}\p_{r\th}^2\psi-
 2\e A_{13}\p_{rz}^2\psi-\x a_1\p_r\psi -\x b_1\p_r\h\Psi-\x b_2\h\Psi+\e F_3, \ \ {\rm{in}}\ \ D.
 \end{aligned}
 \end{equation}
 Then  it follows from \eqref{2-7-e-e} that   \eqref{3-31-t} can be regarded  as a two-dimensional elliptic equation for $ \psi $ in $D $.
Therefore, for each $ r\in(r_0,r_1) $, we consider the following elliptic problem:
 \begin{equation}\label{3-31}
 \begin{cases}
\x A_{22}\p_{\th}^2\psi
  +\x A_{33}\p_{z}^2\psi=\e \mf_3,\ \ &{\rm{in}}\ \ D,\\
  \p_\th \psi(r,\pm\th_0,z)=0, \ \  &{\rm{on}}\ \ \Gamma_{\th_0}^\pm,\\
   \p_z \psi(r,\th,\pm 1)=0, \ \  &{\rm{on}}\ \ \Gamma_{1}^\pm,\\
  \end{cases}
\end{equation}
where
\begin{equation*}
\begin{aligned}
\e \mf_3&=\e F_3-\x a_1h_1-\x b_1\p_r\h\Psi-\x b_2\h\Psi-\e A_{11}\p_{r}h_1-2\e A_{12}\p_{\th}h+
 2\e A_{13}\p_{z}h+(\x A_{22}-\e A_{22})\p_{\th}^2\psi\\
  &\quad +(\x A_{33}-\e A_{33})\p_{z}^2\psi-2\e A_{23}\p_{z\th}^2\psi.\\
  \end{aligned}
  \end{equation*}
Then one can  apply the   classical elliptic theory   in \cite{GT98} to derive that
   \begin{equation}
\label{3-32}
 \begin{aligned}
  &\iint_{D}(\p_{\th}^2\psi+\p_{\th z}^2\psi+\p_{z}^2\psi)^2
  \mathrm{d}\theta \de z
 \leq C\iint_{D}\e \mf_3^2 \mathrm{d}\theta \de z.\\
\end{aligned}
\end{equation}
Integrating the above inequality with respect to $r$ over $[r_0,t]$ and combining  the Cauchy inequality  and \eqref{2-7-f} give that
\begin{equation}\label{3-33}
 \begin{aligned}
&\iiint_{\m_t}\left((\p_{\th}^2\psi)^2+(\p_{\th z}^2\psi)^2+(\p_{z}^2\psi)^2\right)\mathrm{d}r\mathrm{d}\theta \de z
\leq C\iiint_{\m_t}\left(\e F_3^2+h_1^2+\h\Psi^2+(\p_r\h\Psi)^2\right)\mathrm{d}r\mathrm{d}\theta \de z\\
&+C\iiint_{\m_t}\left((\p_r h_1)^2+(\p_\th h_1)^2+(\p_z h_1)^2\right)\mathrm{d}r\mathrm{d}\theta \de z
+C\delta^\ast\iiint_{\m_t}\left((\p_{\th}^2\psi)^2+(\p_{\th z}^2\psi)^2+(\p_{z}^2\psi)^2\right)\mathrm{d}r\mathrm{d}\theta \de z,
\end{aligned}
\end{equation}
where the constant  $C>0 $ depends only on $(b_0,   \A_0, U_{0},S_{0},E_{0},\epsilon_0)$.
 Let $ \delta_2^\ast\in(0,\delta_1^{\ast}] $ be   sufficiently small depending only on $(b_0,   \A_0, U_{0},S_{0},E_{0},\epsilon_0)$ such that
 $C\delta^\ast \leq  C\delta_2^{\ast}\leq \frac12. $ Then
   \eqref{3-33} together with  \eqref{3-8} and     \eqref{3-23} yields that
 \begin{equation}\label{3-34}
 \begin{aligned}
&\iiint_{\Omega_t}\left((\p_{\th}^2\psi)^2+(\p_{\th z}^2\psi)^2+(\p_{z}^2\psi)^2\right)\mathrm{d}r\mathrm{d}\theta \de z\\
&\leq
2C\left(\|\e F_3\|_{L^2(\m)}+\|\e F_4\|_{L^2(\m)}+\|F_5\|_{L^2(\overline D)}\right)^2+2C\iiint_{\Omega_t}\left((\p_r h_1)^2+(\p_\th h_1)^2+(\p_z h_1)^2\right)\mathrm{d}r\mathrm{d}\theta \de z.
\end{aligned}
\end{equation}

\par Collecting the estimates \eqref{3-30} and \eqref{3-34} leads to
\begin{equation}\label{3-35}
\begin{aligned}
 &\frac{\de\left(\iiint_{\Omega_t}\left((\p_r h_1)^2+(\p_\th h_1)^2+(\p_z h_1)^2\right)\de r\mathrm{d}\theta \de z\right)}
 {\de t} \\
 &\leq \mc_1\left(\|\e F_3\|_{H^1(\m)}+\|\e F_4\|_{L^2(\m)}+\|F_5\|_{H^1(\overline D)}\right)^2+\mc_2\iiint_{\Omega_t}\left((\p_r h_1)^2+(\p_\th h_1)^2+(\p_z h_1)^2\right)\de r\mathrm{d}\theta \de z.
 \end{aligned}
\end{equation}
Then the  Gronwall inequality yields that
\begin{equation}\label{3-36}
\begin{aligned}
 \iiint_{\Omega}\left((\p_r h_1)^2+(\p_\th h_1)^2+(\p_z h_1)^2\right)\de r\mathrm{d}\theta \de z
 \leq \mc_3\left(\|\e F_3\|_{H^1(\m)}+\|\e F_4\|_{L^2(\m)}+\|F_5\|_{H^1(\overline D)}\right)^2.
 \end{aligned}
\end{equation}
The estimate constants  $\mc_i>0 \ (i=1,2,3) $ depend only on $(b_0,   \A_0, U_{0},S_{0},E_{0},\th_0,r_0,r_1,\epsilon_0)$. Therefore, \eqref{3-36} together with \eqref{3-34}  yields
\begin{equation}\label{H2-estimate-vm-final}
  \|\psi\|_{H^2(\m)}\leq C\left(\|\e F_3\|_{H^1(\m)}+\|\e F_4\|_{L^2(\m)}+\|F_5\|_{H^1(\overline D)}\right)
\end{equation}
for the constant $C>0$ depending only on $(b_0, \A_0, U_{0},S_{0},E_{0},\th_0,\epsilon_0,r_0,r_1,\epsilon_0)$.
\par { \bf Step 2. A priori $H^3$-estimate.}
\par A priori $H^3$-estimate  can be directly given by the method
of reflection just as in a priori $H^2$-estimate. But a lengthy computation must be done in order to verify it.
So, in this step, we only describe main differences in establishing the estimate for higher order
weak derivatives of $(\psi,\h\Psi)$ in $\m $.
\par 
 Note that the extension of $-\x a_2\p_r\psi+\e F_4$ given by even reflection about $\Gamma_{\th_0}^\pm$  ( or $\Gamma_{1}^\pm$) is $H^1$ across $\Gamma_{\th_0}^\pm$  (or $\Gamma_{1}^\pm$) without any additional compatibility condition. Therefore, back to \eqref{3-22}, applying   the method of reflection  and \cite[Theorems 8.8 and 8.12]{GT98} and combining \eqref{3-23} and \eqref{H2-estimate-vm-final} yield
 \begin{eqnarray}\label{H^4-5-1}
\Vert \h\Psi\Vert_{H^3(\Omega)}
\leq C\left(\|\e F_3\|_{H^1(\m)}+\|\e F_4\|_{H^1(\m)}+\|F_5\|_{H^1(\overline D)}\right)
\end{eqnarray}
for the constant $C>0 $ depending only on $(b_0, \A_0, U_{0},S_{0},E_{0},\th_0,r_0,r_1,\epsilon_0)$.
\par In order to get a priori $H^3$
estimate of $\psi$, a
similar condition is required.  For each $ r\in(r_0,r_1) $, the function $\mf_3(r,\cdot) $ given in \eqref{3-31} is defined on $ D $.
  As a function of $(\th,z) $, it is easy to check that  the extension of $\mf_3$ given by even reflection about $\Gamma_{\th_0}^\pm$  ( or $\Gamma_{1}^\pm$) is $H^1$ across $\Gamma_{\th_0}^\pm$  (or $\Gamma_{1}^\pm$) without any additional compatibility condition. Therefore, the similar argument as in Step 2 yields that the following $H^3$  estimate of  $\psi$:
\begin{eqnarray}\label{H^4-5}
\Vert\psi\Vert_{H^3(\Omega)}
\leq C\left(\|\e F_3\|_{H^2(\m)}+\|\e F_4\|_{H^1(\m)}+\|F_5\|_{H^2(\overline D)}\right),
\end{eqnarray}
 where the constant $C>0 $ depends only on $(b_0, \A_0, U_{0},S_{0},E_{0},\th_0,r_0,r_1,\epsilon_0)$.

\par { \bf Step 3. A priori $H^4$-estimate.}
\par Firstly, a simple computation shows
\begin{equation*}
\begin{cases}
\p_{\th}(-\x a_2\p_r\psi+\e F_4)=0, \ \ {\rm{on}} \ \ \Gamma_{\th_0}^\pm, \\
\p_{z}(-\x a_2\p_r\psi+\e F_4)=0, \ \ {\rm{on}} \ \ \Gamma_{1}^\pm.\\
\end{cases}
\end{equation*}
Then  the extension of $-\x a_2\p_r\psi+\e F_4$ given by even reflection about $\Gamma_{\th_0}^\pm$  ( or $\Gamma_{1}^\pm$) is $H^2$ across $\Gamma_{\th_0}^\pm$  (or $\Gamma_{1}^\pm$). Therefore, a priori $H^4$-estimate of $\h\Psi$ can be obtained by applying \cite[Theorems 8.8 and 8.12]{GT98} and  the reflection argument. Furthermore, with the aid of \eqref{H^4-5}, one has
\begin{eqnarray}\label{H^4-5-1-1-z}
\Vert\h\Psi\Vert_{H^4(\Omega)}
\leq C\left(\|\e F_3\|_{H^2(\m)}+\|\e F_4\|_{H^2(\m)}+\| F_5\|_{H^2(\overline D)}\right).
\end{eqnarray}
\par Next, for each $ r\in(r_0,r_1) $,  it follows from the expression of $\mf_3 $ together with direct computations that
\begin{equation*}
\begin{cases}
\p_{\th}\mf_3=0, \ \ {\rm{on}} \ \ \Gamma_{\th_0}^\pm, \\
\p_{z}\mf_3=0, \ \ {\rm{on}} \ \ \Gamma_{1}^\pm.\\
\end{cases}
\end{equation*}
 Therefore, as a function of $(\th,z) $, the extension of $\mf_3$ given by even reflection about $\Gamma_{\th_0}^\pm$  ( or $\Gamma_{1}^\pm$) is $H^2$ across $\Gamma_{\th_0}^\pm$  (or $\Gamma_{1}^\pm$). Adapting the argument in Step 2 and using the estimate \eqref{H^4-5-1-1-z} yield
\begin{eqnarray}\label{H^4-5-1-1}
\Vert\psi\Vert_{H^4(\Omega)}
\leq C\left(\|\e F_3\|_{H^3(\m)}+\|\e F_4\|_{H^2(\m)}+\|F_5\|_{H^3(\overline D)}\right).
\end{eqnarray}
The above constant $ C>0 $   depends only on $(b_0, \A_0, U_{0},S_{0},E_{0},\th_0,r_0,r_1,\epsilon_0)$. The proof of Lemma \ref{pro5} is completed.
\end{proof}
\subsection{The well-posedness of the linearized  problem}\noindent
 \par With the aid of  Lemmas \ref{pro4} and \ref{pro5}, the following well-posedness of the linearized  problem \eqref{3-6} can be obtained.
\begin{proposition}\label{pro6}

  Fix  $\epsilon_0\in(0,R_0)$,   let  $\bar r_1^\ast\in(r_0,r_0+R_0-\epsilon_0]$ and $ \delta_2^\ast\in(0, {\delta_1^\ast}] $ be from Lemma \ref{pro4} and  Lemma \ref{pro5}, respectively.    For  $r_1\in(r_0,\bar r_1^\ast)$ and $\delta^\ast\leq \delta_2^\ast$, let the iteration set $\ma_{\delta^\ast, r_1} $ be given by \eqref{3-4}. Then for each $ (\e\psi,\e\Psi)\in \ma_{\delta^\ast, r_1} $, the  linear boundary value problem \eqref{3-6}    has a unique solution  $ (\psi,\h\Psi)\in \left(H^4(\m)\right)^2$ satisfying
\begin{eqnarray}\label{H^4-1}
&&\Vert\psi\Vert_{H^4(\Omega)}
\leq C\left(\|\e F_3\|_{H^3(\m)}+\|\e F_4\|_{H^2(\m)}+\|F_5\|_{H^3(\overline D)}\right),\\\label{H^4-2}
&&\Vert\h\Psi\Vert_{H^4(\Omega)}
\leq C\left(\|\e F_3\|_{H^2(\m)}+\|\e F_4\|_{H^2(\m)}+\|F_5\|_{H^2(\overline D)}\right),
\end{eqnarray}
where the constant $C>0$ depends only on $(b_0,   \A_0, U_{0},S_{0},E_{0},\th_0,r_0,r_1,\epsilon_0)$.
 Furthermore, the solution $(\psi,{\Psi})$ satisfies the compatibility conditions
\begin{equation}\label{H^4-1-c}
\begin{cases}
\p_{\th}^k\psi=\p_{\th}^k\h\Psi=0, \ \ {\rm{on}} \ \Gamma_{\th_0}^\pm,\  \  {\rm{for}} \ \ k=1,3,\\
\p_{z}^k\psi=\p_{z}^k\h\Psi=0, \ \ {\rm{on}} \ \Gamma_{1}^\pm,\ \ \  {\rm{for}} \ \ k=1,3.
\end{cases}
\end{equation}
\end{proposition}

 The
 proof is based on the  Galerkin approximations and a limiting argument.
We  divide  three steps  to prove Proposition \ref{pro6}.

\par  { \bf Step 1. Approximation of \eqref{3-6} by a problem with smooth coefficients.}
\par Note that $ (\e\psi,\e\Psi)\in  \left(H^4(\m)\right)^2 $ satisfies
\begin{equation*}
\begin{cases}
\p_{\th}^k\e\psi=\p_{\th}^k\e\Psi=0, \ \ {\rm{on}} \ \Gamma_{\th_0}^\pm,\  \  {\rm{for}} \ \ k=1,3,\\
\p_{z}^k\e\psi=\p_{z}^k\e\Psi=0, \ \ {\rm{on}} \ \Gamma_{1}^\pm,\ \ \  {\rm{for}} \ \ k=1,3.
\end{cases}
\end{equation*}
Then one can extend $ (\e\psi,\e\Psi) $ onto $ \Omega^e=(r_0,r_1)\times(-3\th_0,3\th_0)\times(-3,3)$ by
\begin{equation*}
(\e \psi^e,\e \Psi^e)(r,\theta,z)=
\begin{cases}
(\e \psi,\e \Psi)(r,\theta,z), \ \ &  (r,\th,z)\in (r_0,r_1)\times[-\theta_0,\theta_0]\times[-1,1],\\
(\e \psi,\e \Psi)(r,\theta,-2-z), \ \  &  (r,\th,z)\in (r_0,r_1)\times[-\theta_0,\theta_0]\times(-3,-1),\\
(\e \psi,\e \Psi)(r,\theta,2-z), \ \ &  (r,\th,z) \in(r_0,r_1)\times[-\theta_0,\theta_0]\times(1,3),\\
(\e \psi,\e \Psi)(r,-2\theta_0-\theta,z), \ \  &  (r,\th,z)\in (r_0,r_1)\times(-3\theta_0,-\theta_0)\times[-1,1],\\
(\e \psi,\e \Psi)(r,-2\theta_0-\theta,-2-z), \ \  &  (r,\th,z)\in (r_0,r_1)\times(-3\theta_0,-\theta_0)\times(-3,-1),\\
(\e \psi,\e \Psi)(r,-2\theta_0-\theta,2-z), \ \  &  (r,\th,z)\in (r_0,r_1)\times(-3\theta_0,-\theta_0)\times(1,3),\\
(\e \psi,\e \Psi)(r,2\theta_0-\theta,z), \ \  &  (r,\th,z)\in (r_0,r_1)\times(\theta_0,3\theta_0)\times[-1,1],\\
(\e \psi,\e \Psi)(r,2\theta_0-\theta,-2-z), \ \  &  (r,\th,z)\in (r_0,r_1)\times(\theta_0,3\theta_0)\times(-3,-1),\\
(\e \psi,\e \Psi)(r,2\theta_0-\theta,2-z), \ \  &  (r,\th,z)\in (r_0,r_1)\times(\theta_0,3\theta_0)\times(1,3).\\
\end{cases}
\end{equation*}
 It is easy to check that $(\e \psi^e,\e \Psi^e)\in \left(H^4(\m^e)\right)^2$. For $  (r,\th,z)\in \m^e $, we  define  $(\e \psi^\eta,\e\Psi^\eta) $  by the convolution of $(\e \psi^e,\e \Psi^e) $ with a  radially symmetric  standard mollifier $\chi_\eta $ with a compact support in a disk of radius $\eta>0$, i.e,
$$\e \psi^\eta:=\e \psi^e*\chi_\eta ,\ \ \e \Psi^\eta:=\e \Psi^e*\chi_\eta .$$
Then we have
$(\e \psi^\eta,\e\Psi^\eta)\in \left(C^{\infty}(\overline{\Omega})\right)^2$
and  the following compatibility conditions:
\begin{equation*}
\begin{cases}
\p_{\th}^k\e\psi^\eta=\p_{\th}^k\e\Psi^\eta=0, \ \ {\rm{on}} \ \Gamma_{\th_0}^\pm,\  \  {\rm{for}} \ \ k=1,3,\\
\p_{z}^k\e\psi^\eta=\p_{z}^k\e\Psi^\eta=0, \ \ {\rm{on}} \ \Gamma_{1}^\pm,\  \  \ {\rm{for}} \ \ k=1,3.
\end{cases}
\end{equation*}
 Furthermore, as $ \eta\rightarrow0 $, one derives
 \begin{equation*}
 \|\e \psi^{\eta}-\e \psi\|_{H^4(\Omega)}\rightarrow0 \ \ {\rm{and}} \ \ \|\e \Psi^{\eta}-\e \Psi\|_{H^4(\Omega)}\rightarrow0.
 \end{equation*}
\par For  any fixed $\eta>0$, define
 \begin{equation*}
  \begin{aligned}
 &\e A_{ii}^{\eta}(r,\th,z):=A_{ii}(r,\th,z,\n\e\psi^{\eta},\e\Psi^{\eta}),\quad \ \quad (r,\th,z)\in \Omega, \  i=1,2,3,, \\
&\e A_{ij}^{\eta}(r,\th,z):=A_{ij}(r,\th,z,\n\e\psi^{\eta}),\quad \quad \quad \ \ \   (r,\th,z)\in \Omega,\  i\neq j =1,2,3, \\
 &\e F_3^{\eta}(r,\th,z):=F_3 (r,\th,z,\n\e\psi^{\eta},\n\e\Psi^{\eta},\e\Psi^{\eta}), \    (r,\th,z)\in \Omega, \\
 &\e F_4^{\eta}(r,\th,z):=F_4 (r,\th,z,\n\e\psi^{\eta},\e\Psi^{\eta}),  \ \ \qquad    (r,\th,z)\in \Omega.
 \end{aligned}
 \end{equation*}
 It is easy to see that $\e A_{ij}^{\eta}\in C^{\infty}(\overline\m)$ for $ i,j=1,2,3 $ and $\e F_i^{\eta}\in C^2(\overline\m)$ for $ i=3,4 $. Then a  approximation problem with smooth coefficients of \eqref{3-6} is obtained as follows:
  \begin{equation}\label{3-41}
 \begin{cases}
\e L_1^{\eta}(\psi,\h\Psi)= \e A_{11}^n\p_r^2 \psi+2\e A_{12}^n\p_{r\th}^2 \h\psi+
 2 \e A_{13}^{\eta}\p_{rz}^2 \psi
 +  \e A_{22}^{\eta}\p_{\th}^2 \psi+2\e A_{23}^{\eta}\p_{\th z}^2 \psi +\e A_{33}^{\eta}\p_{z}^2 \psi\\
\qquad\qquad\ \ +\x a_1\p_r\psi+\x b_1\p_r\h\Psi+\x b_2(r)\h\Psi=\e F_3^{\eta},\ \ &{\rm{in}}\ \ \Omega,\\
\e L_2(\psi,\h\Psi)= \left(\p_r^2+\frac 1 r\p_r+\frac{1}{r^2}\p_{\th}^2+\p_z^2\right) \h\Psi+\x a_2\p_r \psi-\x b_3 \h\Psi=\e F_4^{\eta},\ \ &{\rm{in}}\ \ \Omega,\\
 \p_r\psi(r_0,\th, z)=F_5(\th, z),\ \psi(r_0,\th, z)=\p_r\h\Psi(r_0,\th, z)=0,\ \ &{\rm{on}}\ \ \Gamma_{en},\\
\h\Psi(r_1,\th,z)=0, \ \ &{\rm{on}}\ \ \Gamma_{ex},\\
\p_\th\psi(r,\pm\th_0,z)=\p_\th\h\Psi(r,\pm\th_0,z)=0, \ \ &{\rm{on}}\ \ \Gamma_{\th_0}^\pm,\\
\p_z\psi(r,\th,\pm 1)=\p_z\h\Psi(r,\th,\pm 1)=0, \ \ &{\rm{on}}\ \ \Gamma_{1}^\pm.\\
 \end{cases}
 \end{equation}
 \par  { \bf Step 2. Galerkin approximations.}
\par  Firstly, For each $i\in\mathbb{Z}^+$, define
\begin{equation}\label{or-1}
\vartheta_0(\theta)=\sqrt{\frac{1}{2\theta_0}}, \ \ \vartheta_i(\theta)=
\sqrt{\frac{1}{\theta_0}}\cos\left(\frac{i\pi}{\theta_0}\theta\right), \ i=1,2, \cdots.
\end{equation}
Then the set $=\{\vartheta_i\}_{i=0}^{\infty}$   is the collection of all eigenfunctions associated to the eigenvalue problem
\begin{equation*}
\begin{cases}
-\vartheta^{\prime\prime}(\theta)=\mu \vartheta(\theta), \\
\vartheta^{\prime}(\pm\theta_0)=0,
\end{cases}
\end{equation*}
and   the corresponding eigenvalues are $\left\{\mu_i:\mu_i=\left(\frac{i\pi}{\theta_0}\right)^2\right\}$. It is easy to see that
 the set $\{\vartheta_i(\th)\}_{i=0}^{\infty}$ forms an orthonormal basis in $L^2(-\th_0,\th_0)$ and an orthogonal basis in $H^1(-\th_0,\th_0)$.
\par Next, for each $j\in\mathbb{Z}^+$, define
\begin{equation}\label{or-2}
\beta_0(z)=\frac{1}{\sqrt{2}},\ \
\beta_j(z)=\cos\left(j{\pi}z\right),\ j=1,2, \cdots.
\end{equation}
 The set $\{\beta_j\}_{j=0}^{\infty}$   is the collection of all eigenfunctions to the eigenvalue problem
\begin{equation*}
\begin{cases}
-\beta^{\prime\prime}(z)=\tau\beta(z), \\
\beta^{\prime}(\pm 1)=0,
\end{cases}
\end{equation*}
and   the corresponding eigenvalues are $\{\tau_j:\tau_j=\left({j\pi}\right)^2\}$.
Then the set $\mathcal{S}:=\{\vartheta_i(\th)\beta_j(z)\}_{i,j=0}^{\infty}$ forms a complete orthonormal basis in $L^2((-\th_0,\th_0))\times (-1,1))$ and an orthogonal basis in $H^1((-\th_0,\th_0)\times(-1,1))$.

\par  For  any fixed $\eta>0$, we consider the linear boundary value problem \eqref{3-41}. For each $m=1,2,\cdots$, let $(\psi_m^\eta, \Psi_m^\eta)$ be of the form
\begin{equation}\label{m}
\psi_m^\eta(r,\theta,z)=\sum_{i,j=0}^{m}\my_{i,j}^\eta(r)\vartheta_i(\th)\beta_j(z),\ \  \h\Psi_m^\eta(r,\theta,z)=\sum_{i,j=0}^{m}\ml_{i,j}^\eta(r)\vartheta_i(\th)\beta_j(z), \ \ (r,\th,z)\in \m.
\end{equation}
 Then we find a solution $(\psi_m^\eta,\h\Psi_m^\eta)$ satisfying
\begin{equation}\label{Leqs4}
\begin{cases}
\begin{aligned}
&\iint_D \e{L}^{\eta}_1(\psi_m^\eta, \h\Psi_m^\eta)\vartheta_{k_1}(\th)\beta_{k_2}(z)\de \th \de z= \iint_D\e F^{\eta}_3\vartheta_{k_1}(\th)\beta_{k_2}(z)\de \th \de z,\quad  \ \ & r\in (r_0, r_1),\\
&\iint_D \e{L}_2(\psi_m^\eta, \h\Psi_m^\eta)\vartheta_{k_1}(\th)\beta_{k_2}(z)\de \th \de z= \iint_D\e F^{\eta}_4\vartheta_{k_1}(\th)\beta_{k_2}(z)\de \th \de z,\quad\ \   & r\in (r_0, r_1),\\
&\psi_m^\eta=\partial_r\h\Psi_m^\eta=0,  \qquad\qquad\qquad\qquad\qquad\qquad\qquad \qquad\quad\ \ \ \ \quad &\hbox{on} \ \ \Gamma_{en},\\
&\partial_r\psi_m^\eta=\sum_{i,j=0}^{m}\vartheta_i(\th)\beta_j(z)\iint_D F_5\vartheta_i(\th)\beta_j(z)\de \th \de z,  \ &\hbox{on} \ \ \Gamma_{en},\\
&\h\Psi_m^\eta=0,\ \  &\hbox{on}\ \ \Gamma_{ex},\\
&\p_\th\psi_m^\eta=\p_\th\h\Psi_m^\eta=0,     \quad &\hbox{on}\ \ \Gamma_{\th_0}^\pm,\\
&\p_z\psi_m^\eta=\p_z\h\Psi_m^\eta=0,  \quad &\hbox{on}\ \ \Gamma_1^\pm,
\end{aligned}
\end{cases}
\end{equation}
for  all $k_1=0,1,...,m$ and $k_2=0,1,...,m$.  If $\psi_m^\eta  $ and $\h\Psi_m^\eta$ are smooth functions and solve \eqref{Leqs4},  one obtains
\begin{equation*}
\begin{split}
&-\iiint_{\Omega}\e{L}^\eta_1(\psi_m^\eta,\h{\Psi}_m^\eta)Z(r)\partial_r\psi_m^\eta
\mathrm{d}r\mathrm{d}\theta \de z
-\iiint_{\Omega}\e{L}_2(\psi_m^\eta,\h{\Psi}_m^\eta)\h\Psi_m^\eta\mathrm{d}r\mathrm{d}\theta \de z\\
&=-\int_{r_0}^{r_1}Z(r)\iint_D\sum_{i,j=0}^{m}(\my_{i,j}^\eta)^{\prime}(r)
\e{L}^\eta_1(\psi_m^\eta, \h\Psi_m^\eta)\vartheta_{i}(\th)\beta_{j}(z)\de \th \de z\de r\\
&\quad -\int_{r_0}^{r_1}\iint_D\sum_{i,j=0}^{m}\ml_j^\eta(r)
\e{L}_2(\psi_m^\eta, \h\Psi_m^\eta)\vartheta_{i}(\th)\beta_{j}(z)\de \th \de z\de r\\
&=-\iiint_{\Omega}Z(r)\e F^\eta_3\partial_r\psi_m^\eta\mathrm{d}r\de \th \de z
-\iiint_{\Omega}\e F^\eta_4\h\Psi_m^\eta\mathrm{d}r \de \th \de z,
\end{split}
\end{equation*}
 where   the function $ Z $  is constructed in Lemma \ref{pro4}. Then it holds that
\begin{equation}\label{H1}
\Vert(\psi_m^\eta,\h\Psi_m^\eta)\Vert_{H^1(\Omega)}
\leq C\left(\Vert\e F_3^\eta\Vert_{L^2(\Omega)}+\Vert\e F_4^\eta\Vert_{L^2(\Omega)}
+\Vert F_5\Vert_{L^2(\overline D)}\right),
\end{equation}
where the constant $ C>0 $  is independent of $\eta$ and $m$.
\par Next, by the orthonormality of the set $\mathcal{S}$ in $L^2((-\th_0,\th_0)\times(-1,1))$, the problem \eqref{Leqs4} can be rewritten
\begin{equation}\label{3-42}
\begin{cases}
\begin{aligned}
&\sum_{i,j=0}^{m}\bigg(\me_{1,i,j}^{\eta,k_1,k_2}(r)(\my^\eta_{i,j})''(r)
+\me_{2,i,j}^{\eta,k_1,k_2}(r)(\my^\eta_{i,j})'(r)+\me_{3,i,j}^{\eta,k_1,k_2}(r)
(\my^\eta_{i,j})'(r)\\
& -\me_{4,i,j}^{\eta,k_1,k_2}(r)\my_{i,j}^\eta(r)
+\me_{5,i,j}^{\eta,k_1,k_2}(r)\my_{i,j}^\eta(r)-\me_{6,i,j}^{\eta,k_1,k_2}(r)
\my_{i,j}^\eta(r)\bigg)
 +\x a_1(r)(\my_{k_1,k_2}^\eta)'(r)\\
& +\x b_1(r)(\ml_{k_1,k_2}^\eta)'(r)
+\x b_2(r)\ml_{k_1,k_2}^\eta(r)=F^{\eta,k_1,k_2}_3(r), \ \ r\in (r_0, r_1),\\
&(\ml^\eta_{k_1,k_2})''(r)+\frac{1}{r}(\ml^\eta_{k_1,k_2})'(r)-\bigg(
\frac{1}{r^2}\left(\frac{k_1\pi}{\theta_0}\right)^2+(k_2\pi)^2\bigg)
\ml_{k_1,k_2}^\eta(r)\\
& +\x a_2(r)(\my^\eta_{k_1,k_2})'
-\x b_3(r)\ml_{k_1,k_2}^\eta(r)=F^{\eta,k_1,k_2}_4(r),\ \ \ \ \ \  r\in (r_0, r_1),\\
&\my_{k_1,k_2}^\eta(r_0)=(\ml_{k_1,k_2}^\eta)'(r_0)=0,\quad
(\my_{k_1,k_2}^\eta)'(r_0)=\iint_D F_5\vartheta_{k_1}(\th)\beta_{k_2}(z)\de \th \de z,\\
&\ml_{k_1,k_2}^\eta(r_1)=0,
\end{aligned}
\end{cases}
\end{equation}
where
\begin{equation*}
\begin{aligned}
 &\me_{1,i,j}^{\eta,k_1,k_2}(r)=\iint_D\e A_{11}^\eta \vartheta_i(\th)\beta_j(z)\vartheta_{k_1}(\th)\beta_{k_2}(z)\de \th \de z,\ \
 \me_{2,i,j}^{\eta,k_1,k_2}(r)=2\iint_D\e A_{12}^\eta \vartheta_i'(\th)\beta_j(z)\vartheta_{k_1}(\th)\beta_{k_2}(z)\de \th \de z,\\
 &\me_{3,i,j}^{\eta,k_1,k_2}(r)=2\iint_D\e A_{13}^\eta \vartheta_i(\th)\beta_j'(z)\vartheta_{k_1}(\th)\beta_{k_2}(z)\de \th \de z,\
 \me_{4,i,j}^{\eta,k_1,k_2}(r)=\left(\frac{i\pi}{\theta_0}\right)^2\iint_D\e A_{22}^\eta \vartheta_i(\th)\beta_j(z)\vartheta_{k_1}(\th)\beta_{k_2}(z)\de \th \de z,\\
 &\me_{5,i,j}^{\eta,k_1,k_2}(r)=2\iint_D\e A_{23}^\eta \vartheta_i'(\th)\beta_j'(z)\vartheta_{k_1}(\th)\beta_{k_2}(z)\de \th \de z,\
 \me_{6,i,j}^{\eta,k_1,k_2}(r)=(j\pi)^2\iint_D\e A_{33}^\eta \vartheta_i(\th)\beta_j(z)\vartheta_{k_1}(\th)\beta_{k_2}(z)\de \th \de z,\\
& F^{\eta,k_1,k_2}_3(r)=\iint_D\e F^\eta_3\vartheta_{k_1}(\th)\beta_{k_2}(z)\de \th \de z, \qquad \qquad\quad\ \ F^{\eta,k_1,k_2}_4(r)=\iint_D\e F^\eta_4\vartheta_{k_1}(\th)\beta_{k_2}(z)\de \th \de z.
 \end{aligned}
\end{equation*}
\par For each $m\in\mathbb{N}$, set
\begin{equation*}
\begin{split}
\mathbf{X}_1&:=(\my_{0,0}^\eta,...,\my_{m,m}^\eta), \ \ \mathbf{X}_2:=((\my_{0,0}^\eta)',...,(\my_{m,m}^\eta)'),\\
\mathbf{X}_3&:=(\ml_{0,0}^\eta,...,\ml_{m,m}^\eta), \ \ \mathbf{X}_4:=((\ml_{0,0}^\eta)',...,(\ml_{m,m}^\eta)'),\\
\mathbf{X}&:=(\mathbf{X}_1,\mathbf{X}_2,\mathbf{X}_3, \mathbf{X}_4)^{T},
\end{split}
\end{equation*}
and  define a projection mapping $\Pi_{i}\ (i=1,2)$ by
\begin{equation*}
\Pi_1(\mathbf{X})=(\mathbf{X}_1,\mathbf{X}_2,\mathbf{0},\mathbf{X}_4)^{T} \quad \hbox{and} \quad
\Pi_2(\mathbf{X})=(\mathbf{0},\mathbf{0},\mathbf{X}_3,\mathbf{0})^{T}.
\end{equation*}
Then \eqref{3-42} can be reduced to a first order ODE system:
\begin{equation}\label{3-43}
\begin{cases}
\begin{aligned}
&\mathbf{X}^{\prime}=\mathbb{A}\mathbf{X}+\mathbf{H},\\
&\Pi_1(\mathbf{X})(r_0)=\bigg(\mathbf{0},\iint_D F_5\vartheta_0(\th)\beta_0(z)\de \th \de z,\cdots,\iint_D F_5\vartheta_{m}(\th)\beta_{m}(z)\de \th \de z,\mathbf{0},\mathbf{0}\bigg)^{T}:=\mathbf{P}_0,\\
 &\Pi_2(\mathbf{X})(r_1)=\mathbf{0},
 \end{aligned}
\end{cases}
\end{equation}
for smooth mappings $\mathbb{A}:(r_0,r_1)\rightarrow \mathbb{R}^{4(m+1)^2\times 4(m+1)^2}$ and $\mathbf{H}:(r_0,r_1)\rightarrow \mathbb{R}^{4(m+1)^2}$ with respect to the coefficients and right terms of \eqref{3-42}.
Therefore, $\mathbf{X}$ can be given by solving the following integral equations:
\begin{equation}\label{X}
\begin{cases}
\begin{aligned}
&\Pi_1(\mathbf{X})(r)-\int_{r_0}^r\Pi_1\mathbb{A}(\mathbf{X})(s)\mathrm{d}s
=\mathbf{P}_0+\int_{r_0}^r\Pi_1\mathbf{H}(s)\mathrm{d}s,\\
&\Pi_2(\mathbf{X})(r)-\int_{r_1}^r\Pi_2\mathbb{A}(\mathbf{X})(s)\mathrm{d}s
=\int_{r_1}^r\Pi_2\mathbf{H}(s)\mathrm{d}s.
\end{aligned}
\end{cases}
\end{equation}
\par We define a linear operator $\mathfrak{E}:C^1([r_0,r_1];\mathbb{R}^{4(m+1)^2})\rightarrow C^1([r_0,r_1];\mathbb{R}^{4(m+1)^2})$ by
\begin{equation*}
\mathfrak{E}\mathbf{X}(r):=\Pi_1\int_{r_0}^r\mathbb{A}(\mathbf{X})(s)\mathrm{d}s
+\Pi_2\int_{r_1}^r\mathbb{A}(\mathbf{X})(s)\mathrm{d}s.
\end{equation*}
 Note that $\mathbb{A}:(r_0,r_1)\rightarrow \mathbb{R}^{4(m+1)^2\times4(m+1)^2}$ is smooth. Then  there exists a constant
$C_0 > 0$ such that
\begin{equation*}
\|\mathfrak{E}\mathbf{X}\|_{C^2([r_0, r_1])}\leq C_0\|\mathbf{X}\|_{C^1([r_0, r_1])}, \ \ \rm{{for \ all}} \  \mathbf{X}\in \mathfrak{E}.
\end{equation*}
Thus the Arzel\'{a}-Ascoli theorem shows that $\mathfrak{E}$ is compact.  Furthermore, one  can rewrite   \eqref{X} as
\begin{equation}\label{R}
(\mathbf{I}-\mathfrak{E})\mathbf{X}(r)=\mathbf{P}_0+\Pi_1\int_{r_0}^r
\mathbf{H}(s)\mathrm{d}s
+\Pi_2\int_{r_1}^r\mathbf{H}(s)\mathrm{d}s.
\end{equation}
Suppose that $(\mathbf{I}-\mathfrak{E})\mathbf{X}_*=0$. Then $ \mathbf{X}_* $ solves \eqref{3-42}   with $(\mathbf{P}_0,\mathbf{H})=(\mathbf{0},\mathbf{0})$ and  the  corresponding $(\psi_m^\eta,\h\Psi_m^\eta)$ solves \eqref{Leqs4} with $\e F^\eta_3=\e F^\eta_4=F_5=0$. The estimate \eqref{H1} implies $\psi_m^\eta=\h\Psi_m^\eta=0$,  from which one gets $\mathbf{X}_*=\mathbf{0}$ on $[r_0,r_1]$. Hence the  Fredholm alternative theorem  yields  that  \eqref{R} has a unique solution $\mathbf{X}\in C^1([r_0,r_1];\mathbb{R}^{4(m+1)})$.
\par Furthermore, the bootstrap argument for \eqref{X} gives the smoothness of the solution $\mathbf{X}$ on $[r_0,r_1]$. Therefore, $\psi_m^\eta  $ and $\h\Psi_m^\eta$ are smooth functions and solve \eqref{Leqs4}. Then
the following higher order derivatives estimate can be derived similarly as in Lemma \ref{pro5}:
\begin{eqnarray}\label{H^4-1-1-a}
&&\Vert\psi_m^\eta\Vert_{H^4(\Omega)}
\leq C\left(\|\e F_3^\eta\|_{H^3(\m)}+\|\e F_4^\eta\|_{H^2(\m)}+\| F_5\|_{H^3(\overline D)}\right),\\\label{H^4-2-2-b}
&&\Vert\h\Psi_m^\eta\Vert_{H^4(\Omega)}
\leq C\left(\|\e F_3^\eta\|_{H^2(\m)}+\|\e F_4^\eta\|_{H^2(\m)}+\|F_5\|_{H^2(\overline D)}\right).
\end{eqnarray}
The above estimate constant $C>0$  is independent of $\eta$ and $m$.
\par  { \bf Step 3. Taking the limit.}
\par For  any fixed $\eta>0$, let  $ (\psi_m^\eta,\h\Psi_m^\eta) $ given in the form \eqref{m} be the solution to the problem \eqref{Leqs4}. It follows from \eqref{H^4-1-1-a} and \eqref{H^4-2-2-b} that the sequence $\{(\psi_m^\eta,\h\Psi_m^\eta)\}_{m=1}^{\infty}$ is bounded in $[H^4(\Omega)]^2$. Owing to $ \vartheta_i^{\prime}(\pm\theta_0)=\vartheta_i^{\prime\prime\prime}(\pm\theta_0)=0 $ and $\beta_j^{\prime}(\pm 1)=\beta_j^{\prime\prime\prime}(\pm 1)=0$ for $i,j=0,1,...$, thus each $(\psi_m^\eta,\h\Psi_m^\eta)$ satisfies the compatibility conditions $\p_{\th}^k\psi_m^\eta=\p_{\th}^k\h\Psi_m^\eta=0$ on $\Gamma_{\th_0}^\pm $ and $\p_{z}^k\psi_m^\eta=\p_{z}^k\h\Psi_m^\eta=0$ on $\Gamma_1^\pm $ for $k=1,3$. By  the weak compactness property of $H^4(\Omega)$,  there exists a subsequence $\{(\psi_{m_k}^\eta,\h\Psi_{m_k}^\eta)\}_{k=1}^{\infty}$ with  $ m_k\rightarrow \infty$ as $ k\rightarrow \infty$ converging weakly to  a limit $(\psi_*^\eta,\h\Psi_*^\eta)$  in $H^4(\Omega)$. Furthermore,  the Morrey inequality implies that the sequence $\{(\psi_m^\eta,\h\Psi_m^\eta)\}_{m\in\mathbb{N}}$ is bounded in $[C^{2,\frac{1}{2}}(\overline{\Omega})]^2$. Then it follows the Arzel$\mathrm{\grave{a}}$-Ascoli theorem that the subsequence $(\psi_{m_k}^\eta,\h\Psi_{m_k}^\eta)$   converges to $(\psi_*^\eta,\h\Psi_*^\eta) $  in $ [C^{2,\frac{1}{4}}(\overline{\Omega})]^2 $.
  This shows that $(\psi_*^\eta,\h\Psi_*^\eta)$ is a classical solution to \eqref{3-41}.
\par Next, it follows from
 \eqref{H^4-1-1-a}-\eqref{H^4-2-2-b} and the $H^4$ convergence of $\{(\psi_{m_k}^\eta,\h\Psi_{m_k}^\eta)\}$ that one obtains
\begin{eqnarray}\label{*H^4_1}
&&\Vert\psi_*^\eta \Vert_{H^4(\Omega)}
\leq C\left(\Vert\e F_3^\eta\Vert_{H^3(\Omega)}+\Vert\e F_4^\eta\Vert_{H^2(\Omega)}
+\Vert F_5\Vert_{H^3(\overline D)}\right),\\
&&\Vert\h\Psi_*^\eta\Vert_{H^4(\Omega)}\label{*H^4_2}
\leq C\left(\Vert\e F_3^\eta\Vert_{H^2(\Omega)}+\Vert\e F_4^\eta\Vert_{H^2(\Omega)}
+\Vert F_5\Vert_{H^2(\overline D)}\right),
\end{eqnarray}
and \begin{equation}\label{*c}
\begin{cases}
\p_{\th}^{k}\psi^\ast=\p_{\th}^k\h\Psi^\ast=0,
\ \ {\rm{on}} \ \Gamma_{\th_0}^\pm,\  \  &{\rm{for}} \ \ k=1,3,\\
\p_{z}^k\psi^\ast=\p_{z}^k\h\Psi^\ast=0,
\ \ {\rm{on}} \ \Gamma_{1}^\pm,\  \  &{\rm{for}} \ \ k=1,3.
\end{cases}
\end{equation}
The estimate  constant $ C>0 $ in \eqref{*H^4_1}-\eqref{*H^4_2} is independent of $\eta$.
\par  For   any fixed $\eta>0$, let $(\psi_\ast^\eta,\h\Psi_\ast^\eta)\in[H^4(\Omega)]^2$ be the solution to  \eqref{3-41}. It follows from  \eqref{*H^4_1}-\eqref{*c} and the Morrey inequality that the sequence $\{(\psi_\ast^\eta,\h\Psi_\ast^\eta)\}_{n\in\mathbb{N}}$ is bounded in $[C^{2,\frac{1}{2}}(\overline{\Omega})]^2$. By adjusting the limiting argument in Step 2, we extract a subsequence $\{(\psi_\ast^{\eta_k},\h\Psi_\ast^{\eta_k})\}_{k=1}^{\infty}$ with $ \eta_k\rightarrow 0$ as $ k\rightarrow \infty$ such that its limit $(\psi_\ast,\h\Psi_\ast)\in[H^4(\Omega)\cap C^{2,\frac{1}{4}}(\overline{\Omega})]^2$ is a classical solution to  \eqref{3-6}  and satisfies
 \eqref{H^4-1}-\eqref{H^4-1-c}.
This completes the proof of Proposition \ref{pro6}.
\subsection{Proof of Theorem 2.4}\noindent
\par We are now ready to prove Theorem \ref{th1}.  Fix $\epsilon_0\in(0,R_0)$ and $ r_1\in (r_0,\bar r_1^\ast) $ for  $\bar r_1^\ast$ from Lemma \ref{pro4}.  Assume that
\begin{equation}\label{3-44}
 \delta^\ast\leq \delta_2^\ast,
 \end{equation}
 where $ \delta_2^\ast $ is given by Lemma \ref{pro5}.
For any fixed  $ (\e\psi,\e\Psi)\in \ma_{\delta^\ast, r_1} $,  one can define a mapping
\begin{equation*}
\mt(\e\psi,\e\Psi)=(\psi,\Psi), \quad{\rm{ for \ each}} \ \ (\e\psi,\e\Psi)\in \ma_{\delta^\ast, r_1}.
\end{equation*}
 Proposition \ref{pro6} shows that there exists a unique solution
$ (\psi,\h\Psi)\in \left(H^4(\m)\right)^2$  to the problem  \eqref{3-6} with the estimate
\begin{equation}\label{3-45}
\begin{aligned}
 \Vert(\psi,\h\Psi)\Vert_{H^4(\Omega)}
\leq C\left(\|\e F_3\|_{H^3(\m)}+\|\e F_4\|_{H^2(\m)}+\|F_5\|_{H^3(\overline D)}\right).
 \end{aligned}
 \end{equation}
   Here the constant  $C>0 $  depends only on $(b_0,   \A_0, U_{0},S_{0},E_{0},\th_0,r_0,r_1,\epsilon_0)$.
\par Use the abbreviation
    \begin{equation*}
    \sigma_{p}=\sigma(b^\ast,U_{ 1,en}^\ast,E_{ en}^\ast,\Phi_{ ex}^\ast),
    \end{equation*}
where $\sigma(b^\ast,U_{ 1,en}^\ast,E_{ en}^\ast,\Phi_{ ex}^\ast)$  is  defined in \eqref{1-9-a}.  Then   \eqref{3-45}, together with  \eqref{3-45-e}, yields
  \begin{equation}\label{3-45-z-e-z}
 \Vert(\psi,\Psi)\Vert_{H^4(\Omega)}
\leq \mc_3^\ast\left((\delta^\ast)^2+\sigma_p\right)
 \end{equation}
 for the constant $\mc_3^\ast>0 $  depending only on $(b_0,   \A_0, U_{0},S_{0},E_{0},\th_0,r_0,r_1,\epsilon_0)$.
 Set
 \begin{equation}\label{3-46}
  \sigma_1^\ast=\frac{1}{16((\mc_3^\ast)^2+\mc_3^\ast)}  \ \ {\rm{and}} \ \ \delta^\ast=4\mc_3^\ast\sigma_p.
\end{equation}
Then if \begin{equation}\label{4-4}
\sigma_p\leq \sigma_1^\ast,
\end{equation}
one can follow from \eqref{3-45-z-e-z} to obtain that
\begin{equation*}
\begin{aligned}
\Vert(\psi,\Psi)\Vert_{H^4(\Omega)}
\leq \mc_3^\ast\bigg((\delta^\ast)^2+\sigma_p\bigg)
\leq \frac12\delta^\ast.
\end{aligned}
\end{equation*}
This implies that the mapping $\mt$ maps $\ma_{\delta^\ast, r_1}$ into itself.
\par It remains to show that the mapping $\mt$ is contractive in a low order norm. Let $ (\e\psi^{(i)},\e\Psi^{(i)})\in \ma_{\delta^\ast, r_1} $, $ i=1,2 $, one has $\mt(\e\psi^{(i)},\e\Psi^{(i)})=(\psi^{(i)},\Psi^{(i)}).$ Set
   \begin{equation*}
  (\e Y_1,\e Y_2) =(\e\psi^{(1)},\e\Psi^{(1)})-(\e\psi^{(2)},\e\Psi^{(2)}), \quad  ( Y_1, Y_2) =(\psi^{(1)},\Psi^{(1)})-(\psi^{(2)},\Psi^{(2)}).
  \end{equation*}
  Then $   ( Y_1, Y_2) $ satisfies
  \begin{equation}\label{3-47}
 \begin{cases}
 A_{11}(r,\th,z,\n\e\psi^{(1)},\e\Psi^{(1)})\p_r^2Y_1
 +2A_{12}(r,\th,z,\n\e\psi^{(1)})\p_{r\th}^2Y_1\\
 +
 2A_{13}(r,\th,z,\n\e\psi^{(1)})\p_{rz}^2Y_1
 +  A_{22}(r,\th,z,\n\e\psi^{(1)},\e\Psi^{(1)})\p_{\th}^2Y_1\\
 +2A_{23}(r,\th,z,\n\e\psi^{(1)})\p_{\th z}^2Y_1
 +A_{33}(r,\th,z,\n\e\psi^{(1)},\e\Psi^{(1)})\p_{z}^2Y_1\\
 +\x a_1(r)\p_rY_1+\x b_1(r)\p_rY_2
 +\x b_2(r)Y_2\\
 =\mf_1(r,\th,z,
 \n\e\psi^{(1)},\n\e\Psi^{(1)},\e\Psi^{(1)},
 \n\e\psi^{(2)},\n\e\Psi^{(2)},\e\Psi^{(2)}), \ \ &{\rm{in}}\ \ \Omega,\\
 \left(\p_r^2+\frac 1 r\p_r+\frac{1}{r^2}\p_{\th}^2+\p_z^2\right)Y_2+\x a_2(r)\p_rY_1-\x b_3(r)Y_2\\
 =\mf_2(r,\th,z,
 \n\e\psi^{(1)},\e\Psi^{(1)},
 \n\e\psi^{(2)},\e\Psi^{(2)})
 , \ \ &{\rm{in}}\ \ \Omega,\\
 Y_1(r_0,\th,z)=0, \ \p_rY_1(r_0,\th,z)= \p_rY_2(r_0,\th,z)=0,\ \ &{\rm{on}}\ \ \Gamma_{en},\\
Y_2(r_1,\th,z)=0, \ \ &{\rm{on}}\ \ \Gamma_{ex},\\
\p_\th Y_1(r,\pm\th_0,z)=\p_\th Y_2(r,\pm\th_0,z)=0, \ \ &{\rm{on}}\ \ \Gamma_{\th_0}^\pm,\\
\p_zY_1(r,\th,\pm 1)=\p_zY_2(r,\th,\pm 1)=0, \ \ &{\rm{on}}\ \ \Gamma_{1}^\pm,\\
 \end{cases}
 \end{equation}
 where
 \begin{equation*}
 \begin{aligned}
 &\mf_1(r,\th,z,
 \n\e\psi^{(1)},\n\e\Psi^{(1)},\e\Psi^{(1)},
 \n\e\psi^{(2)},\n\e\Psi^{(2)},\e\Psi^{(2)})\\
 &=F_1(r,\th,z,
 \n\e\psi^{(1)},\n\e\Psi^{(1)},\e\Psi^{(1)})-F_1(r,\th,z,
 \n\e\psi^{(2)},\n\e\Psi^{(2)},\e\Psi^{(2)})\\
 &\quad - \left(A_{11}(r,\th,z,\n\e\psi^{(1)},\e\Psi^{(1)})
 -A_{11}(r,\th,z,\n\e\psi^{(2)},\e\Psi^{(2)})\right)\p_r^2\psi^{(2)}\\
&\quad -2\left(A_{12}(r,\th,z,\n\e\psi^{(1)})-A_{12}(r,\th,z,\n\e\psi^{(2)})\right)
\p_{r\th}^2\psi^{(2)}\\
 &\quad -
 2\left(A_{13}(r,\th,z,\n\e\psi^{(1)})-A_{13}(r,\th,z,\n\e\psi^{(2)})\right
 )\p_{rz}^2\psi^{(2)}\\
 &\quad- \left(A_{22}(r,\th,z,\n\e\psi^{(1)},\e\Psi^{(1)})
 -A_{22}(r,\th,z,\n\e\psi^{(2)},\e\Psi^{(2)})\right)
 \p_{\th}^2\psi^{(2)}\\
 &\quad-2\left(A_{23}(r,\th,z,\n\e\psi^{(1)})
 -A_{23}(r,\th,z,\n\e\psi^{(2)})\right)\p_{\th z}^2\psi^{(2)}\\
 &\quad-\left(A_{33}(r,\th,z,\n\e\psi^{(1)},\e\Psi^{(1)})
 -A_{33}(r,\th,z,\n\e\psi^{(2)},\e\Psi^{(2)})\right)
 \p_{z}^2\psi^{(2)},\\
 &\mf_2(r,\th,z,
 \n\e\psi^{(1)},\e\Psi^{(1)},
 \n\e\psi^{(2)},\e\Psi^{(2)})\\
 &=F_2(r,\th,z,
 \n\e\psi^{(1)},\e\Psi^{(1)})-F_2(r,\th,z,
 \n\e\psi^{(2)},\e\Psi^{(2)}).\\
 \end{aligned}
 \end{equation*}
 Since $(\e\psi^{(i)},\e\Psi^{(i)})$, $(\psi^{(i)},\Psi^{(i)}) \in \ma_{\delta^\ast, r_1}$, for $i= 1, 2$, the $H^1$ estimate in Lemma \ref{pro4} shows that
 \begin{equation}\label{3-48}
 \begin{aligned}
\|(Y_1,Y_2)\|_{H^1(\m)}
&\leq C\bigg(
 \| \mf_1(\cdot, \n\e\psi^{(1)},
 \n\e\Psi^{(1)},\e\Psi^{(1)},
 \n\e\psi^{(2)},\n\e\Psi^{(2)},\e\Psi^{(2)})\|_{L^2(\m)}\\
 &\qquad+\|\mf_2(\cdot,
 \n\e\psi^{(1)},\e\Psi^{(1)},
 \n\e\psi^{(2)},\e\Psi^{(2)})\|_{L^2(\m)}\bigg)\\
 &\leq \mc_4^\ast\delta^\ast\|(\e Y_1,\e Y_2)\|_{H^1(\m)}
  \leq 4\mc_3^\ast\mc_4^\ast\sigma_p,
 \end{aligned}
\end{equation}
where the constant $\mc_4^\ast>0 $  depends only on $(b_0,   \A_0, U_{0},S_{0},E_{0},\th_0,r_0,r_1,\epsilon_0)$. Set
 \begin{equation}\label{3-49}
\sigma_2^\ast=\frac{1}{8\mc_3^\ast\mc_4^\ast}.
\end{equation}
Then if
$$ \sigma_p\leq \sigma_2^\ast,$$ $\mt$ is a contractive mapping in $H^1(\m)$ norm and there exists a unique $(\psi,\Psi) \in \ma_{\delta^\ast, r_1}$ such that $\mt (\psi,\Psi) = (\psi,\Psi)$.
\par Finally, we choose
\begin{equation*}
\sigma^\ast=\min\left\{\frac{\delta_2^\ast}{4\mc_3^\ast},\sigma_1^\ast,\sigma_2^\ast\right
\}.
\end{equation*}
Then if $ \sigma_p\leq \sigma^\ast$,  the problem \eqref{3-3} has a unique solution $ (\psi,\Psi)\in \left(H^4(\m)\right)^2$  with the estimate
\begin{equation}\label{3-49-1}
\Vert(\psi,\Psi)\Vert_{H^4(\Omega)}\leq 2\mc_3^\ast\sigma_p,
\end{equation}
and the compatibility conditions
\begin{equation}
\label{comp-cond-nlbvp-full-p}
\begin{cases}
\p_{\th}^k\psi=\p_{\th}^k\Psi=0,  \ \ {\rm{on}} \ \Gamma_{\th_0}^\pm,\  \  {\rm{for}} \ \ k=1,3,\\
\p_{z}^k\psi=\p_{z}^k\Psi=0,  \ \ {\rm{on}} \ \Gamma_{1}^\pm,\  \  {\rm{for}} \ \ \ k=1,3.
\end{cases}
\end{equation}
That is, the background supersonic   flow is structurally stable within irrotational flows under three-dimensional perturbations of the boundary
conditions in \eqref{1-c}.
\section{The stability analysis within axisymmetric rotational flows}\label{rotational}\noindent
\par In this section, we first  utilize the deformation-curl-Poisson decomposition  to reformulate the  steady axisymmetric Euler-Poisson system.  Then  we  establish the well-posedness of   the associated linearized hyperbolic-elliptic coupled system. Finally, we  prove  Theorem \ref{th2}.
\subsection{The deformation-curl-Poisson decomposition  }\noindent
\par  The steady axisymmetric  Euler-Poisson system   is hyperbolic-elliptic mixed in the supersonic region.  Thus the
solvability of nonlinear boundary value problems for such mixed system is extremely difficult.   To resolve Problem \ref{probl2},
    one needs
to  effectively decouple the elliptic and hyperbolic modes for further mathematical analysis. Here
  we  utilize the deformation-curl-Poisson decomposition developed in \cite{WX19,WS19}  to deal with the
hyperbolic-elliptic coupled structure in the system \eqref{2-10}. This decomposition
is based on a simple observation that one can rewrite the continuity equation as a Frobenius inner
product of a symmetric matrix and the deformation matrix by employing the Bernoulli's law and
representing the density as a function of the Bernoulli's quantity, the entropy,  the velocity field and the eletrostatic potential.
The vorticity is resolved by an algebraic equation of the Bernoulli¡¯s quantity and the entropy. 
\par Firstly,    the angular velocity,  the Bernoulli's quantity and  the entropy   are transported by the following equations:
\begin{align}\label{5-1}
&\bigg( U_1\p_r + U_3 \p_{z} \bigg)(rU_2)=0, \\\label{5-c}
 &\bigg( U_1\p_r + U_3 \p_{z} \bigg)K=0, \\\label{5-x}
   &\bigg( U_1\p_r + U_3 \p_{z} \bigg)S=0.
 \end{align}
Next, using the Bernoulli's fucntion,  the density $ \rho $  can be represented  as
\begin{equation}\label{5-3}
\rho=\mh(K,S,\Phi,U_1, U_2, U_3)=
\left(\frac{\gamma-1}{\gamma e^S}(K+\Phi-\frac{1}{2}(U_1^2+U_2^2+U_3^2))\right)
^{\frac{1}{\gamma-1}}.
\end{equation}
Substituting \eqref{5-3}   into the first equation in \eqref{2-10} derives that
\begin{equation}\label{5-5-a}
\begin{aligned}
&\bigg(c^2(K,U_1,  U_2,U_3,\Phi)-U_1^2\bigg)\p_r U_1+\bigg(c^2(K,U_1,  U_2,U_3,\Phi)-U_3^2\bigg)\p_z U_3
-{U_1U_3}\bigg(\p_rU_3+\p_z U_1\bigg)\\
&+{\bigg( U_1\p_r + U_3 \p_{z} \bigg)\Phi}
+\frac {c^2(K,U_1,  U_2,U_3,\Phi)U_1}{r}=U_2{\bigg( U_1\p_r + U_3 \p_{z} \bigg)U_2}-{\bigg( U_1\p_r + U_3 \p_{z} \bigg)K} \\
&+\frac{c^2(K,U_1,  U_2,U_3,\Phi)}{\gamma-1}{\bigg( U_1\p_r + U_3 \p_{z} \bigg)S}.\\
\end{aligned}
\end{equation}
Furthermore, one can follow from the fourth equation in \eqref{2-10} to obtain that
\begin{equation}\label{5-4}
{U_1}\bigg(\partial_rU_3-\partial_{z} U_1\bigg)=U_2\p_zU_2+\frac{e^S\mh^{\gamma-1}(K,S,\Phi,U_1, U_2, U_3) }{\gamma-1}{\p_z S}-\p_z K.
\end{equation}

\par Therefore, if a smooth flow does not contain the vacuum and the stagnation points, then the system \eqref{2-10} is equivalent to the following system:
\begin{equation}\label{5-5}
\begin{cases}
\begin{aligned}
&\bigg(c^2(K,U_1,  U_2,U_3,\Phi)-U_1^2\bigg)\p_r U_1+\bigg(c^2(K,U_1,  U_2,U_3,\Phi)-U_3^2\bigg)\p_z U_3
-{U_1U_3}\bigg(\p_rU_3+\p_z U_1\bigg)\\
&+{\bigg( U_1\p_r + U_3 \p_{z} \bigg)\Phi}
+\frac {c^2(K,U_1,  U_2,U_3,\Phi)U_1}{r}=U_2{\bigg( U_1\p_r + U_3 \p_{z} \bigg)U_2}-{\bigg( U_1\p_r + U_3 \p_{z} \bigg)K} \\
&+\frac{c^2(K,U_1,  U_2,U_3,\Phi)}{\gamma-1}{\bigg( U_1\p_r + U_3 \p_{z} \bigg)S},\\
&{U_1}\bigg(\partial_rU_3-\partial_{z} U_1\bigg)=U_2\p_zU_2+\frac{e^S\mh^{\gamma-1}(K,S,\Phi,U_1, U_2, U_3),
 }{\gamma-1}{\p_z S}-\p_z K,\\
 &\bigg(\p_r^2+\frac 1 r\p_r+\p_{z}^2\bigg)\Phi=\mh(K,S,\Phi,U_1, U_2, U_3)-b,\\
&\bigg( U_1\p_r + U_3 \p_{z} \bigg) (rU_2,K,S)=0.
\end{aligned}
\end{cases}
\end{equation}
 Set
\begin{equation*}
\begin{aligned}
&(V_1,V_2,V_3)(r,z)=(U_1,U_3,\Phi)(r,z)-(\bar U,0,\bar \Phi)(r), \ \ \  (r,z)\in \mn, \\
&(W_1,W_2,W_3)(r,z)=(U_2,K,S)(r,z)-(0,K_0,S_0), \ \ \ \   (r,z)\in \mn,\\
&{\bf V}=(V_1,V_2,V_3), \ \ \bar{\bf V}=(\bar U,0,\bar \Phi), \ \ {\bf W}=(W_1,W_2,W_3) , \ \ {\bf W}_0=(0, K_0,S_0).\\
\end{aligned}
\end{equation*}
Then one can use $ {\bf V}$ and $ {\bf W} $ rewrite the system \eqref{5-5}  in $\mn$   as follows
\begin{equation}\label{5-6}
\begin{cases}
\begin{aligned}
&B_{11}(r,z,{\bf V},{\bf W})\p_r V_1+ B_{22}(r,z,{\bf V},{\bf W})\p_z V_2
+B_{12}(r,z,{\bf V})\p_rV_2+B_{21}(r,z,{\bf V})\p_z V_1\\
&+\x a_1(r)V_1+\x b_1(r)\p_r V_3+\x b_2(r)V_3=G_1(r,z,{\bf V},{\bf W})
\\
&\partial_r V_2-\partial_{z} V_1=G_2(r,z,{\bf V},{\bf W}) ,\\
&\bigg(\p_r^2+\frac 1 r\p_r+\p_{z}^2\bigg)V_3+\x a_2(r)V_1-\x b_3(r)V_3=G_3(r,z,{\bf V},{\bf W}),\\
&\bigg( U_1\p_r + U_3 \p_{z} \bigg) (rW_1)=0,\\
&\bigg( U_1\p_r + U_3 \p_{z} \bigg) W_2=0,\\
&\bigg( U_1\p_r + U_3 \p_{z} \bigg) W_3=0,\\
\end{aligned}
\end{cases}
\end{equation}
where
\begin{equation*}
\begin{aligned}
&B_{11}(r,z,{\bf V},{\bf W})={{c^2({\bf V}+\bar {\bf V},W_1,W_2+ K_0)-(\bar U +V_1)^2}}, \\
&B_{22}(r,z,{\bf V},{\bf W})={c^2({\bf V}+\bar {\bf V},W_1,W_2+K_0)- V_2^2}, \\
&B_{12}(r,z,{\bf V})=B_{21}(r,z,{\bf V})=-{(\bar U +V_1) V_2},\\
&\x a_1(r)=\frac{\bar c^2-(\gamma-1)\bar U^2}{r}-{(\gamma+1)(\bar U\bar U')}+{\bar E(r)}, \ \ \x a_2(r)=\frac{\bar \rho\bar U}{\bar c^2},\\
 &\x b_1(r)={\bar U},\ \ \x b_2(r)={(\gamma-1)\bar U'}+\frac{(\gamma-1)\bar U  }{r}
, \ \  \x b_3(r)=\frac{\bar \rho}{\bar c^2},\\
&G_1(r,z,{\bf V},{\bf W})=-\frac{V_1}r(c^2({\bf V}+\bar {\bf V},W_1,W_2+K_0)-\bar c^2)
-\bigg((\gamma-1)(W_2-\frac12W_1^2-\frac12V_2^2)-\frac{\gamma+1}{2}V_1^2\bigg)\bar U'\\
&\quad  -\frac{1}r\bigg((\gamma-1)(W_2-\frac12W_1^2-\frac12V_2^2-\frac{1}{2}V_1^2)\bigg)\bar U -V_1\p_rV_3-V_2\p_z V_3+W_1{\bigg( (\bar U+V_1)\p_r + U_3 \p_{z} \bigg)W_1}
\\
\end{aligned}
\end{equation*}
\begin{equation*}
\begin{aligned}
&\quad -\bigg( (\bar U+V_1)\p_r + U_3 \p_{z} \bigg)W_2+ \frac{c^2({\bf V}+\bar {\bf V},W_1,W_2+K_0)}{\gamma-1}\bigg( (\bar U+V_1)\p_r + U_3 \p_{z} \bigg)W_3,\\
&G_2(r,z,{\bf V},{\bf W})=\frac{1}{V_1+\bar U}\bigg(W_1\p_zW_1+\frac{e^{(W_2+S_0)}\mh^{\gamma-1}({\bf V}+\bar {\bf V},{\bf W}+ {\bf W}_0) }{\gamma-1}{\p_z W_3}-\p_z W_2\bigg),\\
&G_3(r,z,{\bf V},{\bf W})=\mh({\bf V}+\bar {\bf V},{\bf W}+ {\bf W}_0)-\mh(\bar {\bf V}, {\bf W}_0)+\x a_2V_1-\x b_3V_3-(b-b_0).\\
\end{aligned}
\end{equation*}
Furthermore, the system \eqref{5-6} is  supplemented with the following boundary
conditions:
\begin{equation}\label{5-7}
\begin{cases}
(V_1, V_2,\p_r V_3)(r_0,z)=(U_{1,en}^\star,   U_{3,en}^\star,E_{en}^\star)(z)
-(U_0,0,E_0),\ \  &{\rm{on}} \ \ \Sigma_{en},\\
(W_1,W_2,W_3)(r_0,z)=(U_{2,en}^\star,K_{en}^\star,S_{en}^\star)(z)
-(0,K_0,S_0),\ \ &{\rm{on}} \ \ \Sigma_{en},\\
V_3(r_1,z)=\Phi_{ex}^\star(z)-\bar\Phi(r_1), \ \ &{\rm{on}}\ \  \Sigma_{ex},\\
V_2(r,\pm 1)=\p_z V_3(r,\pm 1)=0,  \ \  &{\rm{on}}\ \ \Sigma_{1}^\pm.\\
\end{cases}
\end{equation}
\par Note that $G_2({\bf V},{\bf W})$ only belongs to $H^2(\mn)$ if one looks for the solution $({\bf V},{\bf W})$ in $\left(H^3(\mn)\right)^2\times H^4(\mn)\times \left(H^3(\mn)\right)^3$.   The first two equations in \eqref{5-6} can be regarded as a first order system for $ (V_1, V_2)$,  the energy estimates obtained in the previous section for irrotational flows
indicate that the regularity of the solutions $V_1$ and $ V_2$ would be at best same as the source terms on the right hand sides in general. Hence  it seems that only $H^2(\mn)$ regularity for $(V_1,V_2)$ is possible and there appears a loss of derivatives. To recover the loss of derivatives,   we use the continuity equation to introduce the stream function which has the advantage of one order higher regularity than the  radial velocity and the vertical velocity. The angular velocity, the Bernoulli's quantity and the entropy can be represented as  functions of the stream function.  This
will enable us to overcome the possibility of losing derivatives. To achieve this, we will choose some
appropriate function spaces and design an elaborate two-layer iteration scheme to prove Theorem \ref{th2}.
\subsection{The linearized problem }\noindent
\par  To solve the nonlinear boundary value problem \eqref{5-6} with \eqref{5-7} by the method of iteration, we define the following the iteration sets:
\begin{equation}\label{5-8}
\begin{aligned}
  \mj_{\delta_{1}^\star, r_1}&=\bigg\{{\bf W}=(W_1,W_2,W_3)\in
  \left(H^4(\mn)\right)^3:\,\,\|(W_1,W_2,W_3)\|_{H^4(\mn)}\leq \delta_{1}^\star,\\
   &\quad\ \ \ \p_{z}^k W_1=\p_{z}^k W_2=\p_{z}^k W_3=0, \ \ {\rm{on}} \ \Sigma_{1}^\pm,  \ \  {\rm{for}} \ \ k=1,3\bigg\},\\
   \end{aligned}
  \end{equation}
  and
  \begin{equation}
  \label{5-9}
  \begin{aligned}
\mj_{\delta_{2}^\star, r_1}&=\bigg\{{\bf V}=(V_1,V_2,V_3)\in \left(H^3(\mn)\right)^2\times H^4(\mn):\,\,\|(V_1,V_2)\|_{H^3(\mn)}+\|V_3\|_{H^4(\mn)}\leq \delta_{2}^\star,\\
&\quad\quad \p_{z} V_1=\p_z^{k-1}V_2=\p_{z}^kV_3=0,  \ \ {\rm{on}} \ \Sigma_{1}^\pm,  \ \  {\rm{for}} \ \ k=1,3\bigg\},
\end{aligned}
\end{equation}
where the positive constants $\delta_{1}^\star$, $\delta_{2}^\star$ and $r_1$ will be determined later.
\par Next, for   fixed  $ \h {\bf W}\in \mj_{\delta_{1}^\star, r_1} $ and for any function $\h{\bf V}\in \mj_{\delta_{2}^\star, r_1} $, we first construct an
operator $ \mb_1^{\h {\bf W}}:\h{\bf V}\in \mj_{\delta_{2}^\star, r_1}\mapsto{\bf V}\in \mj_{\delta_{2}^\star, r_1} $  by resolving the following boundary value problem:
\begin{equation}\label{5-10}
\begin{cases}
B_{11}(r,z,\h{\bf V},\h{\bf W})\p_r V_1+ B_{22}(r,z,\h{\bf V},\h{\bf W})\p_z V_2
+B_{12}(r,z,\h{\bf V})\p_rV_2\\
+B_{21}(r,z,\h{\bf V})\p_z V_1
+\x a_1(r)V_1+\x b_1(r)\p_r V_3+\x b_2(r)V_3=G_1(r,z,\h{\bf V},\h{\bf W}), \ \ &{\rm{in}} \ \ \mn,\\
\partial_r V_2-\partial_{z} V_1=G_2(r,z,\h {\bf V},\h {\bf W}), \ \ &{\rm{in}} \ \ \mn,\\
\bigg(\p_r^2+\frac 1 r\p_r+\p_{z}^2\bigg)V_3+\x a_2(r)V_1-\x b_3(r)V_3=G_3(r,z,\h{\bf V},\h{\bf W}), \ \ &{\rm{in}} \ \ \mn,\\
(V_1, V_2,\p_r V_3)(r_0,z)=(U_{1,en}^\star,   U_{3,en}^\star,E_{en}^\star)(z)
-(U_0,0,E_0),\ \  &{\rm{on}} \ \ \Sigma_{en},\\
V_3(r_1,z)=\Phi_{ex}^\star(z)-\bar\Phi(r_1), \ \ &{\rm{on}}\ \  \Sigma_{ex},\\
V_2(r,\pm 1)=\p_z V_3(r,\pm 1)=0,  \ \  &{\rm{on}}\ \ \Sigma_{1}^\pm.\\
\end{cases}
\end{equation}
To  simplicity the notation, we denote
\begin{equation*}
\begin{aligned}
&\h B_{ii}(r,z):=B_{ii}(r,z,\h{\bf V},\h{\bf W}),\ \  (r,z)\in \mn, \  i=1,2,  \\
&\h B_{ij}(r,z):=B_{ij}(r,z,\h{\bf V}),\quad \ \ \   (r,z)\in \mn, \  i\neq j=1,2, \\
 &\h G_i(r,z):=G_i (r,z,\h{\bf V},\h {\bf W}), \ \  \  (r,z)\in \mn, \ i=1,2,3.
 \end{aligned}
\end{equation*}
Then we have the following proposition.
\begin{proposition}\label{pro5-2}
Let  $R_0$ be given by  Proposition \ref{pro1}.
\begin{enumerate}[ \rm (i)]
\item
For any $r_1\in(r_0,r_0+R_0)$, set
\begin{eqnarray*}
\begin{aligned}
\bar B_{11}(r)=\bar c^2-{\bar U^2}, \ \ \bar B_{12}(r)=\bar B_{21}(r)=0, \ \
\bar B_{22}(r)=\bar c^2, \ \   r\in [r_0, r_1].
\end{aligned}
\end{eqnarray*}
 Then there exists a constant $\bar\mu_2\in(0,1)$ depending only on $(b_0, \A_0,U_{0}, S_{0},E_{0})$ such that
\begin{eqnarray}\label{5-7-a}
&&\bar\mu_2\leq-\bar B_{11}(r)\leq \frac{1}{\bar\mu_2},\ \ \bar\mu_2\leq \x B_{22}(r)\leq \frac{1}{\bar\mu_2},
\ \   r\in [r_0, r_1].
\end{eqnarray}
Furthermore, the coefficients $\bar B_{ii}$,   $\x a_i  $ and  $\x b_j $ for $ i=1,2 $ and  $j=1,2,3$  are smooth functions. More precisely, for each $k\in\mathbb{Z}^+$, there exists a constant $\bar C_k>0$ depending only on $(b_0, \A_0,U_{0}, S_{0}, E_{0},r_0,R,k)$ such that
\begin{equation}\label{5-7-b}
\Vert (\bar B_{11},\bar B_{22}, \x a_1,\x a_2, \x b_1, \x b_2,\x b_3)\Vert_{C^k([r_0,r_1])}\leq\bar C_k.
\end{equation}
\item
For each $\epsilon_0\in(0, R_0)$ ,
there exists a positive constant $\delta_3^\star$ depending only on  $(b_0,   \A_0, U_{0},S_{0},E_{0},\epsilon_0)$ such that    for $r_1\in(r_0,r_0+R_0-\epsilon_0]$ and $\max\{\delta_1^\star,\delta_2^\star\}\leq \delta_3^\star$, the coefficients $ (\h B_{11}, \h B_{12},\h B_{22}) $ for $ (\h {\bf W},\h{\bf V})\in \mj_{\delta_{1}^\star, r_1} \times \in \mj_{\delta_{2}^\star, r_1} $ with $\mj_{\delta_{1}^\star, r_1} $ and $  \mj_{\delta_{2}^\star, r_1} $   given by \eqref{5-8} and \eqref{5-9} satisfy the following
properties.
\begin{enumerate}[ \rm (a)]
 \item There exists a positive constant  $C$ depending only on $(b_0, \A_0, U_{0},S_{0},E_{0},\epsilon_0)$ such that
\begin{eqnarray}\label{5-7-f}
 \|(\h B_{11}, \h B_{12},\h B_{22})-(\bar B_{11}, 0,\bar B_{22})\|_{H^3(\mn)}\leq C(\delta_1^\star+\delta_2^\star).
\end{eqnarray}
\item The coefficients $ (\h B_{11}, \h B_{12},\h B_{22}) $ satisfy the following compatibility conditions:
\begin{eqnarray}\label{5-7-f-1}
  \h B_{12}=\p_{z}\h B_{11}=\p_{z}\h B_{22}=0, \ \ {\rm{on}} \ \Sigma_{1}^\pm.
  \end{eqnarray}
 \item There exists a constant  $\mu_2\in (0,1)$  depending only on $(b_0, \A_0, U_{0},S_{0},E_{0},\epsilon_0)$ so that the coefficients $ (\h B_{11},\h B_{22}) $ satisfy
\begin{equation}\label{5-7-e}
{\mu_2}\leq -\h B_{11}(r,z)\leq\frac1{\mu_2}, \
\ {\mu_2}\leq \h B_{22}(r,z)\leq\frac1{\mu_2},  \ \ (r,z)\in\overline{\mn}.
\end{equation}


\end{enumerate}
\end{enumerate}
\end{proposition}
\subsection{A second-order linear  hyperbolic-elliptic coupled system}\noindent
\par Firstly, we consider  the following problem:
\begin{equation}\label{5-12}
\begin{cases}
\partial_r^2 \varphi_1+\partial_{z}^2 \varphi_1=\h G_2 ,\ \ &{\rm{in}}\ \ \mn,\\
 \p_r\varphi_1(r_0,z)=0, &{\rm{on}}\ \ \Sigma_{en},\\
 \p_r\varphi_1(r_1,z)=0, &{\rm{on}}\ \ \Sigma_{ex},\\
 \varphi_1(r,\pm 1)=0, &{\rm{on}}\ \ \Sigma_{1}^\pm.\\
 \end{cases}
\end{equation}
A simple computation shows that
\begin{equation}\label{5-7-g}
 \|\h G_2\|_{H^3(\mn)}\leq C\delta_1^\star \ \ {\rm{and}} \ \  \h G_2 =\p_{z}^2\h G_2=0, \  {\rm{on}}\ \ \Sigma_{1}^\pm.
 \end{equation}
Hence one can use the standard symmetric extension
technique to deal with the regularity  near the corner.   We extend $ \varphi_1 $ and $  \h G_2 $ onto $ \mn^e=(r_0,r_1)\times(-3,3)$ by
\begin{equation}\label{2-13}
(\varphi_1^e,\h G_2^e)(r,z)=
\begin{cases}
(\varphi_1,\h G_2)(r,z), \ \ &  (r,z)\in (r_0,r_1)\times[-1,1],\\
-(\varphi_1,\h G_2)(r,-2-z), \ \ &  (r,z)\in (r_0,r_1)\times(-3,-1),\\
-(\varphi_1,\h G_2)(r,2-z), \ \ &  (r,z)\in (r_0,r_1)\times(1,3).\\
\end{cases}
\end{equation}
Then   $ \varphi_1^e $ satisfies
\begin{equation}\label{5-14}
\begin{cases}
\p_r^2\varphi_1^e+\p_z^2\varphi_1^e=\h G_2^e, \ \ &{\rm{in}}\ \ \mn^e,  \\
\p_r\varphi_1^e(r_0,z)=\p_r\varphi_1^e(r_1,z)=0,&  z\in (-3,3),\\
\varphi_1^e(r,-3)=\varphi_1^e(r,3)=0, & r\in (r_0,r_1).
\end{cases}
\end{equation}
Therefore, the standard  elliptic theory in \cite{GT98} shows that the problem \eqref{5-12} has a unique solution $ \varphi_1\in H^{5}(\mn) $ satisfying \begin{equation}\label{5-15}
\|\varphi_1\|_{H^5(\mn)} \leq C\|\h G_2\|_{H^3(\mn)} \ \ {\rm{and}} \ \  \p_{z}^2\varphi_1 =\p_{z}^4\varphi_1 =0,\ \ {\rm{on}} \ \Sigma_{1}^\pm.
 \end{equation}
 Furthermore, for $ \alpha\in (0,1) $, the Sobolev inequality yields that
  \begin{equation}\label{5-15-1}
\|\varphi_1\|_{C^{3,\alpha}(\overline\mn)} \leq C\|\h G_2\|_{H^3(\mn)}.
 \end{equation}
 \par
 Set
 \begin{equation*}
 \begin{aligned}
 &Q_1(r,z)=V_1(r,z)+\p_z \varphi_1(r,z), \ \ Q_2(r,z)=V_2(r,z)-\p_r \varphi_1(r,z)-U_{3,en}^\star(z), \ & (r,z)\in \mn,\\ &\Psi(r,z)=V_3(r,z)-(r-r_1)(E_{en}^\star(z)-E_0)-(\Phi_{ex}^\star(z)-\bar\Phi(r_1)), \ & (r,z)\in \mn.
 \end{aligned}
  \end{equation*}
  Then \eqref{5-10} can be transformed into
  \begin{equation}\label{5-16}
\begin{cases}
\h B_{11}\p_r Q_1+ \h B_{22}\p_z Q_2
+\h B_{12}\p_rQ_2+\h B_{21}\p_z Q_1\\
+\x a_1Q_1+\x b_1\p_r \Psi+\x b_2\Psi=\h G_4, \ \ &{\rm{in}} \ \ \mn,\\
\partial_r Q_2-\partial_{z} Q_1=0, \ \ &{\rm{in}} \ \ \mn,\\
\bigg(\p_r^2+\frac 1 r\p_r+\p_{z}^2\bigg)\Psi+\x a_2Q_1-\x b_3\Psi=\h G_5, \ \ &{\rm{in}} \ \ \mn,\\
Q_1(r_0,z)=\h G_6(z),\ \ Q_2(r_0,z)=\p_r \Psi(r_0,z)=0,\ \  &{\rm{on}} \ \ \Sigma_{en},\\
\Psi(r_1,z)=0, \ \ &{\rm{on}}\ \  \Sigma_{ex},\\
Q_2(r,\pm 1)=\p_z \Psi(r,\pm 1)=0,  \ \  &{\rm{on}}\  \ \Sigma_{1}^\pm,\\
\end{cases}
\end{equation}
where
\begin{equation*}
\begin{aligned}
\h G_4(r,z)&=\h G_1-\bigg((-\h B_{11}\p_{rz}^2 \varphi_1+ \h B_{22})\p_{rz}^2 \varphi_1+ \h B_{22}(U_{3,en}^\star)' +\h B_{12}(\p_r^2\varphi_1-\p_z^2 \varphi_1)
-\x a_1\p_z\varphi_1+\x b_1 (E_{en}^\star-E_0)\\
&\qquad\quad \ \ +\x b_2((r-r_1)(E_{en}^\star-E_0)+(\Phi_{ex}^\star-\bar\Phi(r_1)))\bigg), \ \ (r,z)\in \mn,\\
\h G_5(r,z)&=\h G_3-\bigg(\frac1r (E_{en}^\star-E_0 ) +(r-r_1)(E_{en}^\star)''+(\Phi_{ex}^\star)''-\x a_2\p_z \varphi_1\\
&\qquad\quad\ \ -\x b_3((r-r_1)(E_{en}^\star-E_0)+(\Phi_{ex}^\star-\bar\Phi(r_1)))\bigg),\ \ (r,z)\in \mn,\\
 G_6(z)&=U_{1,en}^\star(z)-U_0+\p_z \varphi_1(r_0,z), \ \ z\in[-1,1].
\end{aligned}
\end{equation*}
It  follows from the expressions of  $ \h G_1 $ and $ \h G_3 $ together with direct computations that one obtains
    \begin{eqnarray}\label{5-7-h}
\|\h G_1\|_{H^3(\mn)}+ \|\h G_3\|_{H^2(\mn)}\leq C\left(\delta_1^\star+(\delta_2^\star)^2+\sigma(b^\star,U_{ 1,en}^\star,U_{2, en}^\star,U_{3, en}^\star,K_{en}^\star,S_{en}^\star,E_{ en}^\star,\Phi_{ ex}^\star)\right),
 \end{eqnarray}
 and the  compatibility conditions
 \begin{eqnarray}\label{5-7-j}
\p_{z}\h G_1=\p_{z}\h G_3= 0,\ \ {\rm{on}} \ \Sigma_{1}^\pm.
 \end{eqnarray}
Then one can use  \eqref{5-15}-\eqref{5-15-1} and \eqref{5-7-h}-\eqref{5-7-j} to  derive that
\begin{equation}\label{5-17}
\begin{aligned}
&\|\h G_4\|_{H^3(\mn)}+\|\h G_5\|_{H^2(\mn)}+\| G_6\|_{H^3([-1,1])}\\
&\leq C\left(\|\h G_1\|_{H^3(\mn)}+ \|\h G_3\|_{H^2(\mn)}+\sigma(b^\star,U_{ 1,en}^\star,U_{2, en}^\star,U_{3, en}^\star,K_{en}^\star,S_{en}^\star,E_{ en}^\star,\Phi_{ ex}^\star)\right),
\end{aligned}
\end{equation}
and
\begin{equation}\label{5-17-c}
 \p_z\h G_4=\p_z\h G_5=0,\ \ {\rm{on}} \ \Sigma_{1}^\pm, \ \  G_6'(\pm 1)=0,
\end{equation}
where the constant $C>0$ depends only on $(b_0, \A_0, U_{0},S_{0},E_{0})$.
\par  The second equation in \eqref{5-16} implies that there exists a potential function $ \psi(r,z)$ such that
$$ \p_z\psi=      Q_2, \quad \p_r\psi=Q_1, \ \ \psi(r_0,-1)=0, \ \ {\rm{in}}\ \ \mn.$$
Then $ \psi$ and $   \Psi $ satisfy the following problem:
\begin{equation}\label{5-18}
\begin{cases}
\h L_1(\psi,\Psi)=\h B_{11}\p_r^2\psi+ \h B_{22}\p_z^2\psi
+2\h B_{12}\p_{rz}^2\psi+\x a_1\p_r\psi\\
\qquad\qquad+\x b_1\p_r\Psi+\x b_2\Psi=\h G_4, \ \ &{\rm{in}} \ \ \mn,\\
\h L_2(\psi,\Psi)=
\bigg(\p_r^2+\frac 1 r\p_r+\p_{z}^2\bigg) \Psi+\x a_2\p_r\psi-\x b_3\Psi=\h G_5, \ \ &{\rm{in}} \ \ \mn,\\
\p_r\psi(r_0,z)= G_6(z),\ \psi(r_0,z)=\p_r \Psi(r_0,z)=0,\ \  &{\rm{on}} \ \ \Sigma_{en},\\
\Psi(r_1,z)=0, \ \ &{\rm{on}}\ \  \Sigma_{ex},\\
\p_z\psi(r,\pm 1)=\p_z \Psi(r,\pm 1)=0,  \ \  &{\rm{on}}\ \ \Sigma_{1}^\pm.\\
\end{cases}
\end{equation}
\par Next, by the similar arguments in Lemmas \ref{pro4} and \ref{pro5}, we can obtain  the energy estimates for the problem   \eqref{5-18} under the assumptions that $(\h B_{11}, \h B_{12},\h B_{22})\in \left(C^{\infty}(\overline{\mn})\right)^3$ and   \eqref{5-7-f}-\eqref{5-7-e}  hold.
\begin{lemma}\label{pro5-3}
(A priori $H^1$-estimate)
 Let   $R_0 $ and  $ \delta_3^\star$  be from  Proposition \ref{pro1} and  Proposition \ref{pro5-2}, respectively.
  For each $\epsilon_0\in(0,R_0)$, there exists a constant $\bar r_1^\star\in(r_0,r_0+R_0-\epsilon_0]$ depending only $(b_0,   \A_0, U_{0},S_{0},E_{0},\epsilon_0)$ and a sufficiently small constant  $ \delta_4^\star\in(0,\frac{\delta_3^\star}2]$
such that for  $r_1\in(r_0,\bar r_1^\star)$ and $\max\{\delta_1^\star,\delta_2^\star\}\leq \delta_4^\star$, the classical  solution $ (\psi,\Psi) $ to  \eqref{5-18} satisfies the energy estimate
\begin{equation}\label{5-19}
\|(\psi,\Psi)\|_{H^1(\mn)}\leq C\left(\|\h G_4\|_{L^2(\mn)}+\|\h G_5\|_{L^2(\mn)}+\| G_6\|_{L^2([-1,1])}\right)
\end{equation}
for the constant $C>0$ depending only on $(b_0,   \A_0, U_{0},S_{0},E_{0},r_0,r_1,\epsilon_0)$.
\end{lemma}
\begin{lemma}\label{pro5-4}
(A priori $H^4$-estimate)
  For fixed $\epsilon_0\in(0,R_0)$, let  $\bar r_1^\star\in(r_0,r_0+R_0-\epsilon_0]$ and $ \delta_4^\star\in(0, \frac{\delta_3^\star}2] $ be from  Lemma \ref{pro5-3}. There exists a positive constant $C$ depending only on $(b_0,   \A_0, U_{0},S_{0},E_{0},r_0,r_1,\epsilon_0)$ and  a  sufficiently small constant  $ \delta_5^\star\in(0,{\delta_4^\star}]$
such that for  $ r_1\in (r_0,\bar r_1^\star) $ and $\max\{\delta_1^\star,\delta_2^\star\}\leq \delta_5^\star$, the classical   solution $ (\psi,\Psi) $ to  \eqref{5-18} satisfies
\begin{eqnarray}\label{5-H^4-1-1-1}
&&\Vert\psi\Vert_{H^4(\mn)}
\leq C\left(\|\h G_4\|_{H^3(\mn)}+\|\h G_5\|_{H^2(\mn)}+\| G_6\|_{H^3([-1,1])}\right),\\\label{5-H^4-2-2}
&&\Vert\Psi\Vert_{H^4(\mn)}
\leq C\left(\|\h G_4\|_{H^2(\mn)}+\|\h G_5\|_{H^2(\mn)}+\| G_6\|_{H^2([-1,1])}\right).
\end{eqnarray}

 \end{lemma}
  \subsection{The well-posedness of the linearized  problem}\noindent
  \par With the aid of Lemmas \ref{pro5-3} and \ref{pro5-4}, we have the following well-posedness of the linearized  problem \eqref{5-18}.
  \begin{proposition}\label{pro5-5}

  Fix  $\epsilon_0\in(0,R_0)$,  let    $\bar r_1^\star\in(r_0,r_0+R_0-\epsilon_0]$ and $ \delta_5^\star\in(0, {\delta_4^\star}] $ be from  Lemma \ref{pro5-3} and  Lemma \ref{pro5-4}, respectively.
  Then there  exists a  positive constant $C$ depending only on $(b_0, \A_0,U_{0}, S_{0},E_{0},r_0,r_1,\epsilon_0)$  such that for $r_1\in(r_0,\bar r_1^\star)$ and $\max\{\delta_1^\star,\delta_2^\star\}\leq \delta_5^\star$, the  linear boundary value problem \eqref{5-18} associated with $ (\h {\bf W},\h{\bf V})\in \mj_{\delta_{1}^\star, r_1} \times \in \mj_{\delta_{2}^\star, r_1} $  has a unique solution  $ (\psi,\Psi)\in \left(H^4(\mn)\right)^2$ satisfying
  \begin{eqnarray}\label{5-5-H^4-1-1-1}
&&\Vert\psi\Vert_{H^4(\mn)}
\leq C\left(\|\h G_4\|_{H^3(\mn)}+\|\h G_5\|_{H^2(\mn)}+\| G_6 \|_{H^3([-1,1])}\right),\\\label{5-5-H^4-2-2}
&&\Vert\Psi\Vert_{H^4(\mn)}
\leq C\left(\|\h G_4\|_{H^2(\mn)}+\|\h G_5\|_{H^2(\mn)}+\| G_6\|_{H^2([-1,1])}\right).
\end{eqnarray}
In addition, the solution $(\psi,{\Psi})$ satisfies the compatibility conditions
\begin{equation}\label{5-5-H^4-1-1-c}
\p_{z}^k\psi=\p_{z}^k\Psi=0, \ \ {\rm{on}} \ \Sigma_{1}^\pm, \ \  {\rm{for}} \ \ k=1,3.
\end{equation}
\end{proposition}
\begin{proof}
\par It follows from \eqref{5-7-f}-\eqref{5-7-f-1} and \eqref{5-17}-\eqref{5-17-c} that  $(\h B_{11},\h B_{12},\h B_{22},\h G_4,\h G_5)\in \left(H^3(\mn)\right)^4\times H^2(\mn)$ satisfying
\begin{equation*}
\h B_{12} =\p_{z}\h B_{11}=\p_{z}\h B_{22}=\p_{z}\h G_4=\p_{z}\h G_5=0, \ \ {\rm{on}} \ \Sigma_{1}^\pm.
\end{equation*}
Then one can extend $ (\h B_{11},\h B_{12},\h B_{22},\h G_4,\h G_5) $ onto $ \mn^e=(r_0,r_1)\times(-3,3)$ by
\begin{equation*}
(\h B_{11}^e,\h B_{22}^e,\h G_4^e,\h G_5^e)(r,z)=
\begin{cases}
(\h B_{11},\h B_{22},\h G_4,\h G_5)(r,z), \ \ &  (r,z)\in (r_0,r_1)\times[-1,1],\\
(\h B_{11},\h B_{22},\h G_4,\h G_5)(r,-2-z), \ \ &  (r,z)\in (r_0,r_1)\times(-3,-1),\\
(\h B_{11},\h B_{22},\h G_4,\h G_5)(r,2-z), \ \ &  (r,z)\in (r_0,r_1)\times(1,3),\\
\end{cases}
\end{equation*}
and
\begin{equation*}
\h B_{12}^e(r,z)=
\begin{cases}
\h B_{12}(r,z), \ \ &  (r,z)\in (r_0,r_1)\times [-1,1],\\
-\h B_{12}(r,-2-z), \ \ &  (r,z)\in (r_0,r_1)\times(-3,-1),\\
-\h B_{12}(r,2-z), \ \ &  (r,z)\in (r_0,r_1)\times (1,3).\\
\end{cases}
\end{equation*}
  It is easy to check that $(\h B_{11}^e,\h B_{12}^e,\h B_{22}^e,\h G_4^e,\h G_5^e)\in \left(H^3(\mn^e)\right)^4\times H^2(\mn^e)$.   Let $\chi_\eta $ be a  radially symmetric  standard mollifier  with a compact support in a disk of radius $\eta>0$. For $  (r,z)\in \mn^e $,  define
$$\h B_{11}^\eta:=\h B_{11}^e*\chi_\eta,\ \ \h B_{i2}^\eta:=\h B_{i2}*\chi_\eta, \ \ i=1,2, \ \ \h G_j^\eta:= \h G_j^e*\chi_\eta,\ \ j=4,5.$$
Then we have  $(\h B_{11}^\eta,\h B_{12}^\eta,\h B_{22}^\eta,\h G_4^\eta,\h G_5^\eta )\in \left(C^{\infty}(\overline \mn)\right)^5 $
 and
\begin{equation*}
\h B_{12}^{\eta} =\p_{z}\h B_{11}^\eta=\p_{z}\h B_{22}^\eta=\p_{z}\h G_4^\eta=\p_{z}\h G_5^\eta=0, \ \ {\rm{on}} \ \Sigma_{1}^\pm.
\end{equation*}
 Furthermore, as $ \eta\rightarrow0 $, one derives
$$ \|\h B_{11}^{\eta}-\h B_{11}\|_{H^3(\mn)}\rightarrow0, \ \ \|\h B_{i2}^{\eta}-\h B_{i2}\|_{H^3(\mn)}\rightarrow0,\ i=1,2,\ \  \|\h G_4^{\eta}-\h G_4\|_{H^3(\mn)}\rightarrow0, \  \|\h G_5^{\eta}-\h G_5\|_{H^2(\mn)}\rightarrow0. $$
Therefore, a  approximation problem with smooth coefficients of \eqref{5-18} is obtained as follows:
\begin{equation}\label{5-20}
\begin{cases}
\h L_1^\eta(\psi,\Psi)=\h B_{11}^\eta\p_r^2\psi+ \h B_{22}^\eta\p_z^2\psi
+2\h B_{12}^\eta\p_{rz}^2\psi++\x a_1\p_r\psi\\
\qquad\qquad+\x b_1\p_r\Psi+\x b_2\Psi=\h G_4^\eta, \ \ &{\rm{in}} \ \ \mn,\\
\h L_2(\psi,\Psi)=
\bigg(\p_r^2+\frac 1 r\p_r+\p_{z}^2\bigg) \Psi+\x a_2\p_r\psi-\x b_3\Psi=\h G_5^\eta, \ \ &{\rm{in}} \ \ \mn,\\
\p_r\psi(r_0,z)= G_6(z),\ \psi(r_0,z)=\p_r \Psi(r_0,z)=0,\ \  &{\rm{on}} \ \ \Sigma_{en},\\
\Psi(r_1,z)=0, \ \ &{\rm{on}}\ \  \Sigma_{ex},\\
\p_z\psi(r,\pm 1)=\p_z \Psi(r,\pm 1)=0,  \ \  &{\rm{on}}\ \ \Sigma_{1}^\pm.\\
\end{cases}
\end{equation}
\par Next, one can follow the arguments  in Step 2 of Proposition \ref{pro6} to find smooth functions
\begin{equation*}
\psi_m^\eta(r,z)=\sum_{j=0}^{m}\my_j^\eta(r)\beta_j(z) \ \ {\rm{and}} \ \ \Psi_m^\eta(r,z)=\sum_{j=0}^{m}\ml_j^\eta(r)\beta_j(z),\ \  (r,z)\in \mn,
\end{equation*}
which solve \eqref{5-20} and satisfy
\begin{eqnarray}\label{5-H^4-1-1-a}
&&\Vert\psi_m^\eta\Vert_{H^4(\mn)}
\leq C\left(\|\h G_4^\eta\|_{H^3(\mn)}+\|\h G_5^\eta\|_{H^2(\m)}+\|  G_6\|_{H^3([-1,1])}\right),\\\label{5-H^4-2-2-b}
&&\Vert\Psi_m^\eta\Vert_{H^4(\mn)}
\leq C\left(\|\h G_4^\eta\|_{H^2(\mn)}+\|\h G_5^\eta\|_{H^2(\m)}+\|  G_6\|_{H^2([-1,1])}\right),
\end{eqnarray}
where  the constant $ C>0 $ is independent of $\eta$ and $m$.
\par Finally,  by the arguments  in Step 3 of Proposition \ref{pro6}, we take the limit $m\rightarrow \infty $ of $\{(\psi_m^\eta,\Psi_m^\eta)\}_{m\in\mathbb{N}}$ in order to obtain a $H^4 $
solution to the approximated boundary value problem \eqref{5-20}, then take
another limit $ \eta \rightarrow 0 $  to obtain a $H^4 $
solution to the boundary value problem \eqref{5-18} which satisfies the estimates \eqref{5-5-H^4-1-1-1}-\eqref{5-5-H^4-2-2} and the compatibility condition \eqref{5-5-H^4-1-1-1}.
Thus Proposition \ref{pro5-5} is
proved.

\end{proof}
\par From Proposition \ref{pro5-5}, the well-posedness of \eqref{5-10} directly follows.
\begin{corollary}\label{cor1}
 Fix  $\epsilon_0\in(0,R_0)$,  let    $\bar r_1^\star\in(r_0,r_0+R_0-\epsilon_0]$ and $ \delta_5^\star\in(0, {\delta_4^\star}] $ be from  Lemma \ref{pro5-3} and \ref{pro5-4}, respectively. For  $r_1\in(r_0,\bar r_1^\star)$ and $\max\{\delta_1^\star,\delta_2^\star\}\leq \delta_5^\star$, let the iteration sets $\mj_{\delta_{1}^\star, r_1} $ and $  \mj_{\delta_{2}^\star, r_1} $ be  given by \eqref{5-8} and \eqref{5-9}, respectively.    Then for each $ (\h {\bf W},\h{\bf V})\in \mj_{\delta_{1}^\star, r_1} \times \in \mj_{\delta_{2}^\star, r_1} $, the associated linear boundary value problem \eqref{5-10} has a unique solution  $ (V_1,V_2,V_3)\in \left(H^3(\mn)\right)^2\times H^4(\mn)$ that satisfies
\begin{equation}\label{5-1-e}
\begin{aligned}
\Vert(V_1,V_2)\Vert_{H^3(\mn)}+\Vert V_3\Vert_{H^4(\mn)}
&\leq C\bigg(\|(\h G_1,\h G_2)\|_{H^3(\m)}+\|\h G_3\|_{H^2(\m)}\\
&\quad\quad  +\sigma(b^\star,U_{ 1,en}^\star,U_{2, en}^\star,U_{3, en}^\star,K_{en}^\star,S_{en}^\star,E_{ en}^\star,\Phi_{ ex}^\star)\bigg),\\
\end{aligned}
\end{equation}
 for the constant  $C>0$ depending only on $(b_0, \A_0, U_{0},S_{0},E_{0},r_0,r_1,\epsilon_0)$. Furthermore, the solution $(V_1,V_2,V_3)$ satisfies the compatibility conditions
\begin{equation}\label{5-1-c}
\p_{z} V_1=\p_z^{k-1}V_2=\p_{z}^kV_3=0, \ \ {\rm{on}} \ \Sigma_{1}^\pm,\  \  {\rm{for}} \ \ k=1,3.
\end{equation}
\end{corollary}
\begin{proof}
By Proposition \ref{pro5-5}, for each   $ (\h {\bf W},\h{\bf V})\in \mj_{\delta_{1}^\star, r_1} \times \in \mj_{\delta_{2}^\star, r_1} $,  the  linear boundary value problem \eqref{5-18} has a unique solution  $ (\psi,\Psi)\in \left(H^4(\mn)\right)^2$ satisfying the estimates \eqref{5-5-H^4-1-1-1} and \eqref{5-5-H^4-2-2}. Note that
\begin{equation*}
\begin{aligned}
V_1(r,z)&=\p_r\psi(r,z)-\p_z \varphi_1(r,z),\ & (r,z)\in \mn,\\
V_2(r,z)&=\p_z\psi(r,z)+\p_r \varphi_1(r,z)+U_{3,en}^\star(z), \ & (r,z)\in \mn,\\
V_3(r,z)&=\Psi(r,z)+(r-r_1)(E_{en}^\star(z)-E_0)-(\Phi_{ex}^\star(z)-\bar\Phi(r_1)),
& (r,z)\in \mn.
 \end{aligned}
  \end{equation*}
  Therefore, it is easy to see that  $(V_1,V_2,V_3)$  solves \eqref{5-10}.  The estimate \eqref{5-1-e} can be obtained by using \eqref{5-15} and  \eqref{5-17}. Finally, the compatibility condition \eqref{5-1-c} holds due to \eqref{1-t-2-2},  \eqref{5-15} and  \eqref{5-5-H^4-1-1-c}.
\end{proof}
\subsection{Proof of Theorem 2.6}\noindent
\par In this section, we  establish     the  nonlinear structural stability of
the background supersonic flow within the class of axisymmetric flows. The proof is divided into three steps.
 \par { \bf Step 1. The  existence and uniqueness of a fixed point of $\mb_1^{\h {\bf W}}$.}
 \par Fix $\epsilon_0\in(0,R_0)$ and  $ r_1\in (r_0,\bar r_1^\star) $ for  $\bar r_1^\star$ from Lemma \ref{pro5-3}.  Assume that
\begin{equation}\label{5-44}
 \max\{\delta_1^\star,\delta_2^\star\}\leq \delta_5^\star,
 \end{equation}
 where $ \delta_5^\star $ is given by Lemma \ref{pro5-4}. For any fixed $ \h {\bf W}\in \mj_{\delta_{1}^\star, r_1}$, one can define a mapping
\begin{equation*}
\mb_1^{\h {\bf W}}(\h{\bf V})={\bf V}, \quad{\rm{ for \ each}} \ \ \h{\bf V}\in \mj_{\delta_{2}^\star, r_1}.
\end{equation*}
Use the abbreviation
    \begin{equation*}
    \sigma_{v}=\sigma(b^\star,U_{ 1,en}^\star,U_{2, en}^\star,U_{3, en}^\star,K_{en}^\star,S_{en}^\star,E_{ en}^\star,\Phi_{ ex}^\star).
    \end{equation*}
   Then it follows from    \eqref{5-1-e}, \eqref{5-7-g} and \eqref{5-7-h}  that  $ (V_1,V_2,V_3)\in \left(H^3(\mn)\right)^2\times H^4(\mn)$ satisfying the estimate
\begin{equation}\label{6-1}
\begin{aligned}
\Vert(V_1,V_2)\Vert_{H^3(\mn)}+\Vert V_3\Vert_{H^4(\mn)}\leq
\mc_3^\star(\delta_1^\star+(\delta_2^\star)^2+\sigma_v),
\end{aligned}
\end{equation}
where the constant  $\mc_3^\star>0$ depends only on $(b_0, \A_0,U_{0}, S_{0},E_{0},r_0,r_1,\epsilon_0)$.
\par
Set
\begin{equation}\label{6-3}
 \delta_2^\star=4\mc_3^\star(\delta_1^\star+\sigma_v).
\end{equation}
Then if
\begin{equation}\label{6-4}
\delta_1^\star+\sigma_v\leq \frac{1}{16(\mc_3^\star)^2},
\end{equation}
one can follow from \eqref{6-1} to obtain that
\begin{equation*}
\begin{aligned}
\Vert(V_1,V_2)\Vert_{H^3(\mn)}+\Vert V_3\Vert_{H^4(\mn)}
\leq \mc_3^\star(\delta_1^\star+(\delta_2^\star)^2+\sigma_v)\leq \frac12\delta_2^\star.
\end{aligned}
\end{equation*}
This implies that the mapping $\mb_1^{\h {\bf W}}$ maps $\mj_{\delta_{2}^\star, r_1}$ into itself.
\par Next, for any fixed $ \h {\bf W}\in \mj_{\delta_{1}^\star, r_1}$, we will show that $\mb_1^{\h {\bf W}}$  is a
a contraction mapping in a low order norm
\begin{equation*}
\|{\bf V}\|_w:=\Vert(V_1,V_2)\Vert_{L^2(\mn)}+\Vert V_3\Vert_{H^1(\mn)}
\end{equation*}
so that one can find a unique fixed point by the contraction mapping theorem. Let $ \h{\bf V}^{(i)}\in \mj_{\delta_{2}^\star, r_1} $, $ i=1,2 $, one has $\mb_1^{\h {\bf W}}(\h{\bf V}^{(i)})={\bf V}^{(i)}.$ Set
   \begin{equation*}
  \h {\bf Y}=(\h Y_1,\h Y_2,\h Y_3) =\h{\bf V}^{(1)}-\h{\bf V}^{(2)}, \quad  {\bf Y}=( Y_1, Y_2, Y_3) ={\bf V}^{(1)}-{\bf V}^{(2)}.
  \end{equation*}
  Then $   {\bf Y} $ satisfies
  \begin{equation}\label{6-6}
\begin{cases}
B_{11}(r,z,\h{\bf V}^{(1)},\h{\bf W})\p_r Y_1+ B_{22}(r,z,\h{\bf V}^{(1)},\h{\bf W})\p_z Y_2
+B_{12}(r,z,\h{\bf V}^{(1)})\p_rY_2\\
+B_{21}(r,z,\h{\bf V}^{(1)})\p_z Y_1
+\x a_1(r)Y_1+\x b_1(r)\p_r Y_3+\x b_2(r)Y_3=\mg_1(r,z,\h{\bf V}^{(1)},\h{\bf V}^{(2)},\h{\bf W}), \ \ &{\rm{in}} \ \ \mn,\\
\partial_r Y_2-\partial_{z} Y_1=\mg_2(r,z,\h{\bf V}^{(1)},\h{\bf V}^{(2)},\h{\bf W}), \ \ &{\rm{in}} \ \ \mn,\\
\bigg(\p_r^2+\frac 1 r\p_r+\p_{z}^2\bigg)Y_3+\x a_2(r)Y_1-\x b_3(r)Y_3=\mg_3(r,z,\h{\bf V}^{(1)},\h{\bf V}^{(2)},\h{\bf W}), \ \ &{\rm{in}} \ \ \mn,\\
Y_1(r_0,z)= Y_2(r_0,z)=\p_r Y_3(r_0,z)=0,\ \  &{\rm{on}} \ \ \Sigma_{en},\\
Y_3(r_1,z)=0, \ \ &{\rm{on}}\ \  \Sigma_{ex},\\
Y_2(r,\pm 1)=\p_z Y_3(r,\pm 1)=0,  \ \  &{\rm{on}}\ \ \Sigma_{1}^\pm,\\
\end{cases}
\end{equation}
where
\begin{equation*}
\begin{aligned}
&\mg_1(r,z,\h{\bf V}^{(1)},\h{\bf V}^{(2)},\h{\bf W})\\
&=-\bigg(\left(B_{11}(r,z,\h{\bf V}^{(1)},\h {\bf W})-B_{11}(r,z,\h{\bf V}^{(1)},\h {\bf W})\right)\p_r V_1^{(2)}
 + \left(B_{22}(r,z,\h{\bf V}^{(1)},\h {\bf W})-B_{22}(r,z,\h{\bf V}^{(2)},\h {\bf W})\right)\p_z V_2^{(2)}\\
& \qquad +\left(B_{12}(r,z,\h{\bf V}^{(1)})-B_{12}(r,z,\h{\bf V}^{(2)})\right)\p_rV_2^{(2) } +\left(B_{21}(r,z,\h{\bf V}^{(1)})-B_{21}(r,z,\h{\bf V}^{(2)})\right)\p_z V_1^{(2)}\bigg)\\
&\quad +G_1(r,z,\h{\bf V}^{(1)},\h {\bf W})-G_1(r,z,\h{\bf V}^{(2)},\h {\bf W}),\\
&\mg_2(r,z,\h{\bf V}^{(1)},\h{\bf V}^{(2)},\h{\bf W})=G_2(r,z,\h{\bf V}^{(1)},\h {\bf W})-G_2(r,z,\h{\bf V}^{(2)},\h {\bf W}),\\
&\mg_3(r,z,\h{\bf V}^{(1)},\h{\bf V}^{(2)},\h{\bf W})=G_3(r,z,\h{\bf V}^{(1)},\h {\bf W})-G_3(r,z,\h{\bf V}^{(2)},\h {\bf W}).\\
\end{aligned}
\end{equation*}
\par  Next, to simplicity the notation, denote
\begin{equation*}
\begin{aligned}
&\h B_{11}^{(1)}(r,z):=B_{11}(r,z,\h{\bf V}^{(1)},\h W_1),  \ \ \h B_{22}^{(1)}(r,z):=B_{22}(r,z,\h{\bf V}^{(1)},\h W_1), \ & (r,z)\in \mn,\\
&\h B_{12}^{(1)}(r,z):=B_{12}(r,z,\h{\bf V}^{(1)}), \qquad \ \h B_{21}^{(1)}(r,z):=B_{21}(r,z,\h{\bf V}^{(1)}),\ & (r,z)\in \mn,\\
&\h\mg_1(r,z):=\mg_1(r,z,\h{\bf V}^{(1)},\h{\bf V}^{(2)},\h {\bf W}),\ & (r,z)\in \mn,\\
&\h\mg_2(r,z):=\mg_2(r,z,\h{\bf V}^{(1)},\h{\bf V}^{(2)},\h {\bf W}),\ & (r,z)\in \mn,\\
&\h\mg_3(r,z):=\mg_2(r,z,\h{\bf V}^{(1)},\h{\bf V}^{(2)},\h {\bf W}), \ & (r,z)\in \mn.
\end{aligned}
\end{equation*}
In order to estimate  $   {\bf Y} $, one can decompose $Y_1$ and $Y_2$ as
\begin{equation*}
Y_1=-\p_{z} \varphi_2+\p_r \psi_2,\ \ \ Y_2= \p_r\varphi_2 + \p_{z}\psi_2, \ \ {\rm{in}}\ \ \mn,
\end{equation*}
where
$\varphi_2$ and $\psi_2$ solve the following boundary value problems, respectively:
\begin{equation}\label{6-7}
\begin{cases}
\partial_r^2 \varphi_2+\partial_{z}^2 \varphi_2=\h\mg_2 ,\ \ &{\rm{in}}\ \ \mn,\\
 \p_r\varphi_2(r_0,z)=0, &{\rm{on}}\ \ \Sigma_{en},\\
 \p_r\varphi_2(r_1,z)=0, &{\rm{on}}\ \ \Sigma_{ex},\\
 \varphi_2(r,\pm 1)=0, &{\rm{on}}\ \ \Sigma_{1}^\pm,\\
 \end{cases}
\end{equation}
and
\begin{equation}\label{6-6-1}
\begin{cases}
\h B_{11}^{(1)}\p_r^2\psi_2+ \h B_{22}^{(1)}\p_z^2\psi_2
+2\h B_{12}^{(1)}\p_{rz}\psi_2
+\x a_1\p_r\psi_2
+\x b_1\p_r Y_3+\x b_2Y_3\\
=\h\mg_1-\bigg(-\h B_{11}^{(1)}\p_{rz}^2 \varphi_2
+ \h B_{22}^{(1)}\p_{rz}^2 \varphi_2
 +\h B_{12}^{(1)}(\p_r^2\varphi_2-\p_z^2 \varphi_2)
-\x a_1\p_z\varphi_2\bigg), \ \ &{\rm{in}} \ \ \mn,\\
\bigg(\p_r^2+\frac 1 r\p_r+\p_{z}^2\bigg)Y_3+\x a_2\p_r \psi_2-\x b_3Y_3
=\h\mg_3
+\x a_2\p_z \varphi_2, \ \ &{\rm{in}} \ \ \mn,\\
\p_r \psi_2(r_0,z)=\p_z \varphi_2(r_0,z),  \ \psi_2(r_0,z)=\p_r Y_3(r_0,z)=0,\ \  &{\rm{on}} \ \ \Sigma_{en},\\
Y_3(r_1,z)=0, \ \ &{\rm{on}}\ \  \Sigma_{ex},\\
\p_z \psi_2(r,\pm 1)=\p_z Y_3(r,\pm 1)=0,  \ \  &{\rm{on}}\ \ \Sigma_{1}^\pm.\\
\end{cases}
\end{equation}
Then  similar arguments  as for \eqref{5-12}  yield
\begin{equation*}
\begin{aligned}
&\|\varphi_2\|_{H^2(\mn)}\leq C\|\h \mg_2\|_{L^2(\mn)}\leq C(\delta_1^\star+\delta_2^\star)\left(\|(\h Y_1,\h Y_2)\|_{L^2(\mn)}+\|\h Y_3\|_{H^1(\mn)}\right).
\end{aligned}
\end{equation*}
Then one can  use  Lemma \ref{pro5-3} and combine  the $H^2(\mn)$ estimate of $\varphi_2$  to  obtain that
\begin{equation*}
\begin{aligned}
\|(\psi_2,Y_3)\|_{H^1(\mn)}&\leq C\left(\|\h\mg_1\|_{L^2(\mn)}+\|\h\mg_3\|_{L^2(\mn)}
+\|\varphi_2\|_{H^2(\mn)}+\|\p_{z}\varphi_2(r_0,\cdot)\|_{L^2([-1,1])}
\right)\\
&\leq C(\delta_1^\star+\delta_2^\star) \left(\|(\h Y_1,\h Y_2)\|_{L^2(\mn)}+\|\h Y_3\|_{H^1(\mn)}\right).
\end{aligned}
\end{equation*}
 Collecting the above estimates leads to
\begin{equation}\label{6-9}
\begin{aligned}
\|(Y_1,Y_2)\|_{L^2(\mn)}+\|Y_3\|_{H^1(\mn)}\leq \mc_4^\star(\delta_1^\star+\delta_2^\star) \left(\|(\h Y_1,\h Y_2)\|_{L^2(\mn)}+\|\h Y_3\|_{H^1(\mn)}\right)
\end{aligned}
\end{equation}
  for the constant $\mc_4^\star>0$ depending only on $(b_0, \A_0,U_{0}, S_{0},E_{0},r_0,r_1,\epsilon_0)$.
 This, together with \eqref{6-3}, yields that
 \begin{equation*}
\begin{aligned}
\|{\bf Y}\|_w \leq 4\mc_4^\star(\mc_3^\star+1)(\delta_1^\star+\sigma_v) \|\h{\bf Y}\|_w.
\end{aligned}
\end{equation*}
 If $\delta_1^\star $ and $\sigma_v$ satisfy
\begin{equation}\label{6-10}
\mc_4^\star(\mc_3^\star+1)(\delta_1^\star+\sigma_v)\leq \frac18,
\end{equation}
then  one can conclude that  $\mb_1^{\h {\bf W}}$ has a unique fixed point in $\mj_{\delta_{2}^\star, r_1}$ provided that the conditions \eqref{6-4} and \eqref{6-10} hold.
\par { \bf Step 2. The  existence and uniqueness of a fixed point of $\mb_2$.}
\par For any fixed $ \h {\bf W}\in \mj_{\delta_{1}^\star, r_1}$,  { \bf Step 1} shows that  $\mb_1^{\h {\bf W}}$ has a unique fixed point in $ \h{\bf V}\in \mj_{\delta_{2}\star, r_1}$ provided that the conditions \eqref{6-4} and \eqref{6-10} hold.  Note that
the fixed point $\h {\bf V}$ solves the nonlinear boundary value problem:
\begin{equation}\label{6-11}
\begin{cases}
B_{11}(r,z,\h{\bf V},\h{\bf W})\p_r \h V_1+ B_{22}(r,z,\h{\bf V},\h{\bf W})\p_z \h V_2
+B_{12}(r,z,\h{\bf V})\p_r\h V_2+B_{21}(r,z,\h{\bf V})\p_z \h V_1\\
+\x a_1(r)\h V_1+\x b_1(r)\p_r \h V_3+\x b_2(r)\h V_3=G_1(r,z,\h{\bf V},\h{\bf W}), \ \ &{\rm{in}} \ \ \mn,\\
\partial_r \h V_2-\partial_{z} \h V_1=G_2(r,z,\h {\bf V},\h {\bf W}), \ \ &{\rm{in}} \ \ \mn,\\
\bigg(\p_r^2+\frac 1 r\p_r+\p_{z}^2\bigg)\h V_3+\x a_2(r)\h V_1-\x b_3(r)\h V_3=G_3(r,z,\h{\bf V},\h{\bf W}), \ \ &{\rm{in}} \ \ \mn,\\
(\h V_1, \h V_2,\p_r \h V_3)(r_0,z)=(U_{1,en}^\star,   U_{3,en}^\star,E_{en}^\star)(z)
-(U_0,0,E_0),\ \  &{\rm{on}} \ \ \Sigma_{en},\\
\h V_3(r_1,z)=\Phi_{ex}^\star(z)-\bar\Phi(r_1), \ \ &{\rm{on}}\ \  \Sigma_{ex},\\
\h V_2(r,\pm 1)=\p_z \h V_3(r,\pm 1)=0,  \ \  &{\rm{on}}\ \ \Sigma_{1}^\pm.\\
\end{cases}
\end{equation}
Therefore, $( \h U_1, \h U_3, \h \Phi)=(\h V_1, \h V_2, \h V_3)+(\bar U, 0, \bar \Phi)\in (H^3(\mn))^2\times H^4(\mn) $ solves
\begin{equation}\label{6-12}
\begin{cases}
(c^2(\h{\bf W}+{\bf W}_0,  \h U_1, \h U_3,\h\Phi)- \h U_1^2)\p_r  \h U_1+(c^2(\h{\bf W}+{\bf W}_0, \h U_1,  \h U_3,\h\Phi)- \h U_3^2)\p_z \h U_3\\
-{\h U_1\h U_3}\bigg(\p_r\h U_3+\p_z \h U_1\bigg)
+\frac {c^2(\h{\bf W}+{\bf W}_0,\h U_1,  \h U_3,\h \Phi)\h U_1}{r}+(\h U_1\p_r+\h U_3\p_z)\h\Phi
=0,\ \ &{\rm{in}}\ \ \mn,\\
{\h U_1}(\partial_r \h U_3-\partial_{z}\h U_1)=\h W_1\p_z \h W_1+\frac{e^{\h W_2+S_0} }{\gamma-1}\mh^{\gamma-1}(\h {\bf W}+{\bf W}_0,\h U_1, \h U_3,\h\Phi){\p_z \h W_3}-\p_z \h W_2 ,\ \ &{\rm{in}}\ \ \mn,\\
\bigg(\p_r^2+\frac 1 r\p_r+\frac{1}{r^2}\p_{\th}^2\bigg) \h \Phi=\mh(\h {\bf W}+{\bf W}_0,\h U_1, \h U_3,\h\Phi)-b^\star,\ \ &{\rm{in}}\ \ \mn,\\
(\h U_1, \h U_3,  \p_r\h\Phi)(r_0,z)=(U_{1,en}^\star, U_{3,en}^\star,E_{en}^\star)(z),\ \ &{\rm{on}}\ \ \Sigma_{en},\\
\h\Phi(r_1,z)=\Phi_{ex}^\star(z), \ \ &{\rm{on}}\ \ \Sigma_{ex},\\
\h U_3(r,\pm 1)=\p_z\h\Phi(r,\pm 1)=0, \ \  &{\rm{on}}\ \ \Sigma_{1}^\pm.\\
\end{cases}
\end{equation}
Furthermore, one can follow from \eqref{5-3} to derive that $ \mh(\h {\bf W}+{\bf W}_0,\h U_1, \h U_3,\h\Phi)\in H^3(\mn) $. Applying the Morrey inequality yields
\begin{equation*}
\|\h\Phi\|_{C^{2,\frac12}(\overline\mn)}\leq C\|\h\Phi\|_{H^{4}(\mn)} \ \ {\rm{and}} \ \ \|\mh\|_{C^{1,\frac12}(\overline\mn)}\leq C\|\mh\|_{H^{3}(\mn)}.
\end{equation*}
Then we consider the following problem:
\begin{equation}\label{6-12-f}
\begin{cases}
\bigg(\p_r^2+\frac 1 r\p_r+\frac{1}{r^2}\p_{\th}^2\bigg) \h \Phi=\mh(\h {\bf W}+{\bf W}_0,\h U_1, \h U_3,\h\Phi)-b^\star,\ \ &{\rm{in}}\ \ \mn,\\
  \p_r\h\Phi(r_0,z)=E_{en}^\star(z),\ \ &{\rm{on}}\ \ \Sigma_{en},\\
  \h\Phi(r_1,z)=\Phi_{ex}^\star(z), \ \ &{\rm{on}}\ \ \Sigma_{ex},\\
  \p_z\h\Phi(r,\pm 1)=0, \ \  &{\rm{on}}\ \ \Sigma_{1}^\pm.\\
  \end{cases}
\end{equation}
It follows from \eqref{1-t-2-2}, \eqref{5-8}  and \eqref{5-1-c} that $ \p_z(\mh,b^\star)=0 $ on $ \Sigma_{1}^\pm $ and $ (E_{en}^\star,\Phi_{ex}^\star)'(\pm 1)=0 $. By the standard Schauder
estimate and the method of reflection, we obtain that
\begin{equation}\label{6-12-f-f}
\|\h\Phi\|_{C^{3,\frac12}(\overline{\mn})}\leq C\bigg(\|\h\Phi\|_{H^{4}(\mn)}+\|\mh\|_{H^{3}(\mn)}
+\|b^\star\|_{C^{2}(\overline\mn)}+\|(E_{ en}^\star,\Phi_{ ex}^\star)\|_{C^4([-1,1])}\bigg).
\end{equation}
\par Next, for any $\h {\bf W}\in\mj_{\delta_{1}^\star, r_1} $, we construct a mapping  $\mb_2$: $\h {\bf W}\in\mj_{\delta_{1}^\star, r_1} \mapsto  {\bf W}\in \mj_{\delta_{1}^\star, r_1}$, where $ {\bf W}=(W_1,W_2, W_3)$ solves the following transport equations, respectively:
\begin{equation}\label{6-13}
\begin{cases}
\begin{aligned}
&\mh(\h{\bf W}+{\bf W}_0,\h U_1,  \h U_3,\h \Phi)\bigg(\h U_1\partial_r +{\h U_3}\partial_{z}\bigg) (rW_1)=0, \ \ {\rm{in}}\ \ \mn,\\
&W_1(r_0,z)=  U_{2,en}^\star(z), \qquad\qquad\qquad \qquad\qquad \quad \ \ \ \ {\rm{on}}\ \ \Sigma_{en},\\
 \end{aligned}
\end{cases}
\end{equation}
and
\begin{equation}\label{6-14}
\begin{cases}
\begin{aligned}
&\mh(\h{\bf W}+{\bf W}_0,\h U_1,  \h U_3,\h \Phi)\bigg(\h U_1\partial_r +{\h U_3}\partial_{z}\bigg)  W_2=0,\ \ {\rm{in}}\ \ \mn,\\
&W_2(r_0,z)=  K_{en}^\star(z)-K_0, \qquad\qquad\qquad \qquad \quad  \ {\rm{on}}\ \ \Sigma_{en},\\
 \end{aligned}
\end{cases}
\end{equation}
and
\begin{equation}\label{6-15}
\begin{cases}
\begin{aligned}
&\mh(\h{\bf W}+{\bf W}_0,\h U_1,  \h U_3,\h \Phi)\bigg(\h U_1\partial_r +{\h U_3}\partial_{z}\bigg)  W_3=0,\ \ {\rm{in}}\ \ \mn,\\
 &W_3(r_0,z)=S_{en}^\star(z)-S_0,\qquad\qquad\qquad \qquad \quad \ \ {\rm{on}}\ \ \Sigma_{en}.\\
\end{aligned}
\end{cases}
\end{equation}
\par It follows from the first equation in \eqref{6-12} that
\begin{equation}\label{6-16}
\partial_r\bigg(r\mh(\h {\bf W}+{\bf W}_0,\h U_1, \h U_3,\h\Phi) \h U_1\bigg)+\partial_{z}\bigg(r\mh(\h {\bf W}+{\bf W}_0,\h U_1, \h U_3,\h\Phi) \h U_3\bigg)=0,\ \ {\rm{in}}\ \ \mn,
\end{equation}
from which  one can define a stream function on $[r_0,r_1]\times [-1,1]$ as
\begin{equation*}
\mathscr{L}(r,z)=\int_{-1}^{z}r_0\left(\mh(\h {\bf W}+{\bf W}_0,\h U_1, \h U_3,\h\Phi) \h U_1\right)(r_0,s)\de s
-\int_{r_0}^r s\left(\mh(\h {\bf W}+{\bf W}_0,\h U_1, \h U_3,\h\Phi) \h U_3\right)(s,z) \de s.
\end{equation*}
Then  the function $\mathscr{L}$ has the properties
\begin{equation*}\begin{cases}
\p_r\mathscr{L}(r,z)=-r\left(\mh(\h {\bf W}+{\bf W}_0,\h U_1, \h U_3,\h\Phi) \h U_3\right)(r,z)\in H^3(\mn),\\
\p_z \mathscr{L}(r,z)=r\left(\mh(\h {\bf W}+{\bf W}_0,\h U_1, \h U_3,\h\Phi) \h U_1\right)(r,z)\in H^3(\mn),\\
\mathscr{L}(r,z)\in  H^4(\mn).\\
\end{cases}
\end{equation*}
Since $\p_{r}\mathscr{L}(r,\pm 1)=0$, one has $\mathscr{L}(r,\pm 1)=\mathscr{L}(r_0,\pm 1)$. Note that $\bar{U}(r)>0$ for each $r\in [r_0,r_1]$, and $\h{{\bf V}}\in \mj_{\delta_{2}^\star, r_1}$,  hence $\p_z\mathscr{L}(r,z)=r\left(\mh(\h {\bf W}+{\bf W}_0,\h U_1, \h U_3,\h\Phi)(\bar U+\h V_1)\right)(r,z)>0$, from which one obtains $\mathscr{L}(r,z)$ is an strictly increasing function of $z$ for each fixed $r\in [r_0, r_1]$. These imply that the closed interval $[\mathscr{L}(r,-1),\mathscr{L}(r,1)]$ is simply equal to $[\mathscr{L}(r_0,-1), \mathscr{L}(r_0,1)]$. Denote the inverse function of $\mathscr{L}(r_0,\cdot): [-1,1]\to [\mathscr{L}(r_0,-1), \mathscr{L}(r_0,1)]$ by
$\mathscr{L}_{r_0}^{-1}(\cdot): [\mathscr{L}(r_0,-1), \mathscr{L}(r_0,-1)]\to [-1,1]$.
\par  Define
\begin{equation}\label{6-16-1}
\begin{cases}
\begin{aligned}
&W_1(r,z)=  \frac{r_0}{r}U_{2,en}^\star\left(\mathscr{L}_{r_0}^{-1}(\mathscr{L}(r,z))\right),\\
&W_2(r,z)=  K_{en}^\star\left(\mathscr{L}_{r_0}^{-1}(\mathscr{L}(r,z))\right)-K_0\\
&W_3(r,z)=  S_{en}^\star\left(\mathscr{L}_{r_0}^{-1}(\mathscr{L}(r,z))\right)-S_0.
\end{aligned}
\end{cases}
\end{equation}
Then one can easily verify that $(W_1,W_2,W_3)$ solves \eqref{6-13} and \eqref{6-14} and \eqref{6-15}, respectively. Furthermore, it follows from \eqref{5-3} and \eqref{6-12-f-f} that
\begin{equation*}
\begin{aligned}
 &\bigg(\mh(\h {\bf W}+{\bf W}_0,\h U_1, \h U_3,\h\Phi) \h U_1\bigg)(r_0,\cdot)\\
 &=\left(\bigg(\frac{\gamma-1}{\gamma e^{S_{en}^\star}}\bigg)^{\frac{1}{\gamma-1}}\bigg(K_{en}^\star+\h\Phi-\frac{1}{2}\left(U_{ 1,en}^\star)^2+(U_{ 2,en}^\star)^2+(U_{ 3,en}^\star)^2\right)\bigg)
^{\frac{1}{\gamma-1}}U_{ 1,en}^\star\right)(r_0,\cdot)\in C^{3}([-1,1]).
\end{aligned}
\end{equation*}
  Then one derives $\mathscr{L}_{r_0}^{-1}(\cdot)\in
C^{4}([\mathscr{L}(r_0,-1), \mathscr{L}(r_0,1)])$ and
\begin{equation}\label{6-17}
\begin{cases}
\begin{aligned}
&\|W_1\|_{H^4(\mn)}\leq \mc_5^\star \|U_{2,en}^\star\|_{C^4([-1,1])},\\
&\|W_2\|_{H^4(\mn)}\leq \mc_5^\star \|K_{ en}^\star-K_0\|_{C^4([-1,1])},\\
&\|W_3\|_{H^4(\mn)}\leq \mc_5^\star \|S_{ en}^\star-S_0\|_{C^4([-1,1])},\\
\end{aligned}
\end{cases}
\end{equation}
where the constant $\mc_5^\star>0$ depends only on $(b_0, \A_0,U_{0}, S_{0},E_{0},r_0,r_1,\epsilon_0)$.
Moreover, it follows from \eqref{1-t-2-2} that  \begin{equation}
\label{6-19}
\p_{z}^k W_1=\p_{z}^k W_2=\p_{z}^k W_3=0, \ \ {\rm{on}} \ \Sigma_{1}^\pm,  \ \  {\rm{for}} \ \ k=1,3.
\end{equation}
    \par  Define another iteration mapping
    \begin{equation*}
\mb_2(\h{\bf W})={\bf W}, \quad{\rm{ for \ each}} \ \ \h{\bf W}\in \mj_{\delta_{1}^\star, r_1}.
\end{equation*}
Set
\begin{equation}
\label{6-20}
\delta_1^\star= 6\mc_5^\star\sigma_v.
\end{equation}
Then \eqref{6-17} yields that
\begin{equation*}
\|{\bf W}\|_{H^4(\mn)}\leq 3\mc_5^\star\sigma\leq \frac12 \delta_1.
\end{equation*}
This shows that the mapping $\mb_2$ maps $\mj_{\delta_{1}^\star, r_1}$ into itself.
It remains to show that the mapping $\mb_2$ is contractive in a low order norm for suitably small $\sigma_v$. Let ${\bf W}^{(i)}=\mb_2(\h {\bf W}^{(i)}) $ for
any $\h{\bf W}^{(i)}\in \mj_{\delta_{1}^\star, r_1}, (i=1,2)$ and denote
$$\h {\bf W}^{d}=\h {\bf W}^{(1)}-\h {\bf W}^{(2)}\ \ {\rm{ and }}\ \ {\bf W}^{d}={\bf W}^{(1)}-{\bf W}^{(2)}. $$
Furthermore, set $$( \h U_1^{(i)}, \h U_3^{(i)}, \h \Phi^{(i)})=(\h V_1^{(i)}, \h V_2^{(i)}, \h V_3^{(i)})+(\bar U, 0, \bar \Phi),$$ where $ \h {\bf V}^{(i)} $ is the unique fixed point of $\mb_1^{(\h {\bf W}^{(i)})}$ for $i=1,2$.
 Then it follows from \eqref{6-16-1} that
\begin{equation}\label{6-21}
\begin{cases}
\begin{aligned}
&W_1^{(i)}(r,z)=  \frac{r_0}{r}U_{2,en}^\star\left(\mathscr{T}^{(i)}(r,z)\right),\\
&W_2^{(i)}(r,z)=  K_{en}^\star\left(\mathscr{T}^{(i)}(r,z)\right)-K_0\\
&W_3^{(i)}(r,z)=  S_{en}^\star\left(\mathscr{T}^{(i)}(r,z)\right)-S_0.
\end{aligned}
\end{cases}
\end{equation}
with $ \mathscr{T}^{(i)}(r,z)=(\mathscr{L}_{r_0}^{(i)})^{-1}(\mathscr{L}^{(i)}(r,z)) $ and
\begin{equation*}
\begin{aligned}
\mathscr{L}^{(i)}(r,z)&=\int_{-1}^{z}r_0\left(\mh(\h {\bf W}^{(i)}+{\bf W}_0,\h U_1^{(i)}, \h U_3^{(i)},\h\Phi^{(i)}) \h U_1^{(i)}\right)(r_0,s)\de s\\
&\quad-\int_{r_0}^r s\left(\mh(\h {\bf W}^{(i)}+{\bf W}_0,\h U_1^{(i)}, \h U_3^{(i)},\h\Phi^{(i)}) \h U_3^{(i)}\right)(s,z) \de s.
\end{aligned}
\end{equation*}
Here $(\mathscr{L}_{r_0}^{(i)})^{-1}$: $t\in [\mathscr{L}^{(i)}(r_0,-1), \mathscr{L}^{(i)}(r_0,1)]\mapsto z\in [-1,1]$ is the inverse function of $\mathscr{L}^{(i)}(r_0,\cdot)$: $z\in [-1,1]\mapsto t\in [\mathscr{L}^{(i)}(r_0,-1), \mathscr{L}^{(i)}(r_0,1)]$,
Therefore,
 one obtains
\begin{equation*}
|W_1^d|=|W_1^{(1)}-W_1^{(2)}|\leq \|(U_{2,en}^\star)'\|_{L^\infty([-1,1])}|\mathscr{T}^{(1)}(r,z)-\mathscr{T}^{(2)}(r,z)|.
\end{equation*}
     Note that $ \mathscr{L}^{(i)}_{r_0}\left(\mathscr{T}^{(i)}(r,z)\right)
     =\mathscr{L}^{(i)}(r,z)$. Then it holds that
     \begin{align*}
&\int_{\mathscr{T}^{(2)}(r,z)}^{\mathscr{T}^{(1)}(r,z)}
r_0\left(\mh(\h {\bf W}^{(1)}+{\bf W}_0,\h U_1^{(1)}, \h U_3^{(1)},\h\Phi^{(1)})\h U_1^{(1)}\right)(r_0,s) \de s
=\mathscr{L}^{(1)}(r,z)-\mathscr{L}^{(2)}(r,z)\\
&\quad-\int_{-1}^{\mathscr{T}^{(2)}(r,z)}
r_0\left(\mh(\h {\bf W}^{(1)}+{\bf W}_0,\h U_1^{(1)}, \h U_3^{(1)},\h\Phi^{(1)})\h U_1^{(1)}-\mh(\h {\bf W}^{(2)}+{\bf W}_0,\h U_1^{(2)}, \h U_3^{(2)},\h\Phi^{(2)})\h U_1^{(2)}\right)(r_0,s)\de s,
\end{align*}
from which one obtains
\begin{align*}
&\mathfrak{m}^{(1)}|\mathscr{T}^{(1)}(r,z)-\mathscr{T}^{(2)}(r,z)|
\leq |\mathscr{L}^{(1)}(r,z)-\mathscr{L}^{(2)}(r,z)|\\
&\quad+\int_{-1}^{1}r_0\left|\mh(\h {\bf W}^{(1)}+{\bf W}_0,\h U_1^{(1)}, \h U_3^{(1)},\h\Phi^{(1)})\h U_1^{(1)}-\mh(\h {\bf W}^{(2)}+{\bf W}_0,\h U_1^{(2)}, \h U_3^{(2)},\h\Phi^{(2)})\h U_1^{(2)}\right|(r_0,s)\de s
\end{align*}
with $\mathfrak{m}^{(i)}:=\displaystyle\min_{z\in[-1,1]}r_0(\mh(\h {\bf W}^{(i)}+{\bf W}_0,\e U_1^{(i)}, \h U_3^{(i)},\h\Phi^{(i)})\h U_1^{(i)})(r_0,z)>0$. Noting that
\begin{equation*}
\begin{aligned}
&(\h W_1^{(1)}-\h W_1^{(2)})(r_0,z)=(\h W_2^{(1)}-\h W_2^{(2)})(r_0,z)=(\h W_3^{(1)}-\h W_3^{(2)})(r_0,z)\\
&=(\h U_1^{(1)}-\h U_1^{(2)})(r_0,z)=(\h U_3^{(1)}-\h U_3^{(2)})(r_0,z)\equiv 0.
\end{aligned}
\end{equation*}
Then one gets
\begin{equation}\label{6-22}
\|W_1^d\|_{L^{2}(\mn)}\leq C\sigma_v\bigg(\|(\h{{\bf V}}^{(1)}-\h{{\bf V}}^{(2)},\h {\bf W}^d)\|_{L^{2}(\mn)}+\|(\h V_{3}^{(1)}-\h V_{3}^{(2)})(r_0,\cdot)\|_{L^{2}([-1,1])}\bigg).
\end{equation}
Furthermore, there holds
\begin{equation*}
\begin{aligned}
|\p_{r} W_1^d|&=\frac{r_0}{r}\left|(U_{2,en}^\star)'\left(\mathscr{T}^{(1)}(r,z)\right)\p_r \mathscr{T}^{(1)}(r,z)-(U_{2,en}^\star)'\left(\mathscr{T}^{(2)}(r,z)\right)\p_r \mathscr{T}^{(2)}(r,z)\right|\\
&\quad-\frac{r_0}{r^{2}}\left|U_{2,en}^\star\left(\mathscr{T}^{(1)}(r,z)\right)-
U_{2,en}^\star\left(\mathscr{T}^{(2)}(r,z)\right)\right|\\
&=\frac{r_0}{r}\left|\left((U_{2,en}^\star)'\left(\mathscr{T}^{(1)}(r,z)\right)-
(U_{2,en}^\star)'\left(\mathscr{T}^{(2)}(r,z)\right)\right)\p_r \mathscr{T}^{(1)}+ (U_{2,en}^\star)'\left(\mathscr{T}^{(2)}(r,z)\right)(\p_r \mathscr{T}^{(1)}-\p_r
\mathscr{T}^{(2)})\right|\\
&\quad-\frac{r_0}{r^{2}}\left|U_{2,en}^\star\left(\mathscr{T}^{(1)}(r,z)\right)-
U_{2,en}^\star\left(\mathscr{T}^{(2)}(r,z)\right)\right|\\
&\leq \|(U_{2,en}^\star)''\|_{L^\infty([-1,1])}
\left|\mathscr{T}^{(1)}(r,z)-\mathscr{T}^{(2)}(r,z)\right|\frac{1}{\mathfrak{m}^{(1)}}\|\nabla
\mathscr{L}^{(1)}(r,z)\|_{L^\infty(\mn)}\\
&\quad+ \|(U_{2,en}^\star)'\|_{L^\infty([-1,1])}\frac{\|\nabla
\mathscr{L}^{(1)}\|_{L^\infty(\mn)}}{\mathfrak{m}^{(1)}\mathfrak{m}^{(2)}}
\bigg|\left(\mh(\h {\bf W}^{(1)}+{\bf W}_0,\h U_1^{(1)}, \h U_3^{(1)},\h\Phi^{(1)})\h U_1^{(1)}\right)(r_0,\mathscr{T}^{(1)})
\\
&\quad-\left(\mh(\h {\bf W}^{(2)}+{\bf W}_0,\h U_1^{(2)}, \h U_3^{(2)},\h\Phi^{(2)})\h U_1^{(2)}\right)(r_0,\mathscr{T}^{(2)})\bigg|+ \|(U_{2,en}^\star)'\|_{L^\infty([-1,1])}\left|\mathscr{T}^{(1)}(r,z)
-\mathscr{T}^{(2)}(r,z)\right|,
\end{aligned}
\end{equation*}
and similar computations are valid for $\p_{z} W_1^d$. Then one has
\begin{equation}\label{6-23}
\|\nabla W_1^d\|_{L^{2}(\mn)}\leq C\sigma_v\bigg(\|(\h{{\bf V}}^{(1)}-\h{{\bf V}}^{(2)},\h {\bf W}^d)\|_{L^{2}(\mn)}+\|(\h V_{3}^{(1)}-\h V_{3}^{(2)})(r_0,\cdot)\|_{L^{2}([-1,1])}\bigg).
\end{equation}
Collecting the estimates \eqref{6-22} and \eqref{6-23} leads to
\begin{align}\label{6-24}
\|W_1^d\|_{H^1(\mn)}\leq C\sigma_v\bigg(\|(\h{{\bf V}}^{(1)}-\h{{\bf V}}^{(2)},\h {\bf W}^d)\|_{L^{2}(\mn)}+\|(\h V_{3}^{(1)}-\h V_{3}^{(2)})(r_0,\cdot)\|_{L^{2}([-1,1])}\bigg).
\end{align}
Same estimate holds for $W_2^d$ and $W_3^d$. Therefore, one gets
\begin{align}\label{6-25}
\|{\bf W}^d\|_{H^1(\mn)}\leq \mc_6^\star
\sigma_v\bigg(\|(\h{{\bf V}}^{(1)}-\h{{\bf V}}^{(2)},\h {\bf W}^d)\|_{L^{2}(\mn)}+\|(\h V_{3}^{(1)}-\h V_{3}^{(2)})(r_0,\cdot)\|_{L^{2}([-1,1])}\bigg)
\end{align}
 for the constant $\mc_6^\star>0$ depending only on $(b_0, \A_0,U_{0}, S_{0},E_{0},r_0,r_1,\epsilon_0)$.
\par
In the following, it remains to estimate
\begin{align*}
\|\h{{\bf V}}^{(1)}-\h{{\bf V}}^{(2)}\|_{L^{2}(\mn)}+\|(\h V_{3}^{(1)}-\h V_{3}^{(2)})(r_0,\cdot)\|_{L^{2}([-1,1])}.
\end{align*}
Indeed, set ${\bf R}=\h{{\bf V}}^{(1)}-\h{{\bf V}}^{(2)}$. It follows from \eqref{6-11} that
\begin{equation}\label{6-26}
\begin{cases}
B_{11}(r,z,\h{\bf V}^{(1)},\h{\bf W}^{(1)})\p_r  R_1+ B_{22}(r,z,h{\bf V}^{(1)},\h{\bf W}^{(1)})\p_z  R_2\\
+B_{12}(r,z,\h{\bf V}^{(1)})\p_r R_2
+B_{21}(r,z,\h{\bf V}^{(1)})\p_z  R_1
+\x a_1(r) R_1\\
+\x b_1(r)\p_r  R_3+\x b_2(r) R_3=\mg_4(r,z,\h{\bf W}^{(1)},\h{\bf V}^{(1)},\h{\bf W}^{(2)},\h{\bf V}^{(2)}), \ \ &{\rm{in}} \ \ \mn,\\
\partial_r  R_2-\partial_{z}  R_1=\mg_5(r,z,\h{\bf W}^{(1)},\h{\bf V}^{(1)},\h{\bf W}^{(2)},\h{\bf V}^{(2)}), \ \ &{\rm{in}} \ \ \mn,\\
\bigg(\p_r^{2}+\frac 1 r\p_r+\p_{z}^{2}\bigg) R_3+\x a_2(r)R_1-\x b_3(r) R_3\\
=\mg_6(r,z,\h{\bf W}^{(1)},\h{\bf V}^{(1)},\h{\bf W}^{(2)},\h{\bf V}^{(2)}), \ \ &{\rm{in}} \ \ \mn,\\
R_1(r_0,z)= R_2(r_0,z)=\p_r R_3(r_0,z)=0,\ \  &{\rm{on}} \ \ \Sigma_{en},\\
R_3(r_1,z)=0, \ \ &{\rm{on}}\ \  \Sigma_{ex},\\
R_2(r,\pm 1)=\p_z R_3(r,\pm 1)=0,  \ \  &{\rm{on}}\ \ \Sigma_{1}^\pm,
\\
\end{cases}
\end{equation}
where
\begin{equation*}
\begin{aligned}
&\mg_4(r,z,\h{\bf W}^{(1)},\h{\bf V}^{(1)},\h{\bf W}^{(2)},\h{\bf V}^{(2)})\\
&=-\bigg(B_{11}(r,z,\h{\bf V}^{(1)},\h{\bf W}^{(1)})-B_{11}(r,z,\h{\bf V}^{(2)},\h{\bf W}^{(2)})\p_r \h V_1^{(2)}
 + B_{22}(r,z,\h{\bf V}^{(1)},\h{\bf W}^{(1)})-B_{22}(r,z,\h{\bf V}^{(2)},\h{\bf W}^{(2)})\p_z \h V_2^{(2)}\\
& \quad\quad +B_{12}(r,z,\h{\bf V}^{(1)})-B_{12}(r,z,\h{\bf V}^{(2)}))\p_r\h V_3^{(2)} +B_{21}(r,z,\h{\bf V}^{(1)})-B_{21}(r,z,\h{\bf V}^{(2)}))\p_z \h V_1^{(2)}\bigg) \\ &\quad+G_1(r,z,\h{\bf V}^{(1)},\h{\bf W}^{(1)})-G_1(r,z,\h{\bf V}^{(2)},\h{\bf W}^{(2)}),\\
&\mg_5(r,z,\h {\bf W}^{(1)},\h {\bf V}^{(1)},\h {\bf W}^{(2)},\h {\bf V}^{(2)})=G_2(r,z,\h{\bf V}^{(1)},\h{\bf W}^{(1)})-G_2(r,z,\h{\bf V}^{(2)},\h{\bf W}^{(2)}),\\
&\mg_6(r,z,\h {\bf W}^{(1)},\h {\bf V}^{(1)},\h {\bf W}^{(2)},\h {\bf V}^{(2)})=G_3(r,z,\h{\bf V}^{(1)},\h{\bf W}^{(1)})-G_3(r,z,\h{\bf V}^{(2)},\h{\bf W}^{(2)}).\\
\end{aligned}
\end{equation*}
Then it  holds that
\begin{equation*}
\begin{cases}
\|\mg_4(\cdot,\h{\bf W}^{(1)},\h{\bf V}^{(1)},\h{\bf W}^{(2)},\h{\bf V}^{(2)})\|_{L^{2}(\mn)}\leq C\bigg(\|\h {\bf W}^d\|_{L^{2}(\mn)}+\delta_2^\star\|{\bf R}\|_{L^{2}(\mn)}\bigg),\\
\|\mg_5(\cdot,\h{\bf W}^{(1)},\h{\bf V}^{(1)},\h{\bf W}^{(2)},\h{\bf V}^{(2)})\|_{L^{2}(\mn)}\leq C\bigg(\|\h {\bf W}^d\|_{L^{2}(\mn)}+\delta_1^\star\|{\bf R}\|_{L^{2}(\mn)}\bigg),\\
\|\mg_6(\cdot,\h{\bf W}^{(1)},\h{\bf V}^{(1)},\h{\bf W}^{(2)},\h{\bf V}^{(2)})\|_{L^{2}(\mn)}\leq C\bigg(\|\h {\bf W}^d\|_{L^{2}(\mn)}+\delta_2^\star\|{\bf R}\|_{L^{2}(\mn)}\bigg).\\
\end{cases}
\end{equation*}
\par Similar arguments as for \eqref{6-6} yield
\begin{equation}\label{6-28}
\|(R_1,R_2)\|_{L^{2}(\mn)}+ \|R_3\|_{H^1(\mn)}
\leq C\bigg(\|\h {\bf W}^d\|_{H^1(\mn)}+(\delta_1^\star+\delta_2^\star)\|{\bf R}\|_{L^{2}(\mn)}\bigg).
\end{equation}
By the trace theorem, one has
\begin{equation}\label{6-29}
\|R_3(r_0,\cdot)\|_{L^{2}([-1,1])}\leq C\bigg(\|\h {\bf W}^d\|_{H^1(\mn)}+(\delta_1^\star+\delta_2^\star)\|{\bf R}\|_{L^{2}(\mn)}\bigg).
\end{equation}
Therefore,  collecting the estimates \eqref{6-28}-\eqref{6-29} together with \eqref{6-3} and \eqref{6-20} leads to
\begin{equation}\label{6-30}
\begin{aligned}
&\|\h{{\bf V}}^{(1)}-\h{{\bf V}}^{(2)}\|_{L^{2}(\mn)}+\|(\h V_{3}^{(1)}-\h V_{3}^{(2)})(r_0,\cdot)\|_{L^{2}([-1,1])}\leq \mc_7^\star\bigg(\|\h {\bf W}^d\|_{H^1(\mn)}+(\delta_1^\star+\delta_2^\star)\|\h{{\bf V}}^{(1)}-\h{{\bf V}}^{(2)}\|_{L^{2}(\mn)}\bigg)\\
&
\leq \mc_7^\star\bigg(\|\h {\bf W}^d\|_{H^1(\mn)}+
\bigg(6\mc_5^\star\sigma_v+4\mc_3^\star(6\mc_5^\star\sigma_v+\sigma_v)
\bigg)\|\h{{\bf V}}^{(1)}-\h{{\bf V}}^{(2)}\|_{L^{2}(\mn)}\bigg)
\end{aligned}
\end{equation}
 for the constant $\mc_7^\star>0$ depending only on $(b_0, \A_0,U_{0}, S_{0},E_{0},r_0,r_1,\epsilon_0)$.
If  $ \sigma_v $ satisfies
\begin{align}\label{6-31}
\sigma_v  \leq \frac1{2\mc_7^\star\bigg(6\mc_5^\star+4\mc_3^\star(6\mc_5^\star+1)
\bigg)},
\end{align}
there holds
\begin{align*}
\|\h{{\bf V}}^{(1)}-\h{{\bf V}}^{(2)}\|_{L^{2}(\mn)}+\|(\h V_{3}^{(1)}-\h V_{3}^{(2)})(r_0,\cdot)\|_{L^{2}([-1,1])}\leq 2\mc_7^\star\|\h {\bf W}^d\|_{H^1(\mn)}.
\end{align*}
Furthermore, it  follows from \eqref{6-25} that
\begin{align*}
\|{\bf W}^d\|_{H^1(\mn)}\leq \mc_6^\star(2C_7^\star+1)
\sigma_v\|\h {\bf W}^d\|_{H^1(\mn)}.
\end{align*}
Next, if  $ \sigma_v $ satisfies
\begin{align}\label{6-33}
\mc_6^\star(2C_7^\star+1)
\sigma_v \leq \frac14,
\end{align}
Then $\mb_2$ is a contractive mapping in $H^1(\mn)$-norm and there exists a unique fixed point ${\bf W}\in \mj_{\delta_{1}^\star, r_1}$ provided that the conditions \eqref{6-4}, \eqref{6-10},  \eqref{6-31}  and \eqref{6-33} hold.
\par Now, under the choices of $( \delta_2^\star,\delta_{1}^\star)$ given by \eqref{6-3} and \eqref{6-20}, one can choose $\sigma_1^{\star}$ so that whenever $\sigma_v\in (0, \sigma^{\star}_1]$, the conditions \eqref{5-44}, \eqref{6-4}, \eqref{6-10}, \eqref{6-31}  and \eqref{6-33} hold.
\begin{enumerate}[ \rm(1)]
 \item By  \eqref{6-3} and \eqref{6-20}, the condition \eqref{5-44} holds if
 \begin{equation*}
 \sigma_v \leq \frac{\delta_5^\star}{4\mc_3^\star(6\mc_5^\star+1)+6\mc_5^\star}=:\sigma_1
 \end{equation*}
 for $ (\mc_3^\star,\mc_5^\star) $ from \eqref{6-1} and \eqref{6-17}.
 \item The conditions \eqref{6-4} and \eqref{6-10} hold if
 \begin{equation*}
 \sigma_v \leq \min\left\{\frac{1}{16(6\mc_5^\star+1)(\mc_3^\star)^2},
 \frac{1}{8(6\mc_5^\star+1)\mc_4^\star(\mc_3^\star+1)}
 \right\}=:\sigma_2
  \end{equation*}
   for $ \mc_4^\star$  from \eqref{6-9}.
   \item The conditions \eqref{6-31}  and \eqref{6-33} hold if
   \begin{equation*}
    \sigma_v \leq \min\left\{\frac1{2\mc_7^\star\bigg(6\mc_5^\star+4\mc_3^\star(6\mc_5^\star+1)
\bigg)},\frac{1}{4\mc_6^\star(2C_7^\star+1)}\right\}=:\sigma_3
    \end{equation*}
  for $ (\mc_6^\star,\mc_7^\star) $ from \eqref{6-25} and \eqref{6-30}.
  \end{enumerate}
  Therefore, we choose $\sigma_{1}^{\star}$ as
 \begin{equation}
 \label{sigma-bd-step2}
\sigma_{1}^{\star} =\min\left\{\sigma_k: k=1,2,3
  \right\},
\end{equation}
so that if
\begin{equation}\label{choice1-sigma-full}
\sigma_v \leq \sigma_{1}^{\star},
  \end{equation}
then all the conditions \eqref{5-44}, \eqref{6-4}, \eqref{6-10}, \eqref{6-31}  and \eqref{6-33} hold  under the choices of $( \delta_2^\star,\delta_{1}^\star)$ given by \eqref{6-3} and \eqref{6-20}.
\par { \bf Step 3. Uniqueness.}
\par
Under the assumption of \eqref{choice1-sigma-full},
let $(U_1^{i},U_2^{i},U_3^{i}, \Phi^{i}, K^{i}, S^{i})$  ($i=1$, $2$) be two solutions to Problem \ref{probl2} and assume that both solutions satisfy \eqref{1-t-3-2}-\eqref{1-t-4-2}. Set
\begin{equation*}
  (\hat U_1, \hat U_2,\hat U_3, \hat \Phi, \hat K,\hat S)=(U_1^{1},U_2^{1},U_3^{1}, \Phi^{1}, K^{1}, S^{1})-(U_1^{2},U_2^{2},U_3^{2}, \Phi^{2}, K^{2}, S^{2}).
\end{equation*}
  Similar to proof of   the
contraction of the two mappings $\mb_1^{{\bf W}}$ and $\mb_2$, one can derive
\begin{equation*}
\|(\hat U_1,\hat U_3)\|_{L^2(\mn)}+\|\hat \Phi\|_{H^1(\mn)}\leq \mc_8^\star\sigma_v \bigg(\|(\hat U_1,\hat U_3)\|_{L^2(\mn)} +\|(\hat U_2,\hat \Phi,\hat K,\hat S)\|_{H^1(\mn)}\bigg),
\end{equation*}
and
\begin{equation*}
\|(\hat U_2,\hat K,\hat S)\|_{H^1(\mn)}\leq \mc_8^\star\sigma_v \bigg(\|(\hat U_1,\hat U_3)\|_{L^2(\mn)} +\|(\hat U_2,\hat \Phi,\hat K,\hat S)\|_{H^1(\mn)}\bigg),
\end{equation*}
where the constant $\mc_8^\star>0$ depends only on $(b_0, \A_0,U_{0}, S_{0},E_{0},r_0,r_1,\epsilon_0)$. Collecting these estimates leads to
\begin{equation}\label{6-35}
\|(\hat U_1,\hat U_3)\|_{L^2(\mn)}+\|(\hat U_2,\hat \Phi,\hat K,\hat S)\|_{H^1(\mn)}\leq  2\mc_8^\star\sigma_v \bigg(\|(\hat U_1,\hat U_3)\|_{L^2(\mn)} +\|(\hat U_2,\hat \Phi,\hat K,\hat S)\|_{H^1(\mn)}\bigg),
\end{equation}
So if
\begin{equation}\label{6-37}
\sigma_v\leq \frac{1}{4\mc_8^\star}=:\sigma_2^\star,
\end{equation}
one obtains
\begin{equation*}
(\hat U_1, \hat U_2,\hat U_3, \hat \Phi, \hat K,\hat S)=(U_1^{1},U_2^{1},U_3^{1}, \Phi^{1}, K^{1}, S^{1})-(U_1^{2},U_2^{2}, U_3^{2},\Phi^{2}, K^{2}, S^{2})=0.
\end{equation*}
The proof of Theorem \ref{th2} is completed by choosing $ \sigma^\star $ as
\begin{equation*}
\sigma^\star =\min\{\sigma_1^\star ,\sigma_2^\star\},
\end{equation*}
where $  \sigma_1^\star $ is given in  \eqref{sigma-bd-step2}.
That is, the background supersonic   flow is structurally stable within rotational flows under axisymmetric  perturbations of the boundary
conditions in \eqref{1-c-c}.
\par {\bf Acknowledgement.} Wang is  partially supported by National Natural Science Foundation of China  11925105, 12471221.
\par {\bf Data availability.} No data was used for the research described in the article.
    \par {\bf Conflict of interest.} This work does not have any conflicts of interest.

 \end{document}